\documentclass{article}
\usepackage[frenchb, english]{babel}
\usepackage{tikz}
\usepackage{tkz-graph}
\usepackage{amssymb}
\usepackage{graphicx}
\usetikzlibrary{arrows, decorations.markings}
\usepackage{rotating}
\usepackage{caption}
\usepackage{verbatim}
\usepackage{amsmath}
\usepackage{amsthm}
\usepackage{enumerate}
\usepackage{array}
\newcolumntype{L}[1]{>{\raggedright\let\newline\\\arraybackslash\hspace{0pt}}m{#1}}

\usetikzlibrary{decorations.markings, arrows.meta}
\usepackage{mathtools}
\usepackage{color}   
\usepackage{hyperref}
\hypersetup{
    colorlinks=false, 
    linktoc=all,     
}

\usepackage{float}
\usepackage{mathrsfs}
\usepackage{wrapfig}
\usepackage{picture}
\usepackage{multirow}
\usepackage{MnSymbol}
\usepackage{relsize}
\usepackage{diagbox}
\usetikzlibrary{arrows}
\usetikzlibrary{shapes,snakes}
\usepackage[applemac]{inputenc}
\usepackage[T1]{fontenc}
\usepackage{geometry}
\usepackage{stackengine}
\usepackage{scalerel}
\usepackage{bbm}
\usepackage{float}
\usepackage{chngpage}
\usepackage{longtable}
\usepackage{hhline}
\usepackage{titlesec}

\newlength{\bibitemsep}
\newlength{\bibparskip}
\let\oldthebibliography\thebibliography
\renewcommand\thebibliography[1]{%
  \oldthebibliography{#1}%
  \setlength{\parskip}{\bibitemsep}%
  \setlength{\itemsep}{\bibparskip}%
}

\newcommand{\cc}{\ensuremath{\mathbb{C}}\xspace}

\usepackage{tocloft}
\usetikzlibrary{arrows.meta}
\usetikzlibrary{decorations.pathreplacing,angles,quotes}
\addto\captionsenglish{
  \renewcommand{\contentsname}%
    {Table of Contents}%
}
\tikzset{bullet/.style={draw,ellipse, text width = 4cm, text centered}}
\tikzset{rec/.style={draw, text width = 3.5 cm, text centered}}
\tikzset{plain/.style={->,>=stealth}}
\tikzset{
  barbarrow/.style={ 
     >={Straight Barb[left,length=5pt, width=5pt]}
  },
  strike through/.style={
    postaction=decorate,
    decoration={
      markings,
      mark=at position 0.5 with {
        \draw[-] (-3pt,-3pt)  --  (3pt, 3pt);
      }
    }
  }
}

\renewcommand{\thesection}{\Roman{section}}

\usepackage{secdot}
\usepackage{etoolbox}
\patchcmd{\numberline}{\hfil}{. \quad \hfil}{}{}
\cftsetindents{section}{1em}{2.3em}
\cftsetindents{subsection}{2.3 em}{2.5em}
\cftsetindents{subsubsection}{5em}{3.2em}

\makeatletter
\def\@seccntformat#1{\@ifundefined{#1@cntformat}%
   {\csname the#1\endcsname\space}
   {\csname #1@cntformat\endcsname}}
\newcommand\section@cntformat{\thesection. \space}       
\newcommand\subsection@cntformat{\thesubsection. \space} 
\newcommand\subsubsection@cntformat{\thesubsubsection. } 
\makeatother

\titlespacing{\section}
  {0pt}{0mm}{0mm}
\titlespacing{\subsection}
  {0pt}{0mm}{0mm}
\titlespacing{\subsubsection}
  {0pt}{0mm}{0mm} 

\makeatletter
\def\thm@space@setup{%
  \thm@preskip=0.6\baselineskip
  \thm@postskip=0.1mm 
}
\makeatother

\makeatletter
\def\@addpunct#1{%
  \relax\ifhmode
    \ifnum\spacefactor>\@m \else#1\fi
  \fi}
\newcommand{\mathclassname}{Mathematical Classification}
\def\@setmathclass{%
  {2010 \itshape \mathclassname:}\enspace \@mathclass\@addpunct.}
\def\mathclass#1{\def\@mathclass{#1}}
\let\@mathclass=\@empty
\g@addto@macro{\maketitle}{\begingroup%
  \let\@makefnmark\relax  \let\@thefnmark\relax%
  \ifx\@mathclass\@mpty\else\@footnotetext{\@setmathclass}\fi%
  \endgroup}
\makeatother

\makeatletter
\def\@addpunct#1{%
  \relax\ifhmode
    \ifnum\spacefactor>\@m \else#1\fi
  \fi}
\newcommand{\keywordsname}{Key words and phrases}
\def\@setkeywords{%
  {\itshape \keywordsname:}\enspace \@keywords\@addpunct.}
\def\keywords#1{\def\@keywords{#1}}
\let\@keywords=\@empty
\g@addto@macro{\maketitle}{\begingroup%
  \let\@makefnmark\relax  \let\@thefnmark\relax%
  \ifx\@keywords\@mpty\else\@footnotetext{\@setkeywords}\fi%
  \endgroup}
\makeatother

\newcommand\restr[2]{{
  \left.\kern-\nulldelimiterspace 
  #1 
  \vphantom{\big|} 
  \right|_{#2} 
  }}

 \title{\textsc{Deformations of Inhomogeneous Simple Singularities and Quiver Representations}}
\author{\textsc{Antoine Caradot} \footnote{Université Claude Bernard Lyon 1, France. Email: \texttt{caradot@math.univ-lyon1.fr}}}
\date{}
\mathclass{Primary 14B07; Secondary 16G20; Tertiary 17B22.}

\begin{document}

\selectlanguage{english}

\theoremstyle{plain}
\newtheorem{theorem}{Theorem}[section]
\newtheorem{lemma}{Lemma}[section]
\newtheorem{proposition}{Proposition}[section]
\newtheorem{definition}{Definition}[section]
\newtheorem{corollary}{Corollary}[section]

\theoremstyle{remark}
\newtheorem{remark}{Remark}[section]
\theoremstyle{remark}
\newtheorem{remarks}{Remarks}[section]
\theoremstyle{remark}
\newtheorem{example}{Example}[section]
\theoremstyle{remark}
\newtheorem{examples}{Examples}[section]

\newtheoremstyle{dotless}{}{}{\itshape}{}{\bfseries}{}{ }{}
\theoremstyle{dotless}
\newtheorem*{theorem*}{Theorem}
\newtheorem*{lemma*}{Lemma}
\newtheorem*{proposition*}{Proposition}

\maketitle
\vspace{-\topsep}
\begin{abstract}
This article is a summary of the author's unpublished Ph.D thesis (\cite{Cara17}). Its purpose is to generalise a construction by H. Cassens and P. Slodowy of the semiuniversal deformations of the simple singularities of type $A_r$, $D_r$, $E_6$, $E_7$ and $E_8$ to the inhomogeneous simple singularities of type $B_r$, $C_r$, $F_4$ and $G_2$. To a simple homogeneous singularity, one can associate the representation space of a particular quiver. This space is endowed with an action of the symmetry group of the Dynkin diagram associated to the simple singularity. From this we will construct and compute explicitly the semiuniversal deformations of the inhomogeneous simple singularities. By quotienting such maps, we obtain deformations of other simple singularities. In some cases, the discriminants of these last deformations will be computed.
\end{abstract}

\tableofcontents
\addtocontents{toc}{\vspace{-0.7cm}}

\section*{Introduction}
\addcontentsline{toc}{section}{Introduction}

In \cite{Klein84}, F. Klein studied polynomial equations of degree 5. For this he looked at the rotation groups of the five Platonic solids  and at the finite subgroups of $\mathrm{SL}_2(\cc)$, which are (up to conjugacy) the cyclic group $\mathcal{C}_n$, the binary dihedral group $\mathcal{D}_n$ and the binary polyhedral groups $\mathcal{T} $, $\mathcal{O} $ and $\mathcal{I} $. They are exactly the finite subgroups of $\mathrm{SU}_2$. He then proved that for such a group $\Gamma$, the quotient $\cc^2/\Gamma$ is a surface in $\cc^3$ defined by a polynomial equation (cf. Table~\ref{equationssimplesingularities}). The surface has an isolated singularity and is called a simple singularity. Later P. Du Val exhibited a link between the simple singularities and the simple Lie algebras of type $\Delta(\Gamma)=A_r$, $D_r$, $E_6$, $E_7$ or $E_8$ (cf. Table~\ref{diagramssimplesingularities}) using minimal resolutions. 

In \cite{McK80} J. McKay discovered a connection between the finite subgroups of $\mathrm{SU}_2$ and the simply laced Lie algebras which does not involve resolutions. He exhibited a way of constructing the Cartan matrix of the extended Dynkin diagram $\widetilde{\Delta}(\Gamma)$ from the irreducible representations of $\Gamma$. From this correspondence, one can construct the representation space $M(\Gamma)$ of a quiver obtained from $\Delta(\Gamma)$ and called a McKay quiver. It happens that $M(\Gamma)$ can be endowed with a symplectic structure. P.B. Kronheimer took advantage of such a structure and constructed in \cite{Kron89} the semiuniversal deformation of $\cc^2/\Gamma$ using hyperkähler reduction. Later on, H. Cassens and P. Slodowy worked on P.B. Kronheimer's results to obtain the semiuniversal deformation of $\cc^2/\Gamma$ and its minimal resolution in an algebraic-geometric context (cf. \cite{CaSlo98}).

Dynkin diagrams can be separated in two classes: the simply laced (or homogeneous) ones, namely $A_r$ ($r \geq 1$), $D_r$ ($r \geq 4$), $E_6$, $E_7$ and $E_8$, and the non-simply laced (or inhomogeneous) ones $B_r$ ($r \geq 2$), $C_r$ ($r \geq 3$), $F_4$ and $G_2$. In 1978, P. Slodowy extended the definition of simple singularities to the inhomogeneous types in the following way: a simple singularity of type $B_r$, $C_r$, $F_4$ or $G_2$ is a pair $(X_0,\Omega)$ where $X_0$ is a simple singularity of type $A_r$, $D_r$, $E_6$, $E_7$ or $E_8$, and $\Omega$ is a group of automorphisms of the Dynkin diagram associated to $X_0$ (cf. Table~\ref{definhomogeneous}). Let $\Gamma$ be a cyclic group of even order, a binary dihedral group or a binary tetrahedral group, and set $X_0=\cc^2/\Gamma$. It is then possible to find a finite subgroup $\Gamma'$ of $\mathrm{SU}_2$ such that $\Gamma$ is normal in $\Gamma'$ and $\Omega=\Gamma'/\Gamma$ acts on $X_0$. This action can be lifted to the minimal resolution of the singularity, and this lifting induces an action on the exceptional divisors that corresponds to a group of automorphisms of the Dynkin diagram associated to $X_0$. Using this definition, P. Slodowy generalised the McKay correspondence to the inhomogeneous types.  

The aim of this article is to generalise the construction by H. Cassens and P. Slodowy to the inhomogeneous cases, and obtain the semiuniversal deformations of the simple singularities of type $B_r$, $C_r$, $F_4$ and $G_2$. More precisely, we prove and then apply the following result:
\begin{theorem*}\emph{\textbf{\ref{thm:defrestriction}.}}
The restriction over the $\Omega$-fixed points of the semiuniversal deformation of the simple singularity $\cc^2/\Gamma$ obtained from the McKay quiver is a semiuniversal deformation of the inhomogeneous simple singularity of type $\Delta(\Gamma,\Gamma')$.
\end{theorem*}

In order to do so, we study the representation space $M(\Gamma)$ of a McKay quiver defined by using the Dynkin diagram associated to the simple singularity $\cc^2/\Gamma$. This quiver comes with a symmetry group $\Omega$ based on the inhomogeneous McKay correspondence. The action of the group $\Omega$ on $M(\Gamma)$ is obtained by lifting its action on the underlying graph. The choice of the orientation of the quiver, which did not play any particular role in the homogeneous case, has now to be carefully made. Indeed, the symplectic structure with which is provided the representation space of the McKay quiver depends on the orientation of the quiver. This symplectic structure induces a moment map from which the semiuniversal deformation is obtained. The following lemma implies that, if the action of $\Omega$ is symplectic, then the semiuniversal deformation becomes $\Omega$-equivariant.

\begin{lemma*}\emph{\textbf{\ref{momentequivariant}.}}
Let $\mu:M \rightarrow \mathfrak{g}^*$ be a moment map on a symplectic manifold $(M,\omega)$ with an action of a semisimple Lie group $G$. Assume that a group $\Omega$ acts on $M$ by symplectomorphisms, and that $\Omega$ is a subgroup of the outer automorphism group of $G$. Furthermore, assume that the action of $G$ lifts to an action of $G \rtimes \Omega$. 
Then $\mu$ is $\Omega$-equivariant.
\end{lemma*}

We will need to determine the conditions in which the action is symplectic, as well as their compatibility with the action of $\Omega$ on the special fibre of the deformation. These requirements on the action of $\Omega$ are provided in these theorems:

\begin{theorem*}\emph{\textbf{\ref{thm:orientation}.}}
The action of $\Omega=\Gamma'/\Gamma$ on $M(\Gamma)$ is symplectic and induces the natural action on the simple singularity when:
\vspace{-1.5ex}
\begin{itemize}\setlength\itemsep{0.3pt}
\item for $(A_{2r-1},\mathbb{Z}/2\mathbb{Z})$, the group $\Omega$ reverses the orientation of the McKay quiver.
\item for the other cases, the group $\Omega$ preserves the orientation of the McKay quiver.
\end{itemize}

\end{theorem*}

\begin{theorem*}\emph{\textbf{\ref{thm:compatibility}.}}
For any McKay quiver built on a Dynkin diagram of type $A_{2r-1}$, $D_{r+1}$ or $E_6$, there exists an action of $\Omega=\Gamma'/\Gamma$ on $M(\Gamma)$ that is symplectic and induces the natural action on the simple singularity $\cc^2/\Gamma$.
\end{theorem*}

We endow the representation space $M(\Gamma)$ with the action of $\Omega$ satisfying both Theorems~\ref{thm:orientation} and~\ref{thm:compatibility}. This makes the semiuniversal deformation of $\cc^2/\Gamma$ an $\Omega$-equivariant morphism, and we can apply Theorem~\ref{thm:defrestriction}. We obtain the semiuniversal deformation of the simple singularity of type $\Delta(\Gamma,\Gamma')$. This semiuniversal deformation is such that the group $\Omega$ induces an action on each of its fibres. By taking quotients of the fibres, we obtain a new deformation, but this time of the simple homogeneous singularity $\cc^2/\Gamma'$. Furthermore, the semiuniversal nature of the deformation is lost in the process. Results regarding the regularity of the fibres of the new deformation will also be proved with these next propositions.
\begin{proposition*}\emph{\textbf{\ref{A3singulier}.}}
Every fibre of the quotient of the semiuniversal deformation of a simple singularity of type $B_2$ by the action of $\Omega$ is singular.
\end{proposition*}

\begin{proposition*}\emph{\textbf{\ref{D4Zsur2Zsingulier}.}}
Every fibre of the quotient of the semiuniversal deformation of a simple singularity of type $C_3$ by the action of $\Omega$ is singular.
\end{proposition*}

\begin{proposition*}\emph{\textbf{\ref{D4S3singulier}.}}
Every fibre of the quotient of the semiuniversal deformation of a simple singularity of type $G_2$ by the action of $\Omega$ is singular.
\end{proposition*}

The first part (\ref{sec:FoldingsandSingularities}) of this article is devoted to the definitions of the folding of Dynkin diagrams, as well as of the simple homogeneous and inhomogeneous singularities. In the second part (\ref{sec:Tanisaki}) we recall results of T. Tanisaki regarding foldings of root systems and see how foldings, simple singularities and quiver representations are linked to each other. The third part (\ref{sec:SlodowyCassens}) presents the construction due to H. Cassens and P. Slodowy. In the fourth part (\ref{sec:InhomogeneousDeformations}), we study how to generalise the previously mentioned construction to the inhomogeneous types, and then realise computations in the fifth part (\ref{sec:computations}). The sixth part (\ref{sec:QuotientInhomogeneousDeformations}) is devoted to the description of the quotients of the semiuniversal deformations of the inhomogeneous simple singularities. 

All throughout this article, the base field will be the complex number field $\cc$.

\section{Simple singularities and Dynkin diagrams}\label{sec:FoldingsandSingularities}
\subsection{Foldings: definition and computations}\label{sec:folding}

Let $\mathfrak{g}$ be a finite dimensional simple Lie algebra and $x \in \mathfrak{g}$. If $\mathrm{ad} \ x$ is nilpotent, then $\exp(\mathrm{ad} \ x)$ is a well-defined automorphism of $\mathfrak{g}$. Any automorphism of $\mathfrak{g}$ that can be written in such a form is called inner, and the group generated by inner automorphisms is denoted by $\mathrm{Inn}(\mathfrak{g})$. It is a connected and normal subgroup of $\mathrm{Aut}(\mathfrak{g})$, the group of automorphisms of $\mathfrak{g}$. An outer automorphism is an element of $\mathrm{Out}(\mathfrak{g}) = \mathrm{Aut}(\mathfrak{g})/\mathrm{Inn}(\mathfrak{g})$. This group turns out to be isomorphic to the automorphism group of the Dynkin diagram of $\mathfrak{g}$. One can verify that the irreducible Dynkin diagrams that have a non-trivial outer automorphism group are those of type $A_r$ ($r \geq 2$), $D_r$ ($r \geq 3$) and $E_6$. Let $\sigma$ be an automorphism of one of these Dynkin diagrams. Then $\sigma$ induces an automorphism $\dot{\sigma}$ of the corresponding Lie algebra. By definition of $\sigma$, the morphism $\dot{\sigma}$ has finite order, say $r \in \mathbb{Z}_{\geq0}$. It follows that $\dot{\sigma}$ induces a gradation on $\mathfrak{g}$ such that $\mathfrak{g}=\bigoplus_{i \in \mathbb{Z}/r\mathbb{Z}}\mathfrak{g}_i$ with $\mathfrak{g}_i=\{x \in \mathfrak{g} \ | \ \dot{\sigma}(x)=\omega^ix\}$ and $\omega$ a primitive $r^{\text{th}}$ root of unity. The \textit{folding} of $\mathfrak{g}$ consists in computing the invariants of the automorphism $\dot{\sigma}$. One can also realise the folding of the root lattice of $\mathfrak{g}$, i.e. finding the type of the invariant sub-lattice $Q(\mathfrak{g})^\sigma$ of the root lattice $Q(\mathfrak{g})$ of $\mathfrak{g}$.

In the table below are results of different foldings carried out on the simple Lie algebras whose outer automorphism group is not trivial.  
\begin{center}
\renewcommand{\arraystretch}{1.6}
$\begin{array}[t]{|c|c|c|c|}
\hline
\text{Type of }\mathfrak{g} & \text{Type of }\mathfrak{g}_0 & \text{Type of }Q(\mathfrak{g})^{\sigma} & \text{Order of }\sigma  \\
\hhline{|=|=|=|=|}
A_{2r-1} &  C_r  & B_r & 2 \\
\hline
A_{2r} & B_r &  C_r & 2 \\
\hline
D_{r+1} & B_r & C_r &  2 \\
\hline
E_6 & F_4 & F_4 & 2  \\
\hline
D_4 & G_2 & G_2 &  3 \\
\hline
\end{array}$
 \captionof{table}{Foldings of simply laced root lattices and Lie algebras}
 \label{tablefoldings}
\end{center}

One notices that, in all five cases, the types of $\mathfrak{g}_0$ and $Q(\mathfrak{g})^{\sigma}$ are dual to each other. This is due to the fact that the short roots and the long roots are switched when one goes from the Lie algebra to the root lattice, and vice-versa.

\subsection{Simple singularities of type \texorpdfstring{$A_r$, $D_r$ and $E_r$}{Lg}}
\subsubsection{Definition of simple singularities}

F. Klein studied the finite subgroups of $\mathrm{SL}_2(\cc)$ in terms of transformations of $\cc^2$ and classified them in \cite{Klein84}. They are exactly the finite subgroup of $\mathrm{SU}_2$. Up to conjugacy, they are isomorphic to the cyclic group $\mathcal{C}_n$, the binary dihedral group $\mathcal{D}_n$ or one of the binary polyhedral groups $\mathcal{T}$, $\mathcal{O}$ and $\mathcal{I}$. If $\Gamma$ is a finite subgroup of $\mathrm{SU}_2$, its natural action on $\cc^2$ induces an action of $\Gamma$ on the ring $\cc[\cc^2]$ of complex polynomial functions on $\cc^2$. F. Klein computed the ring $\cc[\cc^2]^\Gamma$ of $\Gamma$-invariant polynomials and noticed that it is generated by three elements satisfying a unique relation.

\begin{theorem}\emph{\textbf{(Klein \cite{Klein84}).}}\label{Klein} Let $\Gamma$ be a finite subgroup of $\mathrm{SU}_2$. Then $\cc^2/\Gamma$ injects into $\cc^3$ as the zeros of a polynomial $R \in \cc[X,Y,Z]$, which is given in the following table:
\begin{center}
\renewcommand{\arraystretch}{1.3}
\begin{tabular}{|c|c|}
\hline
$\Gamma$ & $R$ \\
\hhline{|=|=|}
$\mathcal{C}_n$  & $X^n + YZ$ \\  
 \hline
$\mathcal{D}_n$  & $X(Y^2-X^n) + Z^2$ \\
\hline
$\mathcal{T}$  & $X^4 + Y^3 + Z^2$ \\
\hline
$\mathcal{O}$  & $X^3 + XY^3 + Z^2$ \\
\hline
$\mathcal{I}$  & $X^5 + Y^3 + Z^2$ \\
\hline
\end{tabular}
\captionof{table}{Equation of $\cc^2/\Gamma$}
\label{equationssimplesingularities}
\end{center}
\end{theorem}

According to the above theorem, the quotient $\cc^2/\Gamma$ can be seen as a hypersurface in $\cc^3$. Furthermore, it has a unique singularity at the origin. The surface $\cc^2/\Gamma$ is called a \textit{simple singularity}.

\begin{remark}
In \cite{Durf79}, A.H. Durfee gave many different names for this type of singularities including rational double points, quotient singularities or absolutely isolated double points among others. In this article, the chosen denomination will be simple singularities.
\end{remark}

Singularities are usually defined as germs of analytic spaces. However, using Artin's algebraisation theorem (cf. \cite{Ishii14} Theorem 4.2.4), simple singularities can be seen as algebraic varieties and one can work with the algebraic settings given by the polynomial equations of the previous theorem. Furthermore, the following proposition (cf. \cite{Lamo86}) tells us that one can work globally using the polynomial equations of Theorem~\ref{Klein}:

\begin{proposition}
Let ${}_{3}\mathcal{O}$ denote the germs of holomorphic functions $(\cc^3,0) \rightarrow \cc$ and let $f \in {}_{3}\mathcal{O}$ be a germ without multiple factors. If $Z(f)=\{x \in \cc^3 \ | \ f(x)=0\}$ is isomorphic to the simple singularity $\cc^2/\Gamma$ defined from a finite subgroup $\Gamma$ of $\mathrm{SU}_2$, then there is a biholomorphic germ $\varphi:(\cc^3,0) \cong (\cc^3,0)$ such that $R=f \circ \varphi$ is the polynomial associated to $\Gamma$ by Theorem~\ref{Klein}.
\end{proposition}

Set $p \in \mathbb{Z}_{>0}$ and choose $T_1, \ldots ,T_p$ among $A_r$ ($r \geq 1)$, $D_r$ $(r \geq 4)$, $E_6$, $E_7$ and $E_8$. A \textit{singular configuration} of type $T_1 +  \ldots   + T_p$ is a complex algebraic surface with $p$ isolated singularities $s_1, \ldots ,s_p$ such that, for each $1 \leq i \leq p$, locally around $s_i$ the surface is a simple singularity of type $T_i$.

\begin{example} Set
\vspace{-\topsep}
\begin{center}
\begin{tikzpicture}[scale=1,  transform shape]
\node (1) at ( 0,0) {$\cc^4$};
\node (2) at ( 5,0) {$\cc^2$};
\node (3) at ( 5,-0.5)  {$(z^2-x^3 + 3xy^2 + t(x^2 + y^2), t)$.};
\node (4) at ( 0,-0.5)  {$(x, y, z, t)$};

\node (9) at (-1,-0.03) {$f:$};

\node (5) at (0.8,0) {};
\node (6) at (2.7,0) {};
\node (7) at (0.8,-0.5) {};
\node (8) at (2.7,-0.5) {};

\draw [decoration={markings,mark=at position 1 with
    {\arrow[scale=1.2,>=stealth]{>}}},postaction={decorate}] (5)  --  (6);
\draw [|-,decoration={markings,mark=at position 1 with
    {\arrow[scale=1.2,>=stealth]{>}}},postaction={decorate}] (7)  --  (8);
\end{tikzpicture}
 \end{center}
The fibre $f^{-1}(0,0)$ is a simple singularity of type $D_4$. Besides the special fibre, the singular fibres of $f$ are of two types: $f^{-1}(0,t)$ and $f^{-1}(\frac{4}{27}t^3,t)$ for $t \neq 0$. The former has a single singularity of type $A_1$ at the origin, whereas the latter has three singular points, namely $(\frac{2}{3}t,0,0)$, $(-\frac{t}{3},\frac{t}{\sqrt{3}},0)$ and $(-\frac{t}{3},-\frac{t}{\sqrt{3}},0)$, each of type $A_1$. Therefore $f^{-1}(0,t)$ is a simple singularity of type $A_1$ and $f^{-1}(\frac{4}{27}t^3,t)$ is a singular configuration of type $A_1 + A_1 + A_1$. Below is an illustration in the real plane $\{z=0\}$ of both fibres for $t=1$:

\begin{center}
\includegraphics[scale=0.3]{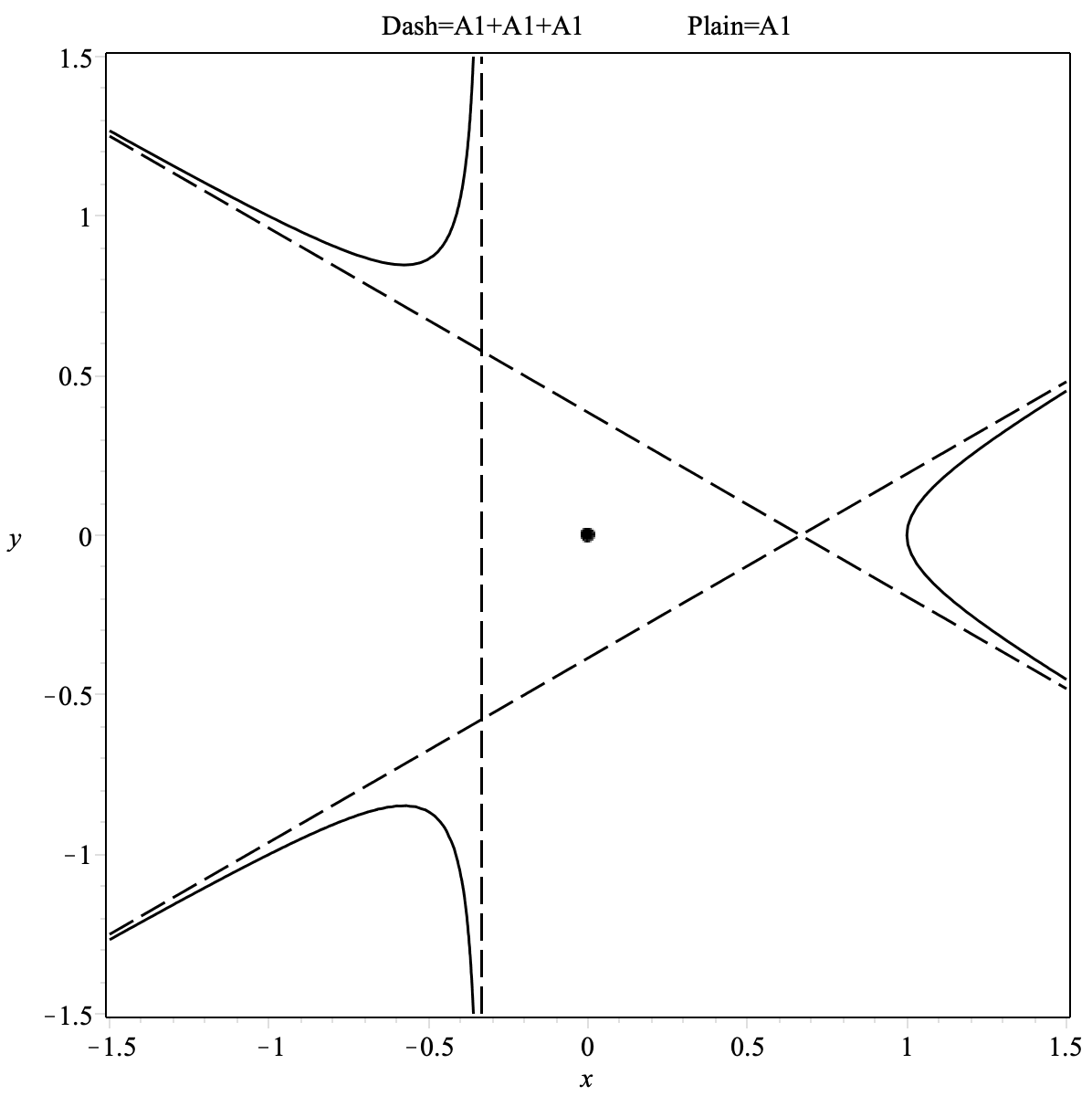}
\end{center}
\end{example}

\subsubsection{Resolutions of simple singularities}\label{subsub:definitionsimplesing}

Let $X=\cc^2/\Gamma$ be a simple singularity. P. Du Val proved in \cite{DuVa34} that if $x$ is the singular point and $\pi:\widetilde{X} \rightarrow X$ is the minimal resolution of $X$, then the preimage of $x$ is a union of projective lines whose intersection matrix is the opposite of a Cartan matrix of type $\Delta(\Gamma)$, according to the following list:
\newpage

\begin{center}
\centering
\begin{tabular}{|c|c|c|}
\hline
 $\Gamma$ & $\Delta(\Gamma)$ & Type of $\Delta(\Gamma)$\\
\hhline{|=|=|=|}
  \multirow{-1.5}{*}{$\mathcal{C}_n$} & 
 \begin{tikzpicture}[scale=0.6,  transform shape]
\tikzstyle{point}=[circle,draw,fill]
\tikzstyle{ligne}=[thick]
\tikzstyle{pointille}=[thick,dotted]

\node (1) at ( -4,0) [point] {};
\node (2) at ( -2,0) [point]{};
\node (3) at ( 0,0) [point] {};
\node (4) at ( 3,0) [point]{} ;
\node (5) at ( 5,0) [point] {};

\node (6) at ( -4,0.5) {1};
\node (7) at ( -2,0.5)  {2};
\node (8) at ( 0,0.5)  {3};
\node (9) at ( 3,0.5) {n-2};
\node (10) at ( 5,0.5) {n-1};

\draw [ligne] (1)  --  (2);
\draw [ligne] (2)  --  (3);
\draw [pointille] (3)  --  (4);
\draw [ligne] (4)  --  (5);

\end{tikzpicture}  & \multirow{-1.5}{*}{$A_{n-1}$} \\
\hline
  \multirow{-4.5}{*}{$\mathcal{D}_n$} &  \begin{tikzpicture}[scale=0.6,  transform shape]
\tikzstyle{point}=[circle,draw,fill]
\tikzstyle{ligne}=[thick]
\tikzstyle{pointille}=[thick,dotted]

\node (1) at ( -6,0) [point] {};
\node (2) at ( -4,0) [point] {};
\node (3) at ( -2,0) [point] {};
\node (4) at ( 0,0) [point] {};
\node (5) at ( 2,1) [point] {};
\node (6) at ( 2,-1) [point] {};

\node (7) at ( -6,0.5)  {1};
\node (8) at ( -4,0.5)  {2};
\node (9) at ( -2,0.5)  {3};
\node (10) at ( 0,0.5) {n};
\node (11) at ( 2,1.5) {n+1};
\node (12) at ( 2,-0.5) {n+2};

\draw [ligne] (1)  --  (2);
\draw [ligne] (2)  --  (3);
\draw [pointille] (3)  --  (4);
\draw [ligne] (4)  --  (5);
\draw [ligne] (4)  --  (6);

\end{tikzpicture} &  \multirow{-4.5}{*}{$D_{n+2}$} \\
\hline
  \multirow{-4}{*}{$\mathcal{T}$} &  \begin{tikzpicture}[scale=0.6,  transform shape]
\tikzstyle{point}=[circle,draw,fill]
\tikzstyle{ligne}=[thick]
\tikzstyle{pointille}=[thick,dotted]

\node (1) at ( -4,0) [point] {};
\node (2) at ( 0,-2) [point] {};
\node (3) at ( -2,0) [point] {};
\node (4) at ( 0,0) [point] {};
\node (5) at ( 2,0) [point] {};
\node (6) at ( 4,0) [point] {};

\node (7) at ( -4,0.5)  {1};
\node (8) at ( -0.5,-2)  {2};
\node (9) at ( -2,0.5)  {3};
\node (10) at ( 0,0.5) {4};
\node (11) at ( 2,0.5) {5};
\node (12) at ( 4,0.5) {6};

\draw [ligne] (1)  --  (3);
\draw [ligne] (3)  --  (4);
\draw [ligne] (4)  --  (2);
\draw [ligne] (4)  --  (5);
\draw [ligne] (5)  --  (6);

\end{tikzpicture} & \multirow{-4}{*}{$E_6$} \\
\hline
  \multirow{-4}{*}{$\mathcal{O}$} &  \begin{tikzpicture}[scale=0.6,  transform shape]
\tikzstyle{point}=[circle,draw,fill]
\tikzstyle{ligne}=[thick]
\tikzstyle{pointille}=[thick,dotted]

\node (1) at ( -4,0) [point] {};
\node (2) at ( 0,-2) [point] {};
\node (3) at ( -2,0) [point] {};
\node (4) at ( 0,0) [point] {};
\node (5) at ( 2,0) [point] {};
\node (6) at ( 4,0) [point] {};
\node (7) at ( 6,0) [point] {};

\node (8) at ( -4,0.5)  {1};
\node (9) at ( -0.5,-2)  {2};
\node (10) at ( -2,0.5)  {3};
\node (11) at ( 0,0.5) {4};
\node (12) at ( 2,0.5) {5};
\node (13) at ( 4,0.5) {6};
\node (14) at ( 6,0.5) {7};

\draw [ligne] (1)  --  (3);
\draw [ligne] (3)  --  (4);
\draw [ligne] (4)  --  (2);
\draw [ligne] (4)  --  (5);
\draw [ligne] (5)  --  (6);
\draw [ligne] (6)  --  (7);

\end{tikzpicture} & \multirow{-4}{*}{$E_7$} \\
\hline
  \multirow{-4.5}{*}{$\mathcal{I}$} &  \begin{tikzpicture}[scale=0.6,  transform shape]
\tikzstyle{point}=[circle,draw,fill]
\tikzstyle{ligne}=[thick]
\tikzstyle{pointille}=[thick,dotted]

\node (1) at ( -4,0) [point] {};
\node (2) at ( 0,-2) [point] {};
\node (3) at ( -2,0) [point] {};
\node (4) at ( 0,0) [point] {};
\node (5) at ( 2,0) [point] {};
\node (6) at ( 4,0) [point] {};
\node (7) at ( 6,0) [point] {};
\node (8) at ( 8,0) [point] {};

\node (9) at ( -4,0.5)  {1};
\node (10) at ( -0.5,-2)  {2};
\node (11) at ( -2,0.5)  {3};
\node (12) at ( 0,0.5) {4};
\node (13) at ( 2,0.5) {5};
\node (14) at ( 4,0.5) {6};
\node (15) at ( 6,0.5) {7};
\node (16) at ( 8,0.5) {8};

\draw [ligne] (1)  --  (3);
\draw [ligne] (3)  --  (4);
\draw [ligne] (4)  --  (2);
\draw [ligne] (4)  --  (5);
\draw [ligne] (5)  --  (6);
\draw [ligne] (6)  --  (7);
\draw [ligne] (7)  --  (8);

\end{tikzpicture} & \multirow{-4.5}{*}{$E_8$} \\
 \hline
 \end{tabular}
 \captionof{table}{Diagrams associated to the simple singularities}
 \label{diagramssimplesingularities}
\end{center}

\begin{remark}
D. Kirby in \cite{Kirby57} has characterized the simple singularities as being the only double points whose minimal resolutions can be obtained by successive blowups. Simple singularities have many characterisations which have been summarised by A.H. Durfee in \cite{Durf79}.
\end{remark}

\subsubsection{Semiuniversal deformation}\label{sub:semiunivdef}

A \textit{deformation} of a complex integral algebraic variety $X_0$ with an isolated singularity $x$ is a flat morphism of germs of algebraic varieties $\varphi:(X,x) \rightarrow (U, u)$ such that there exits an isomorphism $i:(X_0,x) \xrightarrow[]{\cong} (\varphi^{-1}(u),x)$. A deformation $\varphi:(X,x) \rightarrow (U, u)$ of $(X_0,x)$ is called \textit{semiuniversal} if any other deformation $\varphi':(X',x) \rightarrow (T, t)$ of $(X_0,x)$ is isomorphic to a deformation induced from $\varphi$ by a base change $\psi:(T, t) \rightarrow (U, u)$ whose differential at $t \in T$ is uniquely determined. It follows that semiuniversal deformations are unique up to isomorphism. 

The next theorem gives a way of computing the semiuniversal deformation of any simple singularity.

\begin{theorem}\emph{\textbf{(Kas-Schlessinger \cite{KasSchle72}).}}\label{Kas-Schlessinger}
Let $X_0$ be a simple singularity defined in $\cc^3$ by the polynomial equation $f(X,Y,Z)=0$, with the origin being the singular point. Because the singularity is isolated, the vector space $V=\cc[X,Y,Z]/(f, \frac{\partial f}{\partial X}, \frac{\partial f}{\partial Y}, \frac{\partial f}{\partial Z})$ is finite dimensional. Let $(b_i(X,Y,Z))_{1 \leq i \leq k}$ be a base of $V$. Then the semiuniversal deformation of the simple singularity $X_0$ is given by the map
\begin{center}
\begin{tikzpicture}[scale=1,  transform shape]
\node (1) at (-2,-0.03) {$\varphi:$};
\node (2) at (0,0) {$\cc^3\times \cc^{k}$};
\node (3) at (6,0) {$\cc$};
\node (4) at (6,-0.6) {$\displaystyle f(X,Y,Z) + \sum_{i=1}^k a_kb_k(X,Y,Z).$};
\node (5) at (0,-0.6) {$(X,Y,Z,a_1,\ldots,a_k)$};

\draw  [decoration={markings,mark=at position 1 with
    {\arrow[scale=1.2,>=stealth]{>}}},postaction={decorate}] (1.6,0)  --  (3.5,0);
\draw  [|-,decoration={markings,mark=at position 1 with
    {\arrow[|-,scale=1.2,>=stealth]{>}}},postaction={decorate}] (1.6,-0.6)  --  (3.5,-0.6);
\end{tikzpicture}
\end{center}
\end{theorem}

Simple singularities have been extensively studied using Lie algebras. For example, E. Brieskorn proved the following theorem:

\begin{theorem}\emph{\textbf{(Brieskorn \cite{Bries71}).}}\label{Brieskorn} Let $\mathfrak{g}$ be a simple Lie algebra of type $A_r$, $D_r$ or $E_r$, and $G$ be the corresponding simple Lie group. Set $\mathfrak{h}$ a Cartan subalgebra of $\mathfrak{g}$ and $W$ its associated Weyl group. Set $e \in \mathfrak{g}$ a subregular nilpotent element. Let $S_e \ \mathlarger{\mathlarger{\subset}} \ \mathfrak{g}$ be a transversal slice (notion defined below) at $e$ to the $G$-orbit of $e$ in $\mathfrak{g}$. Then the restriction $\restr{\chi}{S_e}:(S, e) \rightarrow (\mathfrak{h}/W,0)$ of the adjoint quotient $\chi:\mathfrak{g} \rightarrow \mathfrak{h}/W$ is a semiuniversal deformation of the simple singularity $(\restr{\chi}{S_e}^{-1}(0),e)$, which is of the same type as $\mathfrak{g}$.
\end{theorem}

A \textit{transversal slice} $S_e$ at $e$ to the $G$-orbit of $e$ is a locally closed subvariety $S_e \ \mathlarger{\mathlarger{\subset}} \ \mathfrak{g}$ containing $e$ and such that the morphism 
\begin{center}
$\begin{array}[t]{ccc}
G \times S_e & \rightarrow & \mathfrak{g} \\
(g, s) & \mapsto & (\mathrm{Ad} \ g)s
\end{array}$
\end{center}
is smooth and $\mathrm{dim} \   S_e= \mathrm{codim \ }  G.e$. In our context, as $e$ is nilpotent, a transversal slice is easily obtained by setting $S_e=e + \mathfrak{z}_{\mathfrak{g}}(f)$, with $f \in \mathfrak{g}$ being part of the $\mathfrak{sl}_2$-triple $(e,f,h)$ obtained by the Jacobson-Morozov theorem. The slice $S_e$ is called the \textit{Slodowy slice} at $e$. 

We conclude this section with the following proposition, whose proof can be found in \cite{Slo80}: 

\begin{proposition}\emph{\textbf{(Slodowy \cite{Slo80}).}}\label{propsubDynkin} Let $\chi:X \rightarrow U$ be a deformation of a simple singularity $X_0$ with associated Dynkin diagram $\Delta$. Let $X_\epsilon \neq X_0$ be a non-special fibre of $\chi$. Then there is a proper subdiagram $\Delta' \ \mathlarger{\mathlarger{\subset}} \ \Delta$ and a type-preserving bijection of the components of $\Delta'$ onto the singular points of $X_\epsilon$. This means that each connected component of $\Delta'$ is sent to a singular point of $X_\epsilon$, which is a simple singularity of the corresponding type. Furthermore, if $\chi$ is semiuniversal, then all subdiagrams of $\Delta$ are realised as singular configurations in non-special fibres.
\end{proposition}

\subsection{Simple singularities of type \texorpdfstring{$B_r$, $C_r$, $F_4$ and $G_2$}{Lg}}\label{BCFGdefinition}

Simple singularities have been defined and linked to the simply laced Lie algebras. However there exist simple Lie algebras of another kind, namely the Lie algebras of inhomogeneous (or non-simply laced) type $B_r$ ($r \geq 2$), $C_r$  ($r \geq 3$), $F_4$ and $G_2$. P. Slodowy gave a definition of simple singularities of inhomogeneous type which requires some preparation. 

\subsubsection{Group action on a simple singularity}\label{sub:groupaction}

Let $\Gamma \ \mathlarger{\mathlarger{\subset}} \ \Gamma' $ be finite subgroups of $\mathrm{SU}_2$. We would like $\Gamma'$ to act on the simple singularity $\cc^2/\Gamma$. If $\Gamma$ is normal in $\Gamma'$, the action of $\Gamma'$ on $\cc^2$ induces an action of $\Gamma'/\Gamma$ on $\cc^2/\Gamma$. It is thus natural to require for $\Gamma$ to be normal in $\Gamma'$. \\
We are looking for $\Gamma'$ in $\mathrm{SU}_2$ such that $\Gamma$ is normal in $\Gamma'$ and $\Gamma'$ acts on $\cc^2/\Gamma$. The action of $\Gamma'$ can be lifted through the minimal resolution of $\cc^2/\Gamma$. As the minimal resolution restricts to an isomorphism on a dense open subset and $\Gamma$ acts trivially on $\cc^2/\Gamma$, it follows that $\Gamma$ acts trivially on the minimal resolution. So the action of $\Gamma'$ on the whole minimal resolution factors through an action of $\Gamma'/\Gamma$. As the origin is a fixed point of $\Gamma'$ in $\cc^2/\Gamma$, the quotient $\Gamma'/\Gamma$ acts on the exceptional locus, which is composed of a union of projective lines whose intersection pattern corresponds to the Dynkin diagram type $\Delta(\Gamma)$. Therefore we want the exceptional locus to be acted upon by $\Gamma'/\Gamma$ in the same way the Dynkin diagram of type $\Delta(\Gamma)$ is acted upon by its symmetry group $\Omega$. This can be achieved by requiring $\Gamma'/\Gamma$ to be isomorphic to $\Omega$ and $\Gamma$ to be normal in $\Gamma'$. 

In the following table are all normal subgroups $\Gamma$ of the finite subgroups $\Gamma'$ of $\mathrm{SU}_2$ (cf. \cite{Cox91}).
\newpage
\renewcommand{\arraystretch}{1.2}
\begin{center} \begin{tabular}{|c|c|c|}
\hline
     $\Gamma'$ & $\Gamma$ & Order of $\Gamma$ \\
 \hhline{|=|=|=|}
     $\mathcal{C}_n$ & $\mathcal{C}_r$, $\forall \ r / n$ & $r$  \\
     \hline
   \multirow{3}{*}{$\mathcal{D}_n$} &  $\mathcal{D}_{\frac{n}{2}} \text{ if }n \text{ even}$ & $2n$ \\ 
    & $\mathcal{C}_{\frac{n}{q}}, \forall \ q / n \text{ and }  \frac{n}{q} \text{ odd}$ & $\frac{n}{q}$  \\
    & $\mathcal{C}_{\frac{2n}{q}}, \forall \ q / n$ & $\frac{2n}{q}$ \\
    \hline
      \multirow{2}{*}{$\mathcal{T}$} & $\mathcal{D}_2$ & $8$  \\
      & $\{\pm1\}$ & $2$ \\
     \hline
   \multirow{3}{*}{$\mathcal{O}$}  &  $\mathcal{T}$ & $24$ \\
    & $\mathcal{D}_2$ & $8$ \\
    & $\{\pm1\}$ & $2$ \\
     \hline
     $\mathcal{I}$ & $\{\pm1\}$ & $2$  \\
     \hline
    \end{tabular}
    \captionof{table}{Normal subgroups of the finite subgroups of $\mathrm{SU}_2$}
    \end{center}

In the rest of this section, $(z_1,z_2)$ will denote the dual of the canonical base of $\cc^2$ (cf. Theorem~\ref{Klein}). We list here the different possibilities of groups $\Gamma'$ verifying our criteria. 

$\bullet$ \underline{Type $A_{2r-1}$}: $\Gamma=\mathcal{C}_{2r}$ and $\Omega=\mathbb{Z}/2\mathbb{Z}$. Based on the previous table, the possibilities are $\Gamma'=\mathcal{C}_{4r}$ or $\mathcal{D}_r$. 
The singularity $\cc^2/\Gamma$ is given by 
\begin{center}
$\cc^2/\Gamma=\{X^{2r}-YZ=0\}$ with $\left\{
    \begin{array}{l}
     X=z_1z_2, \\
     Y=z_1^{2r},\\
     Z=z_2^{2r}.
    \end{array}
\right.$ 
\end{center}
- $\mathcal{C}_{4r}$ is generated by $g=\begin{pmatrix} \zeta & 0 \\ 0 & \zeta^{-1} \end{pmatrix}$ with $\zeta=\exp(\frac{2i\pi}{4r})$. The action on $\cc^2/\Gamma$ is $g.X  =  X$, $g.Y  =  -Y$, $g.Z  =  -Z$. \\
- $\mathcal{D}_{r}$ is generated by $g^2$ and $h=\begin{pmatrix} 0 & i \\ i & 0 \end{pmatrix}$. Then $g^2$ fixes $\cc^2/\Gamma$ and $h.X  =  -X$, $h.Y  =  (-1)^rZ$, $h.Z  =  (-1)^rY$. \\
In order to determine which group is the appropriate one, it is necessary to extend the action of $\Gamma'$ to the minimal resolution of $\cc^2/\Gamma$, and check that $\Gamma'$ permutes the components of the exceptional locus the same way $\Omega$ acts on the Dynkin diagram. This can only be achieved if $\Gamma'=\mathcal{D}_r$. 

$\bullet$ \underline{Type $A_{2r}$}: $\Gamma=\mathcal{C}_{2r+1}$ and $\Omega=\mathbb{Z}/2\mathbb{Z}$. Based on the table, the only possibility is $\Gamma'=\mathcal{C}_{4r + 2}$. 
The singularity $\cc^2/\Gamma$ is defined by 
\begin{center}
$\cc^2/\Gamma=\{X^{2r+1}-YZ=0\}$ with $\left\{
    \begin{array}{l}
     X=z_1z_2, \\ 
     Y=z_1^{2r+1},\\
     Z=z_2^{2r+1}.
    \end{array}
\right.$ 
\end{center}
As $\mathcal{C}_{4r + 2}$ is generated by $g=\begin{pmatrix} \zeta & 0 \\ 0 & \zeta^{-1} \end{pmatrix}$ with $\zeta=\exp(\frac{2i\pi}{4r + 2})$, it follows that $g.X  =  X$, $g.Y  =  -Y$ and $g.Z  =  -Z$. However, by computing the minimal resolution of the simple singularity, one can check that the lift of the action of $\Gamma'$ to the minimal resolution of $\cc^2/\Gamma$ does not permute the components of the exceptional locus the same way $\Omega$ acts on the Dynkin diagram. Hence there is no $\Gamma'$ verifying our criteria. 

$\bullet$ \underline{Type $D_{r+1}$}:  $\Gamma=\mathcal{D}_{r-1}$ and $\Omega=\mathbb{Z}/2\mathbb{Z}$. The only possibility is $\Gamma'=\mathcal{D}_{2(r-1)}$. 
The singularity $\cc^2/\Gamma$ is defined by 
\begin{center}
$\cc^2/\Gamma=\{X(Y^2+(-1)^rX^{r-1}) + Z^2=0\}$ with $\left\{
    \begin{array}{l}
     X=4^{\frac{1}{r}}(z_1z_2)^2, \\ 
     Y=4^{-\frac{1}{2r}}(z_1^{2(r-1)} +(-1)^{r+1} z_2^{2(r-1)}),\\
     Z=iz_1z_2(z_1^{2(r-1)}+(-1)^rz_2^{2(r-1)}).
    \end{array}
\right.$ 
\end{center}
Because $\mathcal{D}_{2(r-1)}$ is generated by $g=\begin{pmatrix} \xi & 0 \\ 0 & \xi^{-1} \end{pmatrix}$ and $h=\begin{pmatrix} 0 & i \\ i & 0 \end{pmatrix}$ with $\xi=\exp(\frac{2i\pi}{4(r-1)})$, it follows that
\begin{center}
$\left\{
    \begin{array}{ccc}
     g.X & = & X, \\
     g.Y & = & -Y, \\
     g.Z & = &-Z,
    \end{array}
\right.$ \hspace{1cm} and \hspace{1cm} $\left\{
    \begin{array}{ccccc}
     h.X & = & X, \\
     h.Y & = & Y, \\ 
     h.Z & = & Z.
    \end{array}
\right.$
\end{center}
\bigskip

$\bullet$ \underline{Type $E_6$}: $\Gamma=\mathcal{T}$ and $\Omega=\mathbb{Z}/2\mathbb{Z}$. According to the previous table, the only possibility is $\Gamma'=\mathcal{O}$. 
The singularity $\cc^2/\Gamma$ is given by 
\begin{center}
$\cc^2/\Gamma=\{X^4 + Y^3 + Z^2=0\}$ with $\left\{
    \begin{array}{l}
     X=108^{\frac{1}{4}}z_1z_2(z_1^4-z_2^4), \\ 
     Y=\exp(\frac{i\pi}{3})(z_1^8 + z_2^8 + 14(z_1z_2)^4),\\
     Z=(z_1^4 + z_2^4)^3-36(z_1z_2)^4(z_1^4 + z_2^4).
    \end{array}
\right.$ 
\end{center}
The group $\mathcal{O}$ is generated by $\mathcal{T}$ and $g=\begin{pmatrix} \epsilon^3 & 0 \\ 0 & \epsilon^5 \end{pmatrix}$ with $\epsilon=\exp(\frac{i\pi}{4})$. It is known that $\mathcal{T}$ fixes $\cc^2/\Gamma$, and one computes $g.X  =  -X$, $ g.Y  =  Y$, $g.Z  =  -Z$. 
\bigskip

$\bullet$ \underline{Type $D_4$}: $\Gamma=\mathcal{D}_2$ and $\Omega=\mathfrak{S}_3$. Then $\Gamma'=\mathcal{O}$. 
The singularity $\cc^2/\Gamma$ is defined by 
\begin{center}
$\cc^2/\Gamma=\{X(Y^2-X^2) + Z^2=0\}$ with $\left\{
    \begin{array}{l}
     X=4^{\frac{1}{3}}(z_1z_2)^2, \\ 
     Y=4^{-\frac{1}{6}}(z_1^{4} + z_2^{4}),\\
     Z=iz_1z_2(z_1^{4}-z_2^{4}).
    \end{array}
\right.$ 
\end{center}
Furthermore, $\mathcal{O}=<\mathcal{D}_2, g, h>$ where $g=\frac{1}{\sqrt{2}}\begin{pmatrix} \epsilon &  \epsilon^3 \\  \epsilon &  \epsilon^7 \end{pmatrix}$, $h=\begin{pmatrix} \epsilon^3 & 0 \\ 0 & \epsilon^5 \end{pmatrix}$ and $\epsilon=\exp(\frac{i\pi}{4})$. Hence $\mathcal{D}_2$ fixes $\cc^2/\Gamma$,
\vspace{-\topsep}
\begin{center}
$\left\{ \renewcommand{\arraystretch}{1.2}
    \begin{array}{ccc}
     g.X & = & \frac{1}{2}(Y-X), \\
     g.Y & = & -\frac{1}{2}(Y + 3X), \\
     g.Z & = & Z,
    \end{array}
\right.$ \hspace{1cm} and \hspace{1cm} $\left\{
    \begin{array}{ccc}
     h.X & = & X, \\
     h.Y & = & -Y, \\
     h.Z & = & -Z.
    \end{array}
\right.$
\end{center}

The results obtained are summarised in the following table:
\begin{center}
{\renewcommand{\arraystretch}{1.3}
 \begin{tabular}{|c|c|>{\centering\arraybackslash}p{1cm}|>{\centering\arraybackslash}p{2cm}|>{\centering\arraybackslash}p{2cm}|}
\hline
    Homogeneous & Inhomogeneous  & $\Gamma$ & $\Gamma'$ & $\Omega$ \\
\hhline{|=|=|=|=|=|}
    $A_{2r-1}$ & $B_r$ & $\mathcal{C}_{2r}$ & $\mathcal{D}_r$ & $\mathbb{Z}/2\mathbb{Z}$ \\
\hline
    $A_{2r}$ & $C_r$ & $\mathcal{C}_{2r+1}$ & $\times$ & $\mathbb{Z}/2\mathbb{Z}$ \\
\hline
    $D_{r+1}$ & $C_r$ & $\mathcal{D}_{r-1}$ & $\mathcal{D}_{2(r-1)}$ & $\mathbb{Z}/2\mathbb{Z}$ \\
\hline
    $E_{6}$ & $F_4$ & $\mathcal{T}$ & $\mathcal{O}$ & $\mathbb{Z}/2\mathbb{Z}$ \\
\hline
    $D_{4}$ & $G_2$ & $\mathcal{D}_{2}$ &  \diagbox[width=2.4cm, innerleftsep=.5cm, innerrightsep=.5cm]{$\mathcal{O}$}{$\mathcal{T}$} &  \diagbox{\quad $\mathfrak{S}_3$}{$\mathbb{Z}/3\mathbb{Z}$}\\
\hline
    \end{tabular}}
     \captionof{table}{Group actions on simple singularities}
\end{center}
\medskip
The "Inhomogeneous" column corresponds to the foldings of the root lattices whose types compose the first column. They were computed in Section~\ref{sec:folding}. We add that if we chose $\Omega$ to be $\mathbb{Z}/3\mathbb{Z}$ in the case $\Gamma=\mathcal{D}_2$, it can be proved that $\Gamma'=\mathcal{T}$.

\subsubsection{Definitions of inhomogeneous simple singularities and their deformations}

The definition of the inhomogeneous simple singularities was first given by P. Slodowy in \cite{Slo80}. 

\begin{definition}\label{def:inhomogeneous}
A simple singularity of type $B_r \ (r \geq 2)$, $C_r \ (r \geq 3)$, $F_4$ or $G_2$ is a pair $(X_0,\Omega)$ of a simple singularity $X_0$ (in the former sense) and a group $\Omega$ of automorphisms of $X_0$ according to the following list:
\begin{center}
\renewcommand{\arraystretch}{1.3}
\begin{tabular}[t]{|c|c|c|c|c|}
\hline
Type of $(X_0,\Omega)$ & Type of $X_0$ & $\Gamma$ & $\Gamma'$ & $\Omega=\Gamma'/\Gamma$ \\
\hhline{|=|=|=|=|=|}
$B_r, \ r \geq 2$ & $A_{2r-1}$ & $\mathcal{C}_{2r}$ & $\mathcal{D}_r$ & $\mathbb{Z}/2\mathbb{Z}$\\
\hline
$C_r, \ r \geq 3$ & $D_{r+1}$ & $\mathcal{D}_{r-1}$ & $\mathcal{D}_{2(r-1)}$ & $\mathbb{Z}/2\mathbb{Z}$ \\
\hline
$F_4$ & $E_6$ & $\mathcal{T}$ & $\mathcal{O}$ & $\mathbb{Z}/2\mathbb{Z}$ \\
\hline
$G_2$  & $D_4$ & $\mathcal{D}_{2}$ & $\mathcal{O}$ & $\mathfrak{S}_3$\\
\hline
\end{tabular}
\captionof{table}{Definition of the simple inhomogeneous singularities}
\label{definhomogeneous}
\end{center}
\end{definition}

A simple singularity of inhomogeneous type is thus a simple homogeneous singularity endowed with the symmetry of the associated Dynkin diagram. One notices from the Section~\ref{sec:folding} that the type of $(X_0,\Omega)$ is the same as the type of the folding of the root lattice of the same type as $X_0$ by the action of $\Omega$.

\begin{remark}\label{A2rabsent}
The case where $X_0$ is of type $A_{2r}$ does not appear in the preceding table because, although the Dynkin diagram of type $A_{2r}$ has a non-trivial symmetry group, it was seen in Subsection~\ref{sub:groupaction} that the action of this symmetry group fails to lift correctly to the exceptional locus of the minimal resolution of $X_0$.
\end{remark}

The notion of symmetry has been added to simple singularities, therefore it is necessary to include it in the definition of deformations of simple singularities of type $B_r$, $C_r$, $F_4$ and $G_2$ as well. P. Slodowy gave the following definition:

\begin{definition}
A deformation of a simple singularity $(X_0,\Omega)$ is a deformation $\chi:X \rightarrow U$ of $X_0$ with $\Omega$-actions such that:
\vspace{-1ex}
\begin{itemize}\setlength\itemsep{0.3pt}
\item the action of $\Omega$ on $U$ is trivial, 
\item the action of $\Omega$ on $X$ induces the given one on $X_0$, 
\item the morphism $\chi$ is $\Omega$-invariant.
\end{itemize}

The \textit{semiuniversal deformation} of $(X_0,\Omega)$ is defined in an analogous way from the definition in Subsection~\ref{sub:semiunivdef}.
\end{definition}

In \cite{Slo80}, P. Slodowy constructed the semiuniversal deformation of $(X_0,\Omega)$ where $X_0$ is a hypersurface with an isolated singularity and $\Omega$ a reductive group. More precisely, he showed the existence of a semiuniversal deformation $\chi:X \rightarrow U$ of $X_0$ in the original sense with the following additional property: there are $\Omega$-actions on $X$ and $U$ such that $\chi$ is $\Omega$-equivariant and the restriction of $\Omega$ on $X_0$ is the given one. Furthermore, if $Y \rightarrow T$ is any $\Omega$-equivariant deformation of $X_0$, it can be induced from $\chi$ by an $\Omega$-equivariant morphism. From here, a semiuniversal deformation $\chi_{\Omega}$ of $(X_0,\Omega)$ is obtained through the following diagram: 
\begin{center}
\begin{tikzpicture}[scale=1,  transform shape]

\node (1) at ( -2,0) {$X \times_{U}U^\Omega$};
\node (2) at ( 0,0) {$X$};
\node (3) at ( 0,-2) {$U$};
\node (4) at ( -2,-2) {$U^\Omega$} ;

\draw  [decoration={markings,mark=at position 1 with
    {\arrow[scale=1.2,>=stealth]{>}}},postaction={decorate}] (1)  --  (2);
\draw  [decoration={markings,mark=at position 1 with
    {\arrow[scale=1.2,>=stealth]{>}}},postaction={decorate}] (2)  --   node [right] {$\chi$} (3);
\draw  [decoration={markings,mark=at position 1 with
    {\arrow[scale=1.2,>=stealth]{>}}},postaction={decorate}] (1)  --  node [left] {$\chi_{\Omega}$} (4);
\draw  [decoration={markings,mark=at position 1 with
    {\arrow[scale=1.2,>=stealth]{>}}},postaction={decorate}] (4)  --  (3);
\end{tikzpicture}
\end{center}
by inducing a deformation from $\chi$ using the base change $U^\Omega \rightarrow U$, where $U^\Omega$ is the fixed point set of $\Omega$ in $U$.

\subsubsection{Inhomogeneous simple singularities from Lie algebras of type $B_r$, $C_r$, $F_4$ and $G_2$}\label{sub: Brieskorngeneralised}

Let $\mathfrak{g}$ be a simple Lie algebra with adjoint group $G$, and let $e \in \mathfrak{g}$ be a nilpotent element. The Jacobson-Morozov theorem states that there exist $f, h \in \mathfrak{g}$ such that $(e, f, h)$ is an $\mathfrak{sl}_2$-triple. In \cite{Slo80} Section 7.5, the reductive centraliser of $e$ with respect to $h$ is defined as $C(e) =Z_G(e) \bigcap Z_G(h)$. The proof of the next proposition can be found in \cite{Slo80}.

\begin{proposition}
If $e \in \mathfrak{g}$ is regular nilpotent, then $C(e)=\{1\}$ and in particular $Z_G(e)$ is connected. If $e$ is subregular nilpotent then, according to type, we have
\medskip
\begin{center}
$\begin{array}{|c|c|c|c|c|c|c|c|c|c|}
\hline
\text{Type of } \mathfrak{g} & A_r \ (r \geq 2) & B_r \ (r \geq 2)& C_r \ (r \geq3)& D_r \ (r \geq4)& E_6 & E_7 & E_8 & F_4 & G_2 \\
\hline
C(e) & \mathbb{G}_m & \mathbb{G}_m \rtimes \mathbb{Z}/2\mathbb{Z}  & \mathbb{Z}/2\mathbb{Z} & \{1\} &  \{1\} &  \{1\} &  \{1\} & \mathbb{Z}/2\mathbb{Z} & \mathfrak{S}_3 \\
\hline
\end{array}$.
\captionof{table}{Reductive centralisers of subregular nilpotent elements}
\end{center}
In the semidirect product $\mathbb{G}_m \rtimes \mathbb{Z}/2\mathbb{Z} $,  the cyclic group $\mathbb{Z}/2\mathbb{Z} $ acts on $ \mathbb{G}_m$ by $x \mapsto x^{-1}$.
\end{proposition}

In Subsection~\ref{sub:semiunivdef}, we stated a theorem due to E. Brieskorn (Theorem~\ref{Brieskorn}) which gives the semiuniversal deformation of a simple singularity of type $A_r$, $D_r$ or $E_r$ from the adjoint quotient of a simple Lie algebra of the same type. This theorem can be extended to the inhomogeneous case (cf. \cite{Slo80} for the proof):

\begin{theorem}\emph{\textbf{(Slodowy).}}\label{Brieskorngeneralised}
Let $\mathfrak{g}$ be a simple Lie algebra of type $B_r$, $C_r$, $F_4$ or $G_2$ and $e$ a subregular nilpotent element of $\mathfrak{g}$. Then there exist a finite subgroup $\Xi$ of the reductive centraliser $C(e)$ of $e$, and a $\Xi$-stable transversal slice $S_e$ at $e$ to the $G$-orbit of $e$ such that the ($\Xi$-invariant) restriction of the adjoint quotient map $\chi:\mathfrak{g} \rightarrow \mathfrak{h}/W$ to $S_e$ realises a semiuniversal deformation of a simple singularity of the same type as $\mathfrak{g}$.
\end{theorem}

As a special case, the theorem states that the intersection $X_0=S \bigcap \mathcal{N}(\mathfrak{g})$ of $S_e$ with the nilpotent variety of $\mathfrak{g}$, endowed with the induced action of $\Xi$, is a simple singularity $(X_0,\Xi)$ of the same inhomogeneous type as $\mathfrak{g}$. 

It was mentioned in Proposition~\ref{propsubDynkin} that for a simple singularity whose Dynkin diagram $\Delta$ is homogeneous, singular configurations whose types correspond to subdiagrams of $\Delta$ can be found in the fibres of the semiuniversal deformation near the singular point. There is a similar result for inhomogeneous singularities. 

If $\Delta$ is an inhomogeneous Dynkin diagram, it was shown in Section~\ref{sec:folding} that it can be obtained as a quotient of a simply laced Dynkin diagram $\Delta_{\mathrm{unfold}}$ by the action of a group $\Xi$ of diagram isometries. The preimages by this quotient map in $\Delta_{\mathrm{unfold}}$ of subdiagrams of $\Delta$ are the $\Xi$-stable subdiagrams of $\Delta_{\mathrm{unfold}}$. The next corollary is also proved in \cite{Slo80}. 

\begin{corollary}
Let $\chi:X \rightarrow U$ be a deformation of a simple singularity $(X_0,\Xi)$ of inhomogeneous type $\Delta$, and let $X_\epsilon \neq X_0$ be a non-special fibre. Then there is a $\Xi$-stable subdiagram $\Delta'_{\mathrm{unfold}}$ of $\Delta_{\mathrm{unfold}}$ and a $\Xi$-equivariant type-preserving bijection of the connected components of $\Delta'_{\mathrm{unfold}}$ onto the singular points of $X_\epsilon$.
\end{corollary}

Let $\mathfrak{h}$ be a Cartan subalgebra of a simple Lie algebra of inhomogeneous type $\Delta$, and let $W$ be the associated Weyl group. Let $\restr{\chi}{S_e} :S_e \rightarrow \mathfrak{h}/W$ be the semiuniversal deformation obtained in Theorem~\ref{Brieskorngeneralised}. Let $\mathfrak{h}_{\mathrm{unfold}}$ be a Cartan subalgebra of a Lie algebra of type $\Delta_{\mathrm{unfold}}$ and denote by $\Xi$ the group of isometries corresponding to the folding $\Delta_{\mathrm{unfold}} \rightarrow \Delta$. We have the following diagram:
\vspace{-\topsep}
\begin{center}
\begin{tikzpicture}[scale=0.9,  transform shape]
\node (1) at (3,2) {$S_e$};
\node (2) at (3,0) {$\mathfrak{h}/W$};
\node (3) at (0,0) {$(\mathfrak{h}_{\mathrm{unfold}})^\Xi \cong \mathfrak{h}$};
\node (4) at (-2.8,0) {$\mathfrak{h}_{\mathrm{unfold}}$};

\draw  [decoration={markings,mark=at position 1 with
    {\arrow[scale=1.2,>=stealth]{>}}},postaction={decorate}] (1)  -- node [right]{$\restr{\chi}{S_e}$}  (2);
\draw  [decoration={markings,mark=at position 1 with
    {\arrow[scale=1.2,>=stealth]{>}}},postaction={decorate}] (3)  -- node [above]{$\pi$} (2);
\draw  [decoration={markings,mark=at position 1 with
    {\arrow[scale=1.2,>=stealth]{>}}},postaction={decorate}] (4)  -- node [above]{$p$} (3);
\end{tikzpicture}
\end{center}
\vspace{-\topsep}
where $p$ is the projection on the $\Xi$-invariant part of $\mathfrak{h}_{\mathrm{unfold}}$. It is natural to wonder if, given $h \in \mathfrak{h}_{\mathrm{unfold}}$, it is possible to anticipate the singular configuration of the fiber $\restr{\chi}{S_e}^{-1}(\pi(p(h)))$. 

Set $h \in  \bigcap_{\alpha \in \Phi}\mathrm{Ker}(\alpha) \subset \mathfrak{h}_{\mathrm{unfold}}$ with $\Phi$ a $\Xi$-stable sub-root system of the root system of type $\Delta_{\mathrm{unfold}}$, and assume $\Phi$ to be maximal. Then it follows that $p(h)$ is also in $\bigcap_{\alpha \in \Phi}\mathrm{Ker}(\alpha)$, so if $p(h)$ is seen as an element of $\mathfrak{h}$, it is in $\bigcap_{\beta \in \mathrm{Fold}(\Phi)}\mathrm{Ker}(\beta)$ where $\mathrm{Fold}(\Phi)$ is the sub-root system of the root system of type $\Delta$ obtained by folding $\Phi$. Hence the root system associated to the singular configuration of the fiber of $\restr{\chi}{S_e}$ above $\pi(p(h))$ will contain the root system $\mathrm{Fold}(\Phi)$. However, as the example below suggests, it is possible that $\mathrm{Fold}(\Phi)$ is not maximal for $p(h)$, and so the root system associated to $h$ is strictly contained in the root system associated to the singular configuration of the fiber of $\restr{\chi}{S_e}$ above $\pi(p(h))$.

\begin{example}\label{ConterexamplewithA5}
Let $\mathfrak{h}_{\mathrm{unfold}}$ be a Cartan subalgebra of a simple Lie algebra of type $\Delta_{\mathrm{unfold}}=A_5$. Let $(e_1,\ldots,e_6)$ be an orthonormal basis of a vector space of dimension $6$ such that $\mathfrak{h}_{\mathrm{unfold}}=\{h=(h_1,\ldots, h_6) \ | \ \sum_{i=1}^6h_i=0\}$, and define $(\epsilon_1,\ldots,\epsilon_6)$ its dual basis. Following \cite{Bou68}, the simple roots of $A_5$ are 
\vspace{-\topsep}
\begin{center}
\renewcommand{\arraystretch}{1.2}
\raisebox{-.55\height}{$\left\lbrace\begin{array}{ccl}
\alpha_1 & = & \epsilon_1-\epsilon_2, \\
\alpha_2 & = & \epsilon_2-\epsilon_3, \\
\alpha_3 & = & \epsilon_3-\epsilon_4, \\
\alpha_4 & = & \epsilon_4-\epsilon_5, \\
\alpha_5 & = & \epsilon_5-\epsilon_6. \\
\end{array}\right.$}
\end{center}
Because $\Xi=\mathbb{Z}/2\mathbb{Z}=<\sigma>$, we take the map $p$ to be the averaging over $\Xi$, i.e. $p(h)=\frac{1}{2}(h+\sigma.h)$. It follows that if $h=\begin{pmatrix} h_1 \\ h_2 \\ -(h_1+h_2) \\ -(h_1+h_2) \\ h_2 \\ h_1 \end{pmatrix}$, then $p(h)=0$. \\
Set $h=(1,2,-3,-3,2,1)$. Then the positive roots of $A_5$ that are canceled by $h$ are $\alpha_3, \alpha_2+\alpha_3+\alpha_4$ and $\alpha_1+ \alpha_2+\alpha_3+\alpha_4+\alpha_5$. Using the negative roots, we obtain $h \in \bigcap_{\alpha \in \Phi}\mathrm{Ker}(\alpha)$ where $\Phi$ is a $\Xi$-stable sub-root system of type $A_1+A_1+A_1$ and $\Phi$ is maximal. But as the projection of $h$ on $(\mathfrak{h}_{\mathrm{unfold}})^\Xi$ is zero, it follows that the singular configuration of the fiber $\restr{\chi}{S_e}^{-1}(\pi(p(h))=\restr{\chi}{S_e}^{-1}(0)$ is of type $\Delta_{\mathrm{unfold}}=A_5$. Therefore the root system associated to $h$ is only contained in the root system associated to the singular configuration of the fiber of $\restr{\chi}{S_e}$ above $\pi(p(h))$.
\end{example}

\subsubsection{Inhomogeneous simple singularities from Lie algebras of type $A_r$, $D_r$ and $E_6$}\label{sub:inhomogenefromADE}
Like in the previous section, let $\Delta$ be a Dynkin diagram of type $A_{2r-1}$, $D_r$ or $E_6$ and $\mathfrak{g}$ a simple Lie algebra of type $\Delta$ with adjoint simple group $G$. Let $e \in \mathfrak{g}$ be a subregular nilpotent element of $\mathfrak{g}$ and $(e, f, h)$ an $\mathfrak{sl}_2$-triple of $\mathfrak{g}$. Define $S_e=e + \mathfrak{z}_\mathfrak{g}(f)$ a Slodowy slice at $e$ and set $\delta$ the restriction to $S_e$ of the adjoint quotient map of $\mathfrak{g}$. One can show that the automorphism group $\mathrm{Aut}(\Delta)$ of the Dynkin diagram acts on $S_e$ as well as on $\mathfrak{h}/W$, and makes $\delta$ equivariant. As a result, there is an action of $\mathrm{Aut}(\Delta)$ on the special fibre $X=\delta^{-1}(0)$. Let $\Delta_0$ be the unique inhomogeneous Dynkin diagram such that $\prescript{}{h}{\Delta_0}=\Delta$ and $AS(\Delta_0)=\mathrm{Aut}(\Delta)$, with $AS(\Delta_0)$ being the associated symmetry group of $\Delta_0$ defined by
\begin{center}
$AS(\Delta_0)=\left\{\renewcommand{\arraystretch}{1.5} \begin{array}{l}
      \mathfrak{S}_3 \text{ if } \Delta_0=G_2, \\
    \mathbb{Z}/2\mathbb{Z} \text{ otherwise}. 
    \end{array}
\right.$
\end{center}
The following two theorems are proved in Section 8.8 of \cite{Slo80}:

\begin{theorem}
$(X,AS(\Delta_0))$ is a simple singularity of type $\Delta_0$.
\end{theorem}

Let $G_0$ denote the simple adjoint group of type $\Delta_0$ with Lie algebra $\mathfrak{g}_0$. Let $(e_0,f_0,h_0)$ be an $\mathfrak{sl}_2$-triple with $e_0$ a subregular nilpotent element of $\mathfrak{g}_0$, and define $S_0=e_0 + \mathfrak{z}_{\mathfrak{g}_0}(f_0)$. Let $\delta_0:S_0 \rightarrow \mathfrak{h}_0/W_0$ denote the restriction to $S_0$ of the adjoint quotient map of $\mathfrak{g}_0$. 

\begin{theorem}\label{thm:restrictionsud}
The $AS(\Delta_0)$-equivariant deformation $\delta:S \rightarrow \mathfrak{h}/W$ of $X$ is $AS(\Delta_0)$-semiuniversal, and the restriction $\delta^{AS(\Delta_0)}$ of $\delta$ over the fixed point space $(\mathfrak{h}/W)^{AS(\Delta_0)}$ is isomorphic to $\delta_0$.
\end{theorem}

\begin{remark}\label{h0andh1}
The above theorem allows an identification of $\mathfrak{h}_0 /W_0$ with $(\mathfrak{h}/W)^{AS(\Delta_0)}$. However, another identification is possible. The group $AS(\Delta_0)$ acts on the Dynkin diagram $\Delta$ and its action can be naturally extended to $\mathfrak{h}$. Set $\mathfrak{h}_1=\mathfrak{h}^{AS(\Delta_0)}$ and $W_1=\{w \in W \ | \ w\gamma=\gamma w, \forall \gamma \in AS(\Delta_0)\}$. Then the natural map $\mathfrak{h}_1 \rightarrow \mathfrak{h}/W$ induces a $\mathbb{G}_m$-equivariant morphism $\mathfrak{h}_1/W_1 \rightarrow (\mathfrak{h}/W)^{AS(\Delta_0)}$. The $\mathbb{G}_m$-weights of this morphism are the same on $\mathfrak{h}_1/W_1$ as they are on $(\mathfrak{h}/W)^{AS(\Delta_0)}$, and they are strictly positive. Using \cite{Slo80} Section 8.1, Lemma 3, one finds that it is an isomorphism.
\end{remark}

\section{Foldings of root systems and quivers}\label{sec:Tanisaki}

In this section, we present a construction by T. Tanisaki of some kind of folding of the representation space of a quiver whose underlying graph is a simply laced Dynkin diagram. Subsequently, observations on the different types of folding will be given.

\subsection{Foldings and quivers representations}

A theorem due to P. Gabriel states if a connected quiver is of finite type, then its graph is a simply laced Dynkin diagram of type $\Delta$, and there is a bijection from the set of isomorphism classes of indecomposable representations of the quiver onto the set of positive roots of a root system of type $\Delta$ given by the dimension vector function. This result has been extended by T. Tanisaki in \cite{Tani80} to the non-simply laced Dynkin diagrams using an approach based on foldings of root systems. 

Let $Q=(Q_0, Q_1, s, t)$ be a simply laced quiver with the set $Q_0$ of vertices, the set $Q_1$ of arrows, the source map $s$ and the target map $t$. Let us label the vertices of $Q$ by integers. An automorphism of $Q$ is a permutation $\sigma$ of the set $Q_0$ such that a pair $(\sigma(i),\sigma(j))$ forms an edge of $Q$ if and only if the pair $(i, j)$ is an edge of $Q$. The set of automorphisms of $Q$ is a group $\mathrm{Aut}(Q)$ given by
\begin{center}
$\mathrm{Aut}(Q)=\{\sigma=(\sigma_0,\sigma_1) \in \mathfrak{S}^{Q_0} \times \mathfrak{S}^{Q_1} \ | \ s(\sigma_1(\alpha))=\sigma_0(s(\alpha)), \ t(\sigma_1(\alpha))=\sigma_0(t(\alpha)) \text{ for all } \alpha \in Q_1\}$,
\end{center}
where $\mathfrak{S}^{Q_i}$ denotes the permutation group of the set $Q_i$, $i=0$, $1$. For any $\sigma \in \mathrm{Aut}(Q)$, one defines the functor $F^\sigma:\mathrm{Rep}(Q) \rightarrow \mathrm{Rep}(Q)$ in the following way:

\renewcommand{\labelitemi}{$\bullet$}
\begin{itemize}\setlength\itemsep{0.3pt}
\item For any representation $(V_i, f_\alpha)_{i \in Q_0,\alpha \in Q_1}$ of $Q$, set $F^\sigma((V_i, f_\alpha)_{i \in Q_0,\alpha \in Q_1})=(W_i, g_\alpha)_{i \in Q_0,\alpha \in Q_1}$ such that $W_i=V_{\sigma_0^{-1}(i)}$ for any $i \in Q_0$, and $g_\alpha=f_{\sigma_1^{-1}(\alpha)}$ for any $\alpha \in Q_1$.
\item For any morphism $\varphi:(V_i, f_\alpha)_{i \in Q_0,\alpha \in Q_1} \rightarrow (W_i, g_\alpha)_{i \in Q_0,\alpha \in Q_1}$ of representations of $Q$, the morphism $F^\sigma(\varphi):F^\sigma(V_i, f_\alpha) \rightarrow F^\sigma(W_i, g_\alpha)$ is defined by $(F^\sigma(\varphi))_i=\varphi_{\sigma_0^{-1}(i)}$ for all $i \in Q_0$.
\end{itemize}

\begin{definition}
Let $G$ be a subgroup of $\mathrm{Aut}(Q)$. Define the full subcategory $\mathrm{Rep}(Q)^G$ of $\mathrm{Rep}(Q)$ as follows: set $(V_i, f_\alpha)_{i \in Q_0,\alpha \in Q_1} \in \mathrm{Rep}(Q)$. Then $(V_i, f_\alpha)_{i \in Q_0,\alpha \in Q_1} \in \mathrm{Rep}(Q)^G$ if, for any $g \in G$, the representation $F^g((V_i, f_\alpha)_{i \in Q_0,\alpha \in Q_1})$ is isomorphic to $(V_i, f_\alpha)_{i \in Q_0,\alpha \in Q_1}$ in the category $\mathrm{Rep}(Q)$. 
\end{definition} 

One can see $\mathrm{Rep}(Q)^G$ as the invariants up to isomorphism of $\mathrm{Rep}(Q)$ by the action of $G$ through the functor $F$. 

We are now able to state Tanisaki's theorem (cf. \cite{Tani80}).

\begin{theorem}\emph{\textbf{(Tanisaki).}}\label{Tanisaki}
Let $Q$ be a connected quiver and $G$ a subgroup of $\mathrm{Aut}(Q)$ which preserves the orientation of $Q$. Then the following assertions are verified:

\begin{enumerate}[(i)]\setlength\itemsep{0.3em}
\item In the category $\mathrm{Rep}(Q)^G$, the Krull-Remak-Schmidt theorem (which states the unicity of the decomposition of a module into a direct sum of indecomposable submodules) is valid. 

\item There are only finitely many isomorphism classes of indecomposable objects in $\mathrm{Rep}(Q)^G$ if and only if: 
\vspace{-\topsep}
\begin{itemize}\setlength\itemsep{0.3pt}
\item Either $Q$ is a simply laced Dynkin diagram and $G$ is trivial,
\item Or $Q$ and $G$ are as in the following table:
\end{itemize}
\begin{center}
\begin{tabular}[h]{|c|c|c|}
\hline
 $\Delta(Q,G)$ & Graph of $Q$ & $G$\\
\hhline{|=|=|=|}
 \multirow{-1.5}{*}{$B_r$ ($r\geq 2$)}& 
 \begin{tikzpicture}[scale=0.6,  transform shape]
\tikzstyle{point}=[circle,draw,fill]
\tikzstyle{ligne}=[thick]
\tikzstyle{pointille}=[thick,dotted]

\node (1) at ( -4,0) [point] {};
\node (2) at ( -2,0) [point]{};
\node (3) at ( 0,0) [point] {};
\node (4) at ( 3,0) [point]{} ;
\node (5) at ( 5,0) [point] {};

\node (6) at ( -4,0.5) {1};
\node (7) at ( -2,0.5)  {2};
\node (8) at ( 0,0.5)  {3};
\node (9) at ( 3,0.5) {2r-2};
\node (10) at ( 5,0.5) {2r-1};

\draw [ligne] (1)  --  (2);
\draw [ligne] (2)  --  (3);
\draw [pointille] (3)  --  (4);
\draw [ligne] (4)  --  (5);

\end{tikzpicture}&  \multirow{-1.5}{*}{$\{1, \tau\}$ with $\tau(i)=2r-i$} \\

\hline
 \multirow{-3.8}{*}{$C_r$ ($r\geq 3$)} & 
  \begin{tikzpicture}[scale=0.6,  transform shape]
\tikzstyle{point}=[circle,draw,fill]
\tikzstyle{ligne}=[thick]
\tikzstyle{pointille}=[thick,dotted]

\node (1) at ( -6,0) [point] {};
\node (2) at ( -4,0) [point] {};
\node (3) at ( -2,0) [point] {};
\node (4) at ( 0,0) [point] {};
\node (5) at ( 2,1) [point] {};
\node (6) at ( 2,-1) [point] {};

\node (7) at ( -6,0.5)  {1};
\node (8) at ( -4,0.5)  {2};
\node (9) at ( -2,0.5)  {3};
\node (10) at ( 0,0.5) {r-1};
\node (11) at ( 2,1.5) {r};
\node (12) at ( 2,-0.5) {r+1};

\draw [ligne] (1)  --  (2);
\draw [ligne] (2)  --  (3);
\draw [pointille] (3)  --  (4);
\draw [ligne] (4)  --  (5);
\draw [ligne] (4)  --  (6);

\end{tikzpicture}  &  \multirow{-3.8}{*}{{\small \begin{tabular}{c} $\{1, \tau\}$ with  \\ $\tau(r)={r+1}$, \\ $\tau({r+1})={r}$, \\ $\tau(i)=i$ if $i \neq r, r+1$. \end{tabular}}} \\
\hline
 \multirow{-3.8}{*}{$F_4$} &  \begin{tikzpicture}[scale=0.6,  transform shape]
\tikzstyle{point}=[circle,draw,fill]
\tikzstyle{ligne}=[thick]
\tikzstyle{pointille}=[thick,dotted]

\node (1) at ( -4,0) [point] {};
\node (2) at ( 0,-2) [point] {};
\node (3) at ( -2,0) [point] {};
\node (4) at ( 0,0) [point] {};
\node (5) at ( 2,0) [point] {};
\node (6) at ( 4,0) [point] {};

\node (7) at ( -4,0.5)  {1};
\node (8) at ( -0.5,-2)  {2};
\node (9) at ( -2,0.5)  {3};
\node (10) at ( 0,0.5) {4};
\node (11) at ( 2,0.5) {5};
\node (12) at ( 4,0.5) {6};

\draw [ligne] (1)  --  (3);
\draw [ligne] (3)  --  (4);
\draw [ligne] (4)  --  (2);
\draw [ligne] (4)  --  (5);
\draw [ligne] (5)  --  (6);

\end{tikzpicture} &  \multirow{-3.9}{*}{{\small $\begin{array}{c} \{1, \tau\} \text{ with} \\ \tau(1)=6, \tau(2)=2, \\ \tau(3)=5, \tau(4)=4, \\ \tau(5)=3, \tau(6)=1. \end{array}$}} \\
\hline
 \multirow{-5}{*}{$G_2$} & 
 \begin{tikzpicture}[scale=0.6,  transform shape]
\tikzstyle{point}=[circle,draw,fill]
\tikzstyle{ligne}=[thick]
\tikzstyle{pointille}=[thick,dotted]

\node (3) at ( -2.1213,0) [point] {};
\node (4) at ( 0,0) [point] {};
\node (5) at ( 1.5,1.5) [point] {};
\node (6) at ( 1.5,-1.5) [point] {};

\node (9) at ( -2.1213,0.5)  {1};
\node (10) at ( 0,0.5) {2};
\node (11) at ( 1.5,2) {3};
\node (12) at ( 1.5,-1) {4};

\draw [ligne] (3)  --  (4);
\draw [ligne] (4)  --  (5);
\draw [ligne] (4)  --  (6);

\end{tikzpicture}  &  \multirow{-5}{*}{\begin{tabular}{l} $G$ acts transitively \\ on $\{1,3,4\}$ and \\  fixes $2$. \end{tabular}}  \\
 \hline
 \end{tabular}
  \captionof{table}{Foldings by T. Tanisaki}
  \label{FoldingsTanisaki}
\end{center}
If $Q$ is a simply laced Dynkin diagram and $G$ is trivial, set $\Delta(Q,G)$ as the type of the Dynkin diagram.

\item Let $(Q,G,\Delta(Q,G))$ be a triple as in (ii). Then there exists a natural one-to-one correspondence between the set of isomorphism classes of indecomposable objects of $\mathrm{Rep}(Q)^G$ and the set of positive roots of the root system of type $\Delta(Q,G)$.
\end{enumerate}
\end{theorem}

\begin{remark}
In the case where $Q$ is of type $D_4$ and $\Delta(Q,G)$ is of type $G_2$, the group $\mathfrak{S}_3$ can be replaced by the smaller group $\mathbb{Z}/3\mathbb{Z}$ and the results of Theorem~\ref{Tanisaki} remain valid.
\end{remark}

\begin{remarks}
\begin{enumerate}\setlength\itemsep{0.3em}
\item One notices that $\Delta(Q,G)$ is the type obtained by folding the root lattice associated with the graph of $Q$ by the action of $G$.
\item The type $A_{2r}$ with the symmetry given by $\sigma(i)={2r+1-i}$ where $G=\mathbb{Z}/2\mathbb{Z}=<\sigma>$ does not appear in the previous theorem.
\begin{center}
 \begin{tikzpicture}[scale=0.5,  transform shape,>=angle 60]
\tikzstyle{point}=[circle,draw,fill]
\tikzstyle{ligne}=[thick]
\tikzstyle{pointille}=[thick,dotted]

\node (1) at ( -2,0) [point]{};
\node (2) at ( 0,0) [point] {};
\node (3) at ( 3,0) [point]{} ;
\node (4) at ( 5,0) [point] {};
\node (5) at ( 8,0) [point] {};
\node (6) at ( 10,0) [point,] {};

\node (7) at ( -2,0.5)  {1};
\node (8) at ( 0,0.5)  {2};
\node (9) at ( 3,0.5) {r};
\node (10) at ( 5,0.5) {r+1};
\node (11) at ( 8,0.5) {2r-1};
\node (12) at ( 10,0.5) {2r};

\node (13) at ( -1.7,-0.3) {};
\node (14) at ( 0.3,-0.3) {};
\node (15) at ( 3.3,-0.3) {} ;
\node (16) at ( 4.7,-0.3) {};
\node (17) at ( 7.7,-0.3) {};
\node (18) at ( 9.7,-0.3) {};

\node (19) at ( 4,-4) {{\Large $\sigma$}};

\draw [ligne] (1)  --  (2);
\draw [pointille] (2)  --  (3);
\draw [ligne] (3)  --  (4);
\draw [pointille] (4)  --  (5);
\draw [ligne] (5)  --  (6);

\draw[decoration={markings, mark=at position 1 with {\arrow[scale=1.2,>=stealth]{>}}},postaction={decorate}]  (13.west) to[out=-60, in=-120] (18.east);
\draw [decoration={markings, mark=at position 1 with {\arrow[scale=1.2,>=stealth]{>}}},postaction={decorate}]  (14.west) to[out=-60, in=-120] (17.east);
\draw [decoration={markings, mark=at position 1 with {\arrow[scale=1.2,>=stealth]{>}}},postaction={decorate}]  (15.west) to[out=-60, in=-120] (16.east);

\draw [decoration={markings, mark=at position 1 with {\arrow[scale=1.2,>=stealth]{>}}},postaction={decorate}] (9.71,-0.495) -- (9.88,-0.2);
\draw [decoration={markings, mark=at position 1 with {\arrow[scale=1.2,>=stealth]{>}}},postaction={decorate}] (7.71,-0.495) -- (7.89,-0.2);
\draw [decoration={markings, mark=at position 1 with {\arrow[scale=1.2,>=stealth]{>}}},postaction={decorate}] (4.71,-0.46) -- (4.92,-0.2);

\draw [decoration={markings, mark=at position 1 with {\arrow[scale=1.2,>=stealth]{>}}},postaction={decorate}] (-1.71,-0.495) -- (-1.88,-0.2);
\draw [decoration={markings, mark=at position 1 with {\arrow[scale=1.2,>=stealth]{>}}},postaction={decorate}] (0.29,-0.495) -- (0.12,-0.2);
\draw [decoration={markings, mark=at position 1 with {\arrow[scale=1.2,>=stealth]{>}}},postaction={decorate}] (3.29,-0.46) -- (3.08,-0.2);

\end{tikzpicture}
\end{center}
Theorem~\ref{Tanisaki} is based upon choosing an orientation of the quiver $Q$ and a subgroup $G$ of $\mathrm{Aut}(Q)$ which preserves the orientation. However, the non-trivial symmetry group of $A_{2r}$ is $\mathbb{Z}/2\mathbb{Z}$ and the action of $\sigma$ would send the arrow $r \rightarrow r+1$ to $r+1 \rightarrow r$, thus reversing the orientation of this edge. It is the same if we choose to orient the edge by $r+1 \rightarrow r$. Therefore there cannot be a subgroup $G$ of $\mathrm{Aut}(Q)$ which preserves the orientation and so Tanisaki's result is not valid anymore.
\end{enumerate}
\end{remarks}

\subsection{Observations on foldings}

In Section~\ref{sec:folding}, we realised foldings on simply laced Lie algebras and their root lattices. The results were dual to each others (cf. Table~\ref{tablefoldings}). Then the definition by P. Slodowy of the inhomogeneous simple singularities was given (cf. Table~\ref{definhomogeneous}). This definition is based on the symmetries of the Dynkin diagrams. P. Slodowy also computed two variants of the McKay correspondence (\cite{McK80}). The first one is called "by restriction" and works as follows: start with a pair $(\Gamma, \Gamma')$ as in Table~\ref{definhomogeneous}. The irreducible representations of $\Gamma'$ can be restricted to $\Gamma$. Applying the same procedure as in the McKay correspondence to the equivalence classes of these restricted representations leads to the Cartan matrix of the extended Dynkin diagram of type $\widetilde{\Delta^\vee}(\Gamma,\Gamma')$. The second variant is called "by induction". Here we induce representations of $\Gamma'$ from the irreducible representations of $\Gamma$, apply the regular procedure of the McKay correspondence and get the Cartan matrix dual to the one obtained by restriction (details can be found in \cite{Slo80} Annex III). Finally, T. Tanisaki realised foldings on representation spaces of Dynkin quivers (cf. Table~\ref{FoldingsTanisaki}). 

The way P. Slodowy defined inhomogeneous singularities is in adequacy with foldings of root lattices as well as with foldings using quiver representations. Furthermore, the McKay correspondences by induction and restriction establish a similar link between the homogeneous and the inhomogeneous Dynkin diagrams, however less direct because of the extension of the Dynkin diagrams as well as their duals.  

The connections between the different foldings are summarised in the following diagram:
\begin{figure}[H]
\resizebox{0.96\textwidth}{!}{\rotatebox{0}{
\begin{tikzpicture}[every text node part/.style={align=center}]
\node[draw,ellipse,text width=4cm](F) at (0,0) {Foldings};
\node[draw,ellipse,text width=4cm](HMC) at ( -9,0) {Homogeneous \\ McKay correspondence} ; 
\node[draw,ellipse,text width=4cm](IMC) at (9,0) {Inhomogeneous \\ McKay correspondence}; 
\node[draw,rectangle,text width=3cm](EDD) at ( -9, -5) {Extended Dynkin \\ diagrams $\widetilde{\Delta(\Gamma)}$  \\ $\widetilde{A_r}, \ \widetilde{D_r}, \ \widetilde{E_r}$}; 

\node[draw,rectangle,text width=3cm](Gamma) at ( -3.5, -2.5) {$\Gamma \subset \ \mathrm{SU}_2$  \\ finite}; 
\node[draw,rectangle,text width=3cm](singularity) at ( -3.5, -7.5) {Singularity $\cc^2 / \Gamma$  \\ resolution $\rightarrow \Delta(\Gamma)$ \\ $A_r, \ D_r, \ E_r$  };
 
\node[draw,rectangle,text width=3cm](lie algebra) at ( -5, -12.5) {Lie algebra $\mathfrak{g}_\Phi$}; 

\node[draw,rectangle,text width=3cm](root system) at ( -2, -10) {Root system $\Phi$}; 

\node[draw,rectangle,text width=3cm](inv) at (2.5, -10) {Invariants $\Phi^{\sigma}$}; 
\node[draw,rectangle,text width=3cm](dual inv) at (2.5, -12.5) {Invariants $(\mathfrak{g}_\Phi)^\sigma$}; 

\node[draw,rectangle,text width=3cm](Gamma') at (3.5, -2.5) {$\Gamma \lhd \Gamma' \subset \ \mathrm{SU}_2$  \\ finite}; 
\node[draw,rectangle,text width=4cm](singularity') at (3.5, -7.5) {Singularity $\left(\cc^2/ \Gamma, \Gamma'/\Gamma \right)$  \\ $\Delta(\Gamma, \Gamma')$ \\ $B_r, \ C_r, \ F_4, \ G_2$  }; 

\node[draw, text width = 4.2 cm, text centered](Delta) at (9.5, -2.5) {$\widetilde{\Delta^\vee(\Gamma,\Gamma')}$  \\ $\widetilde{C_r}, \ \widetilde{B_r}, \ \widetilde{F_4}, \ \widetilde{G_2}$}; 
\node[draw, text width = 4.2 cm, text centered](Delta*) at (9.5, -7.5) {$\left(\widetilde{\Delta^\vee(\Gamma,\Gamma')}\right)^\vee$  \\ $\left(\widetilde{C_r}\right)^\vee \hspace{-0.2cm}, \left(\widetilde{B_r}\right)^\vee \hspace{-0.2cm}, \left(\widetilde{F_4}\right)^\vee \hspace{-0.2cm}, \left(\widetilde{G_2}\right)^\vee$}; 

\draw[->,>=stealth] (Gamma) --  (EDD)node[above, pos=0.7, rotate = 25]{Homogeneous \\ McKay};
\draw[->,>=stealth] (singularity) --  (EDD)node[below, pos=0.5, rotate =  -25]{extended};

\draw[->,>=stealth] (Gamma) --  (singularity)  node[right, pos=0.5]{quotient};

\draw[->,>=stealth, dashed] (Gamma) -- (Gamma');
\draw[->,>=stealth, dashed] (singularity) -- (singularity');

\draw[->,>=stealth] ( -5, -8.1) -- (lie algebra);
\draw[->,>=stealth] ( -2, -8.1) -- (root system);

\draw[->,>=stealth, dashed] (root system) -- (inv);
\draw[->,>=stealth, dashed] (lie algebra) -- (dual inv);

\draw[->,>=stealth] (Gamma') --  (Delta)node[above, pos=0.5]{McKay \\ restriction};

\draw[->,>=stealth] (5.8, -2.5)|-  (Delta*)node[above, pos=0.22, rotate = 90]{McKay induction};

\draw[<->,>=stealth] (inv) -- (dual inv)node[right, pos=0.5]{dual};

\draw (inv)  -- (6.7, -10)node[below, pos=0.5]{affinisation \\ of the dual};
\draw[->,>=stealth] (6.7, -10) |- (7.26, -2.8);

\draw[->,>=stealth] (dual inv) -| (Delta*)node[below, pos=0.25]{dual of the affinisation};

\draw[<->,>=stealth] (Delta) -- (Delta*)node[right, pos=0.5]{dual};

\end{tikzpicture}}}
\caption{Foldings and McKay correspondences}
\end{figure}
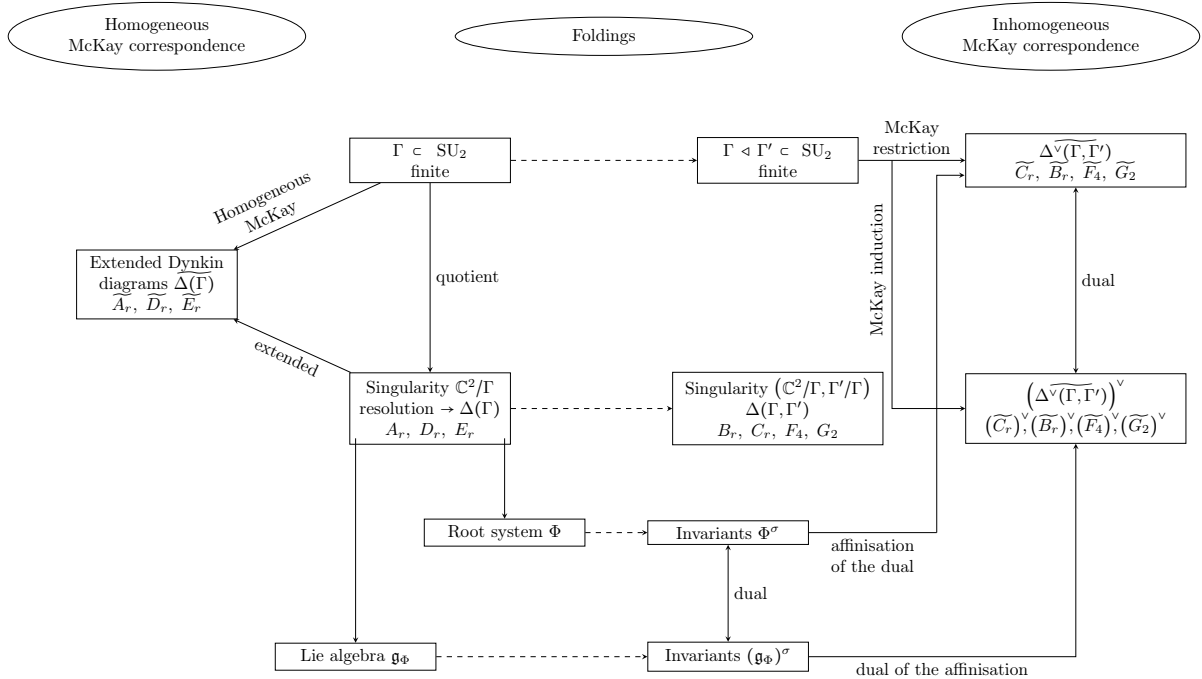

In the next section, we are going to present a construction by H. Cassens and P. Slodowy of the semiuniversal deformations of the homogeneous simple singularities using quiver representations and symplectic geometry. Their method will then be generalised to obtain the semiuniversal deformations of the inhomogeneous simple singularities. The process is based on the affinisation of the Dynkin diagram associated to the singularity through its minimal resolution, as illustrated below:

\begin{figure}[H]
\resizebox{0.96\textwidth}{!}{\rotatebox{0}{
\begin{tikzpicture}[every text node part/.style={align=center}]
\node[draw,ellipse,text width=4cm](QS) at (-4.5,0) {Quivers and symmetries};
\node[draw,ellipse,text width=4cm](DRwQ) at (4.5,0) {Deformations and resolutions with quivers} ;

\node[draw,rectangle,text width=3cm](A) at (-9,-2.5) {Quiver of type \\ $\Delta = A_r, \ D_r, \ E_r$};
\node[draw,rectangle,text width=3cm](B) at (-9,-10.5) {$B_r, \ C_r, \ F_4, \ G_2 $ \\ = quiver of type \\ $\Delta = A_r, \ D_r, \ E_r  +  \mathrm{Aut}(\Delta)$  };

\node[draw,rectangle,text width=3cm](McK) at (0,-2.5) {McKay quiver \\ $\widetilde{\Delta} = \widetilde{A_r}, \ \widetilde{D_r}, \ \widetilde{E_r}$};
\node[draw,rectangle,text width=3cm](tilde) at (0,-6.5) {$\widetilde{\Delta} = \widetilde{A_r}, \ \widetilde{D_r}, \ \widetilde{E_r}  +  \mathrm{Aut}(\widetilde{\Delta})$  };
\node[draw,rectangle,text width=5.8cm](tildeMcK) at (0,-10.5) {\stackon[0pt]{$\Delta = A_r, \ D_r, \ E_r  +  \mathrm{Aut}(\Delta)$}{\vstretch{1.5}{\hstretch{25.0}{\sim}}} };

\node[draw,rectangle,text width=3cm](SDRA) at (9,-2.5) {Semiuniversal deformations and \\ resolutions of type $A_r, \ D_r, \ E_r$}; 
\node[draw,rectangle,text width=3cm](SDRB) at (9,-10.5) {Semiuniversal deformations and \\ resolutions (in progress) of type \\ $B_r, \ C_r, \ F_4, \ G_2$}; 

\draw[->,>=stealth] (A) --  (McK)node[above, pos=0.5]{affinisation};
\draw[->,>=stealth] (McK) --  (SDRA)node[above, pos=0.5]{Section~\ref{sec:SlodowyCassens}};

\draw[->,>=stealth] (B) --  (tildeMcK)node[above, pos=0.5]{affinisation};
\draw[->,>=stealth] (tildeMcK) --  (SDRB)node[above, pos=0.5]{Section~\ref{sec:InhomogeneousDeformations}};

\draw[->,>=stealth] (McK) --  (tilde);
\draw[double, double distance=2.5pt, line width=1pt]  (tilde) -- (tildeMcK);
\draw[-, line width=1pt]  (-0.2,-8.6) -- (0.2,-8.4);

\draw[->,>=stealth] (A) --  (B)node[right, pos=0.5]{folding \\ (root system)};
\draw[->,>=stealth] (SDRA) --  (SDRB)node[left, pos=0.5]{folding \\ ( + $\mathrm{Aut}(\Delta)$)};

\end{tikzpicture}}}
\caption{Quivers, symmetries and simple singularities}
\end{figure}
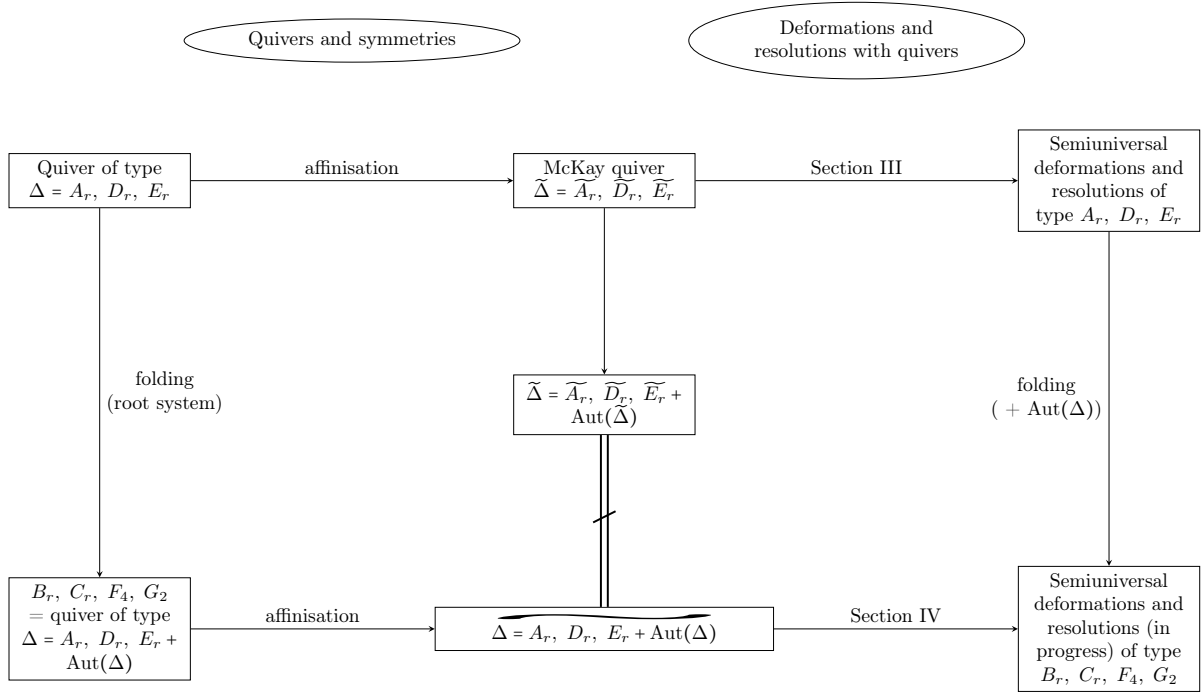 

\section{Deformations of homogeneous simple singularities}\label{sec:SlodowyCassens}

In this section, we present a construction by H. Cassens and P. Slodowy (cf. \cite{CaSlo98}) of the semiuniversal deformations of the simple singularities of type $A_r$, $D_r$ and $E_r$. 

Let $\Gamma$ be a finite subgroup of $\mathrm{SU}_2$ and $\Delta(\Gamma)$ the associated Dynkin diagram (cf. Subsection~\ref{subsub:definitionsimplesing}). Let $R_0,\ldots ,R_r$ be the irreducible representations of $\Gamma$ with $R_0$ the trivial one, and let $N$ be the natural representation obtained through the inclusion $\Gamma \ \mathlarger{\mathlarger{\subset}} \ \mathrm{SU}_2$. The regular representation of $\Gamma$ is $R= \bigoplus_{i=0}^{r} d_i R_i$ with $d_i=\mathrm{dim} \  R_i$. Then define
\begin{center}
\renewcommand{\arraystretch}{1.2}
$\begin{array}[t]{ccl}
M(\Gamma) & = & (\mathrm{End}(R) \otimes N)^\Gamma=(R^* \otimes R \otimes N)^\Gamma, \\
 & = &\displaystyle \mathrm{Hom}_\Gamma(R,R \otimes N)=\mathrm{Hom}_\Gamma(\bigoplus_{i=0}^rR_i \otimes \cc^{d_i},\bigoplus_{j=0}^rR_j \otimes N \otimes \cc^{d_j}), \\
 & = &\displaystyle \bigoplus_{i, j=0}^r\mathrm{Hom}_{\Gamma}(R_i ,R_j \otimes N) \otimes \mathrm{Hom}(\cc^{d_i},\cc^{d_j}), \\
 & = & \displaystyle\bigoplus_{\begin{tikzpicture}[scale=0.3,  transform shape]
\tikzstyle{point}=[circle,draw,fill]
\tikzstyle{ligne}=[thick]
\clip (-1.7,-0.2) rectangle (0.2, 0.9);
\node (1) at ( -1.5,0) [point]{};
\node (2) at ( 0,0) [point] {};
\node (3) at ( -1.5,0.6)  {{\huge i}};
\node (4) at ( 0,0.6)  {{\huge j}};
\draw [ligne] (1)  --  (2);

\end{tikzpicture}} \mathrm{Hom}(\cc^{d_i},\cc^{d_j}) \text{ by McKay's correspondence}, \\
 & = & \mathrm{Rep}(Q,\underline{d}),
\end{array}$
\end{center}
for a quiver $Q$ whose vertices are the vertices of the extended Dynkin diagram $\widetilde{\Delta}(\Gamma)$, with two arrows (one in each direction) for any edge in $\widetilde{\Delta}(\Gamma)$, and whose dimension vector is $\underline{d}=(d_0,d_1,\ldots .,d_r)$. Such a quiver is called a \textit{McKay quiver}. For every arrow $a:i \rightarrow j$ of $Q$, the opposite arrow $j \rightarrow i$ is denoted $\bar{a}$.

\begin{remark}
The dimension vector $\underline{d}$ associated to the McKay quiver $Q$ based on $\widetilde{\Delta}(\Gamma)$ is the same as the minimal imaginary root of the extended root system of type $\widetilde{\Delta}(\Gamma)$ (cf. \cite{Kac90}).
\end{remark}

\begin{example} If $\Gamma= \mathcal{O}$, then $\Delta(\Gamma)=E_7$ and the quiver $Q$ is
\begin{center}
\begin{tikzpicture}[scale=0.6, transform shape,>=angle 60]
\tikzstyle{point}=[circle,draw,fill]
\tikzstyle{ligne}=[thick]

\node (1) at (-4,0) [point]{};
\node (2) at ( -2,0) [point] {};
\node (3) at ( 0,0) [point]{};
\node (4) at ( 2,0) [point] {};
\node (5) at ( 4,0) [point]{};
\node (6) at ( 6,0) [point] {};
\node (7) at ( 8,0) [point]{};
\node (8) at ( 2,-2) [point] {};

\node (9) at (-4,0.5) []{0};
\node (10) at ( -2,0.5) [] {1};
\node (11) at ( 0,0.5) []{3};
\node (12) at ( 2,0.5) [] {4};
\node (13) at ( 4,0.5) []{5};
\node (14) at ( 6,0.5) [] {6};
\node (15) at ( 8,0.5) []{7};
\node (16) at ( 2,-2.5) [] {2};

\draw [decoration={markings,mark=at position 1 with
    {\arrow[scale=1.2,>=stealth]{>}}},postaction={decorate}] (-3.8,0.1)  --  (-2.2,0.1);
\draw [decoration={markings,mark=at position 1 with
    {\arrow[scale=1.2,>=stealth]{>}}},postaction={decorate}] (-2.2,-0.1) -- (-3.8,-0.1);
\draw [decoration={markings,mark=at position 1 with
    {\arrow[scale=1.2,>=stealth]{>}}},postaction={decorate}] (-1.8,0.1)  --  (-0.2,0.1);
\draw [decoration={markings,mark=at position 1 with
    {\arrow[scale=1.2,>=stealth]{>}}},postaction={decorate}] (-0.2,-0.1) -- (-1.8,-0.1) ;
\draw [decoration={markings,mark=at position 1 with
    {\arrow[scale=1.2,>=stealth]{>}}},postaction={decorate}] (0.2,0.1)  --  (1.8,0.1);
\draw [decoration={markings,mark=at position 1 with
    {\arrow[scale=1.2,>=stealth]{>}}},postaction={decorate}]  (1.8,-0.1) -- (0.2,-0.1);
\draw [decoration={markings,mark=at position 1 with
    {\arrow[scale=1.2,>=stealth]{>}}},postaction={decorate}] (2.2,0.1)  --  (3.8,0.1);
\draw [decoration={markings,mark=at position 1 with
    {\arrow[scale=1.2,>=stealth]{>}}},postaction={decorate}]  (3.8,-0.1) -- (2.2,-0.1);
\draw [decoration={markings,mark=at position 1 with
    {\arrow[scale=1.2,>=stealth]{>}}},postaction={decorate}] (4.2,0.1)  --  (5.8,0.1);
\draw [decoration={markings,mark=at position 1 with
    {\arrow[scale=1.2,>=stealth]{>}}},postaction={decorate}] (5.8,-0.1) -- (4.2,-0.1);
\draw [decoration={markings,mark=at position 1 with
    {\arrow[scale=1.2,>=stealth]{>}}},postaction={decorate}] (6.2,0.1)  --  (7.8,0.1);
\draw [decoration={markings,mark=at position 1 with
    {\arrow[scale=1.2,>=stealth]{>}}},postaction={decorate}] (7.8,-0.1) -- (6.2,-0.1);
\draw [decoration={markings,mark=at position 1 with
    {\arrow[scale=1.2,>=stealth]{>}}},postaction={decorate}] (2.1,-0.2)  --  (2.1,-1.8);
\draw [decoration={markings,mark=at position 1 with
    {\arrow[scale=1.2,>=stealth]{>}}},postaction={decorate}] (1.9,-1.8) -- (1.9,-0.2);
\end{tikzpicture}
\end{center}
with dimension vector $\underline{d}=\{1,2,2,3,4,3,2,1 \}$. 
\end{example}

The group $\prod_{i=0}^{r}\mathrm{GL}_{d_i}(\cc)$ acts on $M(\Gamma)$ by simultaneous conjugation and $\cc^*$ acts trivially. Therefore an action of $G(\Gamma)=(\prod_{i=0}^{r}\mathrm{GL}_{d_i}(\cc))/\cc^*$ on $M(\Gamma)$ is induced. 

Let $Q_1$ be the set of arrows of $Q$. For each pair of arrows between two vertices of $Q$, select one and then define a subset $Q_1^+$ of $Q_1$ composed of these selected arrows. An \textit{orientation} of $Q$ is a function $\epsilon:Q_1 \rightarrow \cc^*$ such that $\epsilon(a)=-\epsilon(\overline{a})=1$ for every arrow $a \in Q_1^+$. 

Fix an orientation of $Q$ and, for every pair $\varphi=(\varphi_a, a \in Q_1), \psi=(\psi_a, a \in Q_1)$ of elements of $M(\Gamma)$, set 
\begin{center}
$\langle \varphi,\psi  \rangle=\displaystyle \sum_{a \in Q_1} \epsilon(a) \mathrm{Tr}(\varphi_a \psi_{\overline{a}})$.
\end{center}
One can verify that $\langle.,.\rangle$ is a non-degenerate $G(\Gamma)$-invariant symplectic form on $M(\Gamma)$ and induces a moment map
\begin{center}
$\mu_{CS}:M(\Gamma) \rightarrow (\mathrm{Lie \ }G(\Gamma))^* \ \mathlarger{\mathlarger{\subset}} \ \displaystyle \bigoplus_{i=0}^{r}M_{d_i}(\cc)$
\end{center}
given by
\begin{center}
$\mu_{CS}(\varphi)= \displaystyle (\ldots ,\raisebox{-0.7\height}{\begin{tikzpicture}[scale=1.1,  transform shape]
\node (0) at ( 0,0) []{$\displaystyle \sum_{\begin{tikzpicture}[scale=0.35,  transform shape]
\tikzstyle{point}=[circle,draw,fill]
\tikzstyle{ligne}=[thick]

\node (1) at ( -2,0) [point]{};
\node (2) at ( 0,0) [point] {};
\node (3) at ( -2,0.6)  {{\huge j}};
\node (4) at ( 0,0.6)  {{\huge i}};
\node (5) at ( -1,0.6)  {{\huge a}};
\draw [ligne] (1)  --  (2);
\draw[decoration={markings, mark=at position 1 with {\arrow[scale=1.5,>=stealth]{>}}},postaction={decorate}] (-0.9,0)--(-0.7,0);
\end{tikzpicture}}\epsilon(a)\varphi_a \varphi_{\overline{a}} $};
\draw[decoration={brace,mirror,raise=14pt},decorate]
  (-1.2,0)  --  node[below=12pt] {\footnotesize $i^{\text{th}}$ entry} (1.2,0);
\end{tikzpicture}},\ldots )$.
\end{center} 
Here $\mathrm{Lie \ }G(\Gamma)$ is identified with its dual $(\mathrm{Lie \ }G(\Gamma))^*$ using the bilinear form defined from the trace. 

Let $Z$ be the dual of the center of $\mathrm{Lie \ }G(\Gamma)$. One can identify $Z$ with the subspace
\begin{center}
$\displaystyle\{ (\mu_0\mathrm{Id}_{d_0},\mu_1\mathrm{Id}_{d_1},\ldots ,\mu_r\mathrm{Id}_{d_r}) \in \prod_{i=0}^r M_{d_i}(\cc) \ | \ \mu_i \in \cc, \sum_{i=0}^{r}d_i \mu_i=0 \}$
\end{center}
of $\mathrm{Lie \ }G(\Gamma)$. As the moment map is $G(\Gamma)$-equivariant, the group $G(\Gamma)$ acts on the fibre $\mu_{CS}^{-1}(z)$ for any $z \in Z$, and so the quotient $\mu_{CS}^{-1}(z)//G(\Gamma)$ is well-defined. Finally, based on results by P.B. Kronheimer (\cite{Kron89}), H. Cassens and P. Slodowy proved that the map
 \begin{center}
\begin{tikzpicture}[scale=1,  transform shape]
\node (1) at ( 0,0) {$\mu_{CS}^{-1} (Z)// G(\Gamma)$};
\node (2) at ( 2.5,0) {$Z$};

\draw  [decoration={markings,mark=at position 1 with
    {\arrow[scale=1.2,>=stealth]{>}}},postaction={decorate}] (1.2,0)  --  (2.3,0);

\end{tikzpicture}
\end{center}
induced from the moment map is the pullback of the semiuniversal deformation of the simple singularity $\cc^2/\Gamma$. We hence have a construction of the pullback of the semiuniversal deformation of $\cc^2/\Gamma$ purely in terms of invariant theory:

 \begin{center}
\begin{tikzpicture}[scale=1,  transform shape]
\tikzstyle{point}=[circle,draw,fill]
\tikzstyle{ligne}=[thick]

\node (1) at ( 0,0) {$\mu_{CS}^{-1} (Z)//G(\Gamma)$};
\node (2) at ( 2.8,0) {$X \times_{\mathfrak{h}/W} \mathfrak{h}$};
\node (3) at ( 2.8,-2.5)  {$\mathfrak{h}$};
\node (4) at ( 0,-2.5)  {$Z$};

\draw  [decoration={markings,mark=at position 1 with
    {\arrow[scale=1.2,>=stealth]{>}}},postaction={decorate}] (1)  -- node[above] {$\cong$} (2);
\draw  [decoration={markings,mark=at position 1 with
    {\arrow[scale=1.2,>=stealth]{>}}},postaction={decorate}] (2)  --  (3);
\draw  [decoration={markings,mark=at position 1 with
    {\arrow[scale=1.2,>=stealth]{>}}},postaction={decorate}] (1)  --  (4);
\draw  [decoration={markings,mark=at position 1 with
    {\arrow[scale=1.2,>=stealth]{>}}},postaction={decorate}] (4)  -- node[below] {$\cong$}  (3);
\end{tikzpicture}
\end{center}
with $\mathfrak{h}$ a Cartan subalgebra of a Lie algebra of type $\Delta(\Gamma)$, and $W$ the associated Weyl group.

\begin{remark}
H. Cassens and P. Slodowy also used the representation space $M(\Gamma)$ to study resolutions of simple singularities. Indeed, by linearising the quotient $\mu_{CS}^{-1}(Z)//G(\Gamma)$ using a well-suited character 
$\chi:G(\Gamma)\rightarrow\cc^*$ (cf. \cite{CaSlo98} Section 7), they obtained the following diagram:
 \begin{center}
\begin{tikzpicture}[scale=1,  transform shape]
\tikzstyle{point}=[circle,draw,fill]
\tikzstyle{ligne}=[thick]

\node (1) at ( 0,0) {$\mu_{CS}^{-1} (Z)//^\chi G(\Gamma)$};
\node (2) at ( 3.5,0) {$\mu_{CS}^{-1} (Z)// G(\Gamma)$};
\node (3) at ( 1.75,-2)  {$Z$};

\draw  [decoration={markings,mark=at position 1 with
    {\arrow[scale=1.2,>=stealth]{>}}},postaction={decorate}] (1)  --  (2);
\draw  [decoration={markings,mark=at position 1 with
    {\arrow[scale=1.2,>=stealth]{>}}},postaction={decorate}] (2)  --  (3);
\draw  [decoration={markings,mark=at position 1 with
    {\arrow[scale=1.2,>=stealth]{>}}},postaction={decorate}] (1)  --  (3);
\end{tikzpicture}
\end{center}
which is a simultaneous minimal resolution. In particular, the space $\mu_{CS}^{-1} (0)//^\chi G(\Gamma)$ is a minimal resolution of the simple singularity $\cc^2/ \Gamma$.
\end{remark}

\section{Deformations of inhomogeneous simple singularities}\label{sec:InhomogeneousDeformations}

\subsection{Objectives}\label{sub:objectives}

Set $\Delta(\Gamma)$ a Dynkin diagram of type $A_{2r-1}$, $D_{r+1}$ or $E_6$ with $\Gamma$ the associated finite subgroup of $\mathrm{SU}_2$. The notations and results of Section~\ref{sec:SlodowyCassens} give the following diagram:

\begin{center}
\begin{tikzpicture}[scale=1,  transform shape,>=angle 60]
\tikzstyle{point}=[circle,draw,fill]
\tikzstyle{ligne}=[thick]

\node (1) at ( 0,0) {$X \times_{\mathfrak{h}/W}\mathfrak{h}$};
\node (2) at ( 2,0) {$X$};
\node (3) at ( 2,-2)  {$\mathfrak{h}/W$};
\node (4) at ( 0,-2)  {$\mathfrak{h}$};
\node (10) at ( -2.1,0)  {$\mu_{CS}^{-1} (Z)//G(\Gamma)= \ $};
\node (11) at ( -5,0)  {$\mu_{CS}^{-1} (Z)$};
\node (13) at ( -0.5,-2)  {$Z \cong$};

\node (9) at ( 1,-1) {$\circlearrowleft$};

\draw  [decoration={markings,mark=at position 1 with
    {\arrow[scale=1.2,>=stealth]{>}}},postaction={decorate}] (1)  -- node[above] {$\Psi$}  (2);
\draw  [decoration={markings,mark=at position 1 with
    {\arrow[scale=1.2,>=stealth]{>}}},postaction={decorate}] (2)  -- node[right] {$\alpha$} (3);
\draw  [decoration={markings,mark=at position 1 with
    {\arrow[scale=1.2,>=stealth]{>}}},postaction={decorate}] (1)  -- node[left] {$\widetilde{\alpha}$}  (4);
\draw  [decoration={markings,mark=at position 1 with
    {\arrow[scale=1.2,>=stealth]{>}}},postaction={decorate}] (4)  -- node[below] {$\pi$} (3);
\draw  [decoration={markings,mark=at position 1 with
    {\arrow[scale=1.2,>=stealth]{>}}},postaction={decorate}] (11)  --  (10);
\draw  [decoration={markings,mark=at position 1 with
    {\arrow[scale=1.2,>=stealth]{>}}},postaction={decorate}] (11)  --  (-3.5,0);    
\end{tikzpicture}
\end{center}
with $\alpha$ the semiuniversal deformation of the simple singularity $X_0=\cc^2/\Gamma$ of type $\Delta(\Gamma)$, and $\widetilde{\alpha}$ its pullback by the natural quotient map $\pi$. Let $\Gamma'$ be the finite subgroup of $\mathrm{SU}_2$ such that there exists an inhomogeneous simple singularity of type $\Delta(\Gamma,\Gamma')$ (cf. Definition~\ref{def:inhomogeneous}). Then $\Omega=\Gamma'/\Gamma$ acts on the simple singularity $X_0=\alpha^{-1}(0)$. Our aim is to define natural actions of $\Omega$ on $X$ and $\mathfrak{h}/W$ such that $\alpha$ becomes $\Omega$-equivariant, which would lead to the next theorem being a direct consequence of Theorem~\ref{thm:restrictionsud}.

\begin{theorem}\label{thm:defrestriction}
The restriction $\restr{\alpha}{\alpha^{-1}((\mathfrak{h}/W)^{\Omega})}$ over the fixed points $(\mathfrak{h}/W)^{\Omega}$ of the semiuniversal deformation of the simple singularity $\cc^2/\Gamma$ is a semiuniversal deformation of the inhomogeneous simple singularity of type $\Delta(\Gamma,\Gamma')$.
\end{theorem}

A natural way to accomplish this is to turn $\widetilde{\alpha}$ into an $\Omega$-equivariant morphism. By Lemma~\ref{momentequivariant} below, it is sufficient to analyse when the action of $\Omega$ on $M(\Gamma)$ is symplectic.

\begin{lemma}\label{momentequivariant}
Let $\mu:M \rightarrow \mathfrak{g}^*$ be a moment map on a symplectic manifold $(M,\omega)$ with an action of a semisimple Lie group $G$. Assume that a group $\Omega$ acts on $M$ by symplectomorphisms, and that $\Omega$ is a subgroup of the outer automorphism group of $G$. Furthermore, assume that the action of $G$ lifts to an action of $G \rtimes \Omega$. 
Then $\mu$ is $\Omega$-equivariant.
\end{lemma}

\begin{proof}
The action of $\Omega$ on $G$ induces an action on the dual $\mathfrak{g}^*$ of the Lie algebra $\mathfrak{g}$. If $\varpi \in \Omega$, let $\sigma_\varpi:M \rightarrow M$ and $\sigma_\varpi:\mathfrak{g}^* \rightarrow \mathfrak{g}^*$ denote the actions of $\varpi$ on $M$ and $\mathfrak{g}^*$, respectively. The $\mathfrak{g}^* \times \mathfrak{g}$ pairing is denoted by $\langle.,.\rangle$. By definition of the moment map, it is known that for any $x \in M, v \in T_xM$ and $\xi \in \mathfrak{g}$,
\begin{center}
$\langle d_x\mu(v),\xi \rangle = \omega_x(V_\xi(x),v)$,
\end{center}
with $V_\xi$ the vector field on $M$ induced by $\xi$. Hence 
\begin{center}
$\renewcommand{\arraystretch}{1.3} \begin{array}[t]{cl}
\langle d_{\sigma_\varpi(x)}\mu(d\sigma_\varpi(v)),\xi \rangle & =  \omega_{\sigma_\varpi(x)}(V_\xi(\sigma_\varpi(x)),d\sigma_\varpi(v)), \\
 & =  \omega_x((d\sigma_\varpi)^{-1}V_\xi(\sigma_\varpi(x)),v)  \text{ because }  \sigma_\varpi  \text{ preserves the symplectic form}.
 \end{array}$
 \end{center}
 But $\renewcommand{\arraystretch}{1.3} \begin{array}[t]{ccl}
 (d\sigma_\varpi)^{-1}V_\xi(\sigma_\varpi(x)) &=& (d\sigma_\varpi)^{-1} \left. \frac{d}{dt}  \right\vert_{t=0} \exp(t\xi).(\sigma_\varpi(x)) \\
 & = & \left.  \frac{d}{dt}  \right\vert_{t=0} \sigma_{\varpi^{-1}}(\exp(t\xi).(\sigma_\varpi(x))), \\
  & = & \left.  \frac{d}{dt}  \right\vert_{t=0} \exp(t\sigma_{\varpi^{-1}}(\xi)).(x)  \text{ because the action of $G$ lifts}, \\
   & = & V_{\sigma_{\varpi^{-1}}(\xi)}(x).
\end{array}$

\noindent So $\langle d_{\sigma_\varpi(x)}\mu(d\sigma_\varpi(v)),\xi \rangle$ $\renewcommand{\arraystretch}{1.3} \begin{array}[t]{ccl}
 & = & \omega_x(V_{\sigma_{\varpi^{-1}}(\xi)}(x),v), \\
 & = & \langle d_x\mu(v),\sigma_{\varpi^{-1}}(\xi) \rangle, \\
 & = & \langle \sigma_\varpi \circ d_x\mu(v), \xi \rangle \text{ by the definition of the action on the dual space}.
\end{array}$

\noindent Thus $d_{\sigma_\varpi(x)}\mu \circ d\sigma_\varpi=\sigma_\varpi \circ d_x\mu$. It follows that $\mu(\varpi.x)=\varpi.\mu(x)+f_\varpi$, with $f_\varpi$ an element of $\mathfrak{g}^*$ (we replaced $\sigma_\varpi$ by $\varpi.$ in order to lighten the notation). \\
Let us prove that $f_\varpi=0$. For any $g \in G$, we have 
\begin{center}$\renewcommand{\arraystretch}{1.3} \begin{array}[t]{rcl}
f_\varpi & = & \mu(\varpi.(g.x))-\varpi.(\mu(g.x)), \\
& =  & \mu(\varpi.(g.x))-\varpi.(g.\mu(x))  \text{ because }\mu \text{ is }G\text{-equivariant}, \\
& =  & \mu((\varpi.g).(\varpi.x))-\varpi.(g.\mu(x)) \text{ because the action of $G$ lifts}, \\ 
& =  & (\varpi.g).\mu(\varpi.x)-\varpi.(g.\mu(x)) \text{ because }\mu \text{ is }G\text{-equivariant}, \\
& =  & (\varpi.g).(\varpi.\mu(x)+f_\varpi)-\varpi.(g.\mu(x)), \\
& =  & (\varpi.g).(\varpi.\mu(x))+ (\varpi.g).f_\varpi-\varpi.(g.\mu(x)), \\
& =  & \varpi.(g.\mu(x))+ (\varpi.g).f_\varpi-\varpi.(g.\mu(x)), \\
& =  & (\varpi.g).f_\varpi.
 \end{array}$
 \end{center}
Hence for any $g \in G$, we have $g.f_\varpi=f_\varpi$, implying that $f_\varpi$ is a $G$-invariant element of $\mathfrak{g}^*$. But as $G$ is semisimple, a theorem by C. Chevalley states that the ring $S(\mathfrak{g}^*)^G$ is a polynomial ring generated by $\mathrm{rank}(G)$ homogeneous polynomials of degrees at least 2. As $f_\varpi$ is of degree 1, it follows that $f_\varpi=0$. Therefore $\mu(\varpi.x)=\varpi.\mu(x)$ for any $x \in M$, $\varpi \in \Omega$, and so $\mu$ is $\Omega$-equivariant. \hfill $\Box$
\end{proof}

We need to define actions of $\Omega$ on every object in the diagram at the beginning of this section such that:
\vspace{-1ex}
\begin{enumerate}\setlength\itemsep{0.3pt}
\item The restriction of the action on $X$ to the singularity $X_0$ coincides with the natural action computed in Subsection~\ref{sub:groupaction}.
\item The action on $M(\Gamma)$ is symplectic. 
\item The action on $G(\Gamma)$ is induced from the one on $M(\Gamma)$. 
\item The action on $\mathrm{Lie}(G(\Gamma))$ stabilises $Z$. Hence the isomorphism $\tau:Z \rightarrow \mathfrak{h}$ becomes $\Omega$-equivariant, with the action on $\mathfrak{h}$ coming from the action on the Dynkin diagram $\Delta(\Gamma)$.
\end{enumerate}

In Subsection~\ref{sub:compatibility}, we will compute the conditions required on the action of $\Omega$ so that it restricts to the natural action on the singularity $\cc^2/\Gamma$ seen in Subsection~\ref{sub:groupaction}. In order for $\Omega$ to act symplectically on $M(\Gamma)$, it has to preserve the symplectic form $\langle.,.\rangle$. However, by definition, this form depends on the orientation of the McKay quiver based on $\Delta(\Gamma)$. The orientation will then have to be chosen accordingly, which will be done in Subsection~\ref{sub:orientation}. The questions regarding the action of $\Omega$ on $G(\Gamma)$ as well as the $\Omega$-equivariance of $\tau: Z \rightarrow \mathfrak{h}$ will be handled in Subsection~\ref{sub:compatibilitysurgroupe}. \\ 
Once the semiuniversal deformation $\alpha$ becomes $\Omega$-equivariant, the restriction $\restr{\alpha}{\alpha^{-1}((\mathfrak{h}/W)^{\Omega})}$ turns into an $\Omega$-invariant morphism, and as such each of its fibre is acted upon by $\Omega$. In order to analyse the quotients of these fibres, an explicit description of the semiuniversal deformation is required. A system of coordinates on the base space $\mathfrak{h}/W$ of the semiuniversal deformation called flat coordinates will be introduced in Subsection~\ref{flatcoordinates}. This system of coordinates will help with the computations because it makes the action of $\Omega$ on $\mathfrak{h}/W$ linear.

\subsection{Action of \texorpdfstring{$\Gamma'/\Gamma$}{Lg} on \texorpdfstring{$\cc^2/\Gamma$}{Lg}}\label{sub:compatibility}

The action of $\Omega=\Gamma'/\Gamma$ on $M(\Gamma)$ induces an action on $\mu_{CS}^{-1}(Z)//G(\Gamma)$ and it is known that the simple singularity $\cc^2/\Gamma$ is in $\mu_{CS}^{-1}(Z)//G(\Gamma)$. However, it has been seen in Subsection~\ref{sub:groupaction} that $\Gamma'/\Gamma$ acts on the singularity in a natural way. We therefore want to impose that the action of $\Gamma'/\Gamma$ on $M(\Gamma)$ shall induce the right action on the singularity $\cc^2/\Gamma$.

 \subsubsection{ Computations for $\cc^2/\Gamma$ of type $A_{2r-1}$}\label{computationactiontypeA}
Set $\Gamma=\mathcal{C}_{2r}$. It follows that $\Gamma'=\mathcal{D}_r$ and $\Gamma'/\Gamma=\mathbb{Z}/2\mathbb{Z}=<\sigma>$. The McKay quiver when $\Delta(\Gamma)=A_{2r-1}$ is 

\begin{center}
 \begin{tikzpicture}[scale=0.8,  transform shape]
\tikzstyle{point}=[circle,draw,fill]
\tikzstyle{ligne}=[thick]
\tikzstyle{pointille}=[thick,dotted]

\node (1) at ( 5,0) [point]{};
\node (2) at ( 4,2) [point] {};
\node (3) at ( 2,2) [point]{} ;
\node (4) at ( -2,2) [point] {};
\node (5) at ( -4,2) [point] {};
\node (6) at ( -5,0) [point] {};
\node (7) at ( -4,-2) [point] {};
\node (8) at ( -2,-2) [point]{} ;
\node (9) at ( 2,-2) [point] {};
\node (10) at ( 4,-2) [point] {};

\node at ( -3,2.4)  {$b_{1}$};
\node at ( 3,2.4)  {$b_{r-2}$};
\node at ( -3,1.6)  {$a_{1}$};
\node at ( 3,1.6)  {$a_{r-2}$};

\node at ( -3,-2.4)  {$b_{2r-2}$};
\node at ( 3,-2.4)  {$b_{r+1}$};
\node at ( -3,-1.6)  {$a_{2r-2}$};
\node at ( 3,-1.6)  {$a_{r+1}$};

\node at ( -4.1,0.8)  {$a_{0}$};
\node at ( -4.8,1.3)  {$b_{0}$};
\node at ( -3.9,-0.8)  {$a_{2r-1}$};
\node at ( -5,-1.3)  {$b_{2r-1}$};

\node at ( 4,0.8)  {$a_{r-1}$};
\node at ( 5,1.3)  {$b_{r-1}$};
\node at ( 4.1,-0.8)  {$a_{r}$};
\node at ( 4.8,-1.3)  {$b_{r}$};

\node at (-5.5,0)  {\small{$0$}};
\node at (-4,2.4)  {\small{$1$}};
\node at (-2,2.4)  {\small{$2$}};
\node at ( 2,2.4)  {\small{$r-2$}};
\node at ( 4,2.4)  {\small{$r-1$}};
\node at (5.5,0)  {\small{$r$}};
\node at ( 4,-2.4)  {\small{$r+1$}};
\node at ( 2,-2.4)  {\small{$r+2$}};
\node at (-2,-2.4)  {\small{$2r-2$}};
\node at (-4,-2.4)  {\small{$2r-1$}};

\draw [pointille] (1.8,2.1)  --  (-1.8,2.1);
\draw [pointille] (1.8,1.9)  --  (-1.8,1.9);
\draw [pointille] (1.8,-1.9)  --  (-1.8,-1.9);
\draw [pointille] (1.8,-2.1)  --  (-1.8,-2.1);

\draw [decoration={markings,mark=at position 1 with
    {\arrow[scale=1.2,>=stealth]{>}}},postaction={decorate}] (3.8,2.1)  --  (2.2,2.1);
\draw [decoration={markings,mark=at position 1 with
    {\arrow[scale=1.2,>=stealth]{>}}},postaction={decorate}](2.2,1.9)  -- (3.8,1.9) ;
\draw [decoration={markings,mark=at position 1 with
    {\arrow[scale=1.2,>=stealth]{>}}},postaction={decorate}] (-2.2,2.1) --   (-3.8,2.1);
\draw [decoration={markings,mark=at position 1 with
    {\arrow[scale=1.2,>=stealth]{>}}},postaction={decorate}] (-3.8,1.9) -- (-2.2,1.9) ;

\draw [decoration={markings,mark=at position 1 with
    {\arrow[scale=1.2,>=stealth]{>}}},postaction={decorate}] (2.2,-2.1)  --  (3.8,-2.1);
\draw [decoration={markings,mark=at position 1 with
    {\arrow[scale=1.2,>=stealth]{>}}},postaction={decorate}](3.8,-1.9)  -- (2.2,-1.9) ;
\draw [decoration={markings,mark=at position 1 with
    {\arrow[scale=1.2,>=stealth]{>}}},postaction={decorate}] (-3.8,-2.1) --  (-2.2,-2.1);
\draw [decoration={markings,mark=at position 1 with
    {\arrow[scale=1.2,>=stealth]{>}}},postaction={decorate}](-2.2,-1.9)  -- (-3.8,-1.9) ;

\draw [decoration={markings,mark=at position 1 with
    {\arrow[scale=1.2,>=stealth]{>}}},postaction={decorate}] (4,1.8)  --   (4.8,0) ;
\draw [decoration={markings,mark=at position 1 with
    {\arrow[scale=1.2,>=stealth]{>}}},postaction={decorate}]  (5,0.2)  --  (4.2,2) ;
\draw [decoration={markings,mark=at position 1 with
    {\arrow[scale=1.2,>=stealth]{>}}},postaction={decorate}] (4.8,0)   --   (4,-1.8);
\draw [decoration={markings,mark=at position 1 with
    {\arrow[scale=1.2,>=stealth]{>}}},postaction={decorate}] (4.2,-2)   --  (5,-0.2) ;

\draw [decoration={markings,mark=at position 1 with
    {\arrow[scale=1.2,>=stealth]{>}}},postaction={decorate}]   (-4.8,0) -- (-4,1.8) ;
\draw [decoration={markings,mark=at position 1 with
    {\arrow[scale=1.2,>=stealth]{>}}},postaction={decorate}]   (-4.2,2) -- (-5,0.2)  ;
\draw [decoration={markings,mark=at position 1 with
    {\arrow[scale=1.2,>=stealth]{>}}},postaction={decorate}] (-4,-1.8) -- (-4.8,0);
\draw [decoration={markings,mark=at position 1 with
    {\arrow[scale=1.2,>=stealth]{>}}},postaction={decorate}]  (-5,-0.2) -- (-4.2,-2) ;

\end{tikzpicture}
\end{center}
and its dimension vector is $(1, \ldots, 1)$.  The action of $\Omega$ on the Dynkin diagram exchanges the vertices $i$ and $2r-i$ for $1 \leq i \leq 2r-1$. We extend this action to the extended Dynkin diagram of type $\widetilde{\Delta}(\Gamma)$ by making $\Omega$ fix the additional vertex labelled $0$. The action of $\Omega$ on $M(\Gamma)$ arises naturally from its action on the extended Dynkin diagram. Hence the action of $\Omega$ on $M(\Gamma)$ needs to verify  
\begin{center}
$\sigma.(a_0,\ldots ,a_{2r-1},b_0,\ldots, b_{2r-1})= (\lambda_{2r-1}b_{2r-1},\ldots,\lambda_0b_0,\delta_{2r-1}a_{2r-1},\ldots,\delta_0a_0)$.
\end{center}
with $\lambda_i,\delta_i \in \cc$. 

The action of $\mathbb{Z}/2\mathbb{Z}$ needs to be consistent with the action on the singularity. Recall that the singularity $\cc^2/\Gamma$ is defined by $\{Z^{2r}-XY=0\}$ with $Z=z_1z_2,X=z_1^{2r},Y=z_2^{2r}$ (cf. Subsection~\ref{sub:groupaction}), and 

\begin{center}
$\renewcommand{\arraystretch}{1.3} \raisebox{-.6\height}{$\left\lbrace \begin{array}[h]{l}
\sigma.X=(-1)^rY, \\
\sigma.Y=(-1)^rX, \\
\sigma.Z=-Z.
\end{array}\right.$}$ 
\end{center} 

After computation of the special fibre of the moment map, we find that $\cc^2/\Gamma \cong \{z^{2r}-xy=0 \}$ with $x=\prod_{i=0}^{2r-1}a_i$, $y=\prod_{i=0}^{2r-1}b_i$, and $z=\frac{1}{2r}\sum_{i=0}^{2r-1}a_i b_i$, which implies 
\begin{center}
$\renewcommand{\arraystretch}{1.5} \raisebox{-.6\height}{$\left\lbrace \begin{array}[h]{l}
\sigma.x= (\prod_{i=0}^{2r-1}\lambda_i) y, \\
\sigma.y=(\prod_{i=0}^{2r-1}\delta_i) x, \\
\sigma.z=\frac{1}{2r}\sum_{i=0}^{2r-1}\lambda_i\delta_i a_i b_i.
\end{array}\right.$}$ 
\end{center}
For the action to satisfy our conditions, it needs to verify $\renewcommand{\arraystretch}{1.5} \raisebox{-.4\height}{$\left\lbrace \begin{array}[h]{l}
\lambda_i\delta_i=-1, \quad 0 \leq i \leq 2r-1, \\
 \prod_{i=0}^{2r-1}\lambda_i=\prod_{i=0}^{2r-1}\delta_i=(-1)^r.
\end{array}\right.$}$

 \subsubsection{ Computations for $\cc^2/\Gamma$ of type $D_{r+1}$} 
By setting $\Gamma=\mathcal{D}_{r-1}$, one gets $\Gamma'=\mathcal{D}_{2(r-1)}$ from Table~\ref{definhomogeneous}, and so $\Gamma'/\Gamma=\mathbb{Z}/2\mathbb{Z}=<\sigma>$. The McKay quiver when $\Delta(\Gamma)=D_{r+1}$ is 

\begin{center}
\begin{tikzpicture}[scale=0.8,  transform shape,>=angle 60]
\tikzstyle{point}=[circle,draw,fill]

\node (0) at ( -5.4,1.4) [point] {};
\node (1) at ( -5.4,-1.4) [point] {};
\node (2) at ( -4,0) [point] {};
\node (3) at ( -2,0) [point] {};
\node (4) at ( 1,0) [point] {};
\node (5) at ( 3,0) [point] {};
\node (6) at ( 4.4,1.4) [point] {};
\node (7) at ( 4.4,-1.4) [point] {};

\node (8) at ( -5.4,1.9)  {\small{$0$}};
\node (9) at ( -5.4,-1.9)  {\small{$1$}};
\node (10) at ( -4,0.5)  {\small{$2$}};
\node (11) at ( -2,0.5)  {\small{$3$}};
\node (12) at ( 1,0.5)  {\small{$r-2$}};
\node (13) at ( 2.8,0.5) {\small{$r-1$}};
\node (14) at ( 4.4,1.9) {\small{$r$}};
\node (15) at ( 4.4,-1.9) {\small{$r+1$}};

\node (16) at ( -4.5,1.1)  {$\varphi_0^b$};
\node (17) at ( -5.1,0.5)  {$\varphi_0^a$};
\node (18) at ( -4.3,-0.9)  {$\varphi_1^a$};
\node (19) at ( -5,-0.4)  {$\varphi_1^b$};
\node (20) at ( -3,0.5)  {$\varphi_2^b$};
\node (21) at ( -3,-0.5) {$\varphi_2^a$};
\node (22) at ( 2,0.5) {$\varphi_{r-2}^b$};
\node (23) at ( 2,-0.5) {$\varphi_{r-2}^a$};
\node (24) at ( 3.6,1.2)  {$\varphi_r^a$};
\node (25) at ( 4.1,0.5) {$\varphi_r^b$};
\node (26) at ( 3.4,-0.9) {$\varphi_{r+1}^b$};
\node (27) at ( 4.1,-0.3) {$\varphi_{r+1}^a$};

\draw  [decoration={markings,mark=at position 1 with
    {\arrow[scale=1.2,>=stealth]{>}}},postaction={decorate}]  (-5.2,-1.4) -- (-4,-0.2);
\draw  [decoration={markings,mark=at position 1 with
    {\arrow[scale=1.2,>=stealth]{>}}},postaction={decorate}] (-4.2,0) --  (-5.4,-1.2);

\draw  [decoration={markings,mark=at position 1 with
    {\arrow[scale=1.2,>=stealth]{>}}},postaction={decorate}](-4,0.2) --  (-5.2,1.4) ;
\draw  [decoration={markings,mark=at position 1 with
    {\arrow[scale=1.2,>=stealth]{>}}},postaction={decorate}] (-5.4,1.2) --  (-4.2,0);

\draw [decoration={markings,mark=at position 1 with
    {\arrow[scale=1.2,>=stealth]{>}}},postaction={decorate}]  (-2.2,0.1) -- (-3.8,0.1);
\draw [decoration={markings,mark=at position 1 with
    {\arrow[scale=1.2,>=stealth]{>}}},postaction={decorate}] (-3.8,-0.1)  --  (-2.2,-0.1);
    
\draw [-,dashed]  (0.8,0.1) -- (-1.7,0.1);
\draw  [decoration={markings,mark=at position 1 with
    {\arrow[scale=1.2,>=stealth]{>}}},postaction={decorate}] (-1.7,0.1) -- (-1.75,0.1);
    
\draw [-,dashed] (-1.8,-0.1)  --  (0.8,-0.1);  
\draw  [decoration={markings,mark=at position 1 with
    {\arrow[scale=1.2,>=stealth]{>}}},postaction={decorate}] (0.8,-0.1) -- (0.85,-0.1);  
  
\draw [decoration={markings,mark=at position 1 with
    {\arrow[scale=1.2,>=stealth]{>}}},postaction={decorate}]  (2.8,0.1) -- (1.2,0.1);
\draw [decoration={markings,mark=at position 1 with
    {\arrow[scale=1.2,>=stealth]{>}}},postaction={decorate}]  (1.2,-0.1) -- (2.8,-0.1);
         
\draw  [decoration={markings,mark=at position 1 with
    {\arrow[scale=1.2,>=stealth]{>}}},postaction={decorate}] (4.2,1.4)  --  (3,0.2);
\draw  [decoration={markings,mark=at position 1 with
    {\arrow[scale=1.2,>=stealth]{>}}},postaction={decorate}]   (3.2,0) -- (4.4,1.2);   
    
\draw  [decoration={markings,mark=at position 1 with
    {\arrow[scale=1.2,>=stealth]{>}}},postaction={decorate}]  (3,-0.2) --  (4.2,-1.4);
\draw  [decoration={markings,mark=at position 1 with
    {\arrow[scale=1.2,>=stealth]{>}}},postaction={decorate}]  (4.4,-1.2) --  (3.2,0);           
\end{tikzpicture}
\end{center}
and the dimension vector is $(1,1,2, \ldots, 2,1,1)$. The action of $\Omega$ on the Dynkin diagram exchanges the vertices $r$ and $r+1$ and fixes the others. We extend this action to the extended Dynkin diagram of type $\widetilde{\Delta}(\Gamma)$ by making $\Omega$ fix the additional vertex labelled $0$. The action of $\Omega$ on $M(\Gamma)$ arises naturally from its action on the extended Dynkin diagram. Hence the action of $\Omega$ on $M(\Gamma)$ needs to verify    
\begin{center}
\begin{tabular}{l}
$\sigma.(\varphi_0^a,\varphi_0^b,\ldots ,\varphi_{r-1}^a,\varphi_{r-1}^b,\varphi_{r}^a,\varphi_{r}^b,\varphi_{r+1}^a,\varphi_{r+1}^b)= $\\
$(\varphi_0^a,\varphi_0^b,\ldots ,\varphi_{r-1}^a,\varphi_{r-1}^b,\lambda_{r+1}\varphi_{r+1}^a,\delta_{r+1}\varphi_{r+1}^b,\lambda_{r}\varphi_{r}^a,\delta_{r}\varphi_{r}^b)$.
\end{tabular}
\end{center}
with $\lambda_i,\delta_i \in \cc$. 

Assume $r$ to be odd. The action of $\sigma$ has to induce the correct action on the singularity $\cc^2/\Gamma$. It has been seen in Subsection~\ref{sub:groupaction} that $\cc^2/\Gamma$ is defined by $\{X(Y^2-X^{r-1}) + Z^2=0\}$ with $X=4^{\frac{1}{r}}(z_1z_2)^2$, $Y=4^{-\frac{1}{2r}}(z_1^{2(r-1)} + z_2^{2(r-1)})$, $Z=iz_1z_2(z_1^{2(r-1)}-z_2^{2(r-1)})$. Furthermore, $\Gamma'$ is generated by $g=\begin{pmatrix} \xi & 0 \\ 0 & \xi^{-1} \end{pmatrix}$ with $\xi=\exp(\frac{2i\pi}{4(r-1)})$ and $h=\begin{pmatrix} 0 & i \\ i & 0 \end{pmatrix}$. It follows that $h$ fixes $X$, $Y$, $Z$, and 
\vspace{-\topsep}
\begin{center}
$\renewcommand{\arraystretch}{1.3} \left\{
    \begin{array}{ccccc}
     g.X & = & X, \\
     g.Y & = & -Y, \\
     g.Z & = &-Z.
    \end{array}
\right.$ 
\end{center}

After computation of the special fibre of the moment map, we obtain $\cc^2/\Gamma \cong \{x^r + y^2x + z^2=0 \}$ with 
\begin{center}
\begin{tabular}{l}
$\left\lbrace \renewcommand{\arraystretch}{1.3}\begin{array}[h]{l}
x=4^{-\frac{1}{r}}p_{r, r+1}, \\
y=i4^{\frac{1}{2r}}(q_{0,\ldots, r}-\frac{1}{2}p_{r, r+1}^{\frac{r-1}{2}}), \\
z=s_{0,\ldots, r, r+1},
\end{array}\right.$ such that \\[6ex]
$\left\lbrace \renewcommand{\arraystretch}{1.3} \begin{array}[h]{l}
p_{r, r+1}=\mathrm{Tr}(\varphi_r^b\varphi_{r+1}^a\varphi_{r+1}^b\varphi_r^a), \\
q_{0,\ldots, r}=\mathrm{Tr}(\varphi_0^b \ldots  \varphi_{r-2}^b \varphi_r^a \varphi_r^b  \varphi_{r-2}^a \ldots \varphi_0^a), \\
s_{0,\ldots, r, r+1}=\mathrm{Tr}(\varphi_0^b \ldots  \varphi_{r-2}^b \varphi_{r+1}^a\varphi_{r+1}^b \varphi_r^a \varphi_r^b  \varphi_{r-2}^a \ldots \varphi_0^a),
\end{array}\right.$
\end{tabular}
\end{center}
 which implies 
\begin{center}
$\raisebox{-.6\height}{$\left\lbrace \renewcommand{\arraystretch}{1.3} \begin{array}[h]{l}
\sigma.x= \lambda_r\delta_r\lambda_{r+1}\delta_{r+1}x, \\
\sigma.y=i4^{\frac{1}{2r}}(-\lambda_{r+1}\delta_{r+1}q_{0,\ldots, r} + (\lambda_{r+1}\delta_{r+1}-\frac{1}{2}\lambda_{r}\delta_{r}\lambda_{r+1}\delta_{r+1})p_{r, r+1}^{\frac{r-1}{2}}), \\
\sigma.z=-\lambda_r\delta_r\lambda_{r+1}\delta_{r+1}z.
\end{array}\right.$}$ 
\end{center}
\medskip
In order for the action to satisfy our conditions, it is required that $\lambda_{r}\delta_{r}=\lambda_{r+1}\delta_{r+1}=1$. 

The conditions when $r$ is even are obtained in a similar way and are also $\lambda_{r}\delta_{r}=\lambda_{r+1}\delta_{r+1}=1$.

 \subsubsection{Computations for $\cc^2/\Gamma$ of type $E_{6}$}
 The McKay quiver when $\Delta(\Gamma)=E_{6}$ is 
\begin{center}
\begin{tikzpicture}[scale=0.68, transform shape,>=angle 60]
\tikzstyle{point}=[circle,draw,fill]
\tikzstyle{ligne}=[thick]

\node (1) at ( -2,0) [point] {};
\node (2) at ( 0,0) [point]{};
\node (3) at ( 2,0) [point] {};
\node (4) at ( 4,0) [point]{};
\node (5) at ( 6,0) [point] {};
\node (6) at ( 2,-2) [point] {};
\node (7) at ( 2,-4) [point] {};

\node (8) at ( -2,0.5) [] {0};
\node (9) at ( 0,0.5) []{3};
\node (10) at ( 2,0.5) [] {6};
\node (11) at ( 4,0.5) []{5};
\node (12) at ( 6,0.5) [] {2};
\node (13) at ( 2.5,-2) [] {4};
\node (14) at ( 2.5,-4) [] {1};

\node (15) at ( -1,0.5) [] {$\varphi_0^a$};
\node (16) at ( -1,-0.5) [] {$\varphi_0^b$};

\node (17) at ( 1,0.5) [] {$\varphi_3^a$};
\node (18) at ( 1,-0.5) [] {$\varphi_3^b$};

\node (19) at ( 3,0.5) [] {$\varphi_5^a$};
\node (20) at ( 3,-0.5) [] {$\varphi_5^b$};

\node (21) at ( 5,0.5) [] {$\varphi_2^a$};
\node (22) at ( 5,-0.5) [] {$\varphi_2^b$};

\node (23) at ( 2.5,-1) [] {$\varphi_4^a$};
\node (24) at ( 1.5,-1) [] {$\varphi_4^b$};

\node (25) at ( 2.5,-3) [] {$\varphi_1^a$};
\node (26) at ( 1.5,-3) [] {$\varphi_1^b$};

\draw  [decoration={markings,mark=at position 1 with
    {\arrow[scale=1.2,>=stealth]{>}}},postaction={decorate}]  (-1.8,0.1)  --  (-0.2,0.1);
\draw  [decoration={markings,mark=at position 1 with
    {\arrow[scale=1.2,>=stealth]{>}}},postaction={decorate}]  (-0.2,-0.1) -- (-1.8,-0.1);
\draw  [decoration={markings,mark=at position 1 with
    {\arrow[scale=1.2,>=stealth]{>}}},postaction={decorate}]  (0.2,0.1)  --  (1.8,0.1);
\draw  [decoration={markings,mark=at position 1 with
    {\arrow[scale=1.2,>=stealth]{>}}},postaction={decorate}]   (1.8,-0.1) -- (0.2,-0.1);
\draw  [decoration={markings,mark=at position 1 with
    {\arrow[scale=1.2,>=stealth]{>}}},postaction={decorate}]   (3.8,0.1) -- (2.2,0.1);
\draw  [decoration={markings,mark=at position 1 with 
    {\arrow[scale=1.2,>=stealth]{>}}},postaction={decorate}]   (2.2,-0.1) -- (3.8,-0.1) ;
\draw  [decoration={markings,mark=at position 1 with
    {\arrow[scale=1.2,>=stealth]{>}}},postaction={decorate}]   (5.8,0.1) -- (4.2,0.1) ;
\draw  [decoration={markings,mark=at position 1 with
    {\arrow[scale=1.2,>=stealth]{>}}},postaction={decorate}]   (4.2,-0.1) -- (5.8,-0.1) ;
\draw  [decoration={markings,mark=at position 1 with
    {\arrow[scale=1.2,>=stealth]{>}}},postaction={decorate}]   (2.1,-1.8) -- (2.1,-0.2) ;
\draw  [decoration={markings,mark=at position 1 with
    {\arrow[scale=1.2,>=stealth]{>}}},postaction={decorate}]   (1.9,-0.2) -- (1.9,-1.8);
\draw  [decoration={markings,mark=at position 1 with
    {\arrow[scale=1.2,>=stealth]{>}}},postaction={decorate}]    (2.1,-3.8) -- (2.1,-2.2);
\draw  [decoration={markings,mark=at position 1 with
    {\arrow[scale=1.2,>=stealth]{>}}},postaction={decorate}]   (1.9,-2.2) -- (1.9,-3.8);
\end{tikzpicture}
\end{center}
and the dimension vector is $(1,1,1,2,2,2,3)$. Because $\Gamma=\mathcal{T}$, it follows that $\Gamma'=\mathcal{O}$ and $\Gamma'/\Gamma=\mathbb{Z}/2\mathbb{Z}=<\sigma>$. The action of $\Omega$ on the Dynkin diagram permutes the vertices $1 \leftrightarrow 2$, $4 \leftrightarrow 5$ and fixes the others. We extend this action to the extended Dynkin diagram of type $\widetilde{\Delta}(\Gamma)$ by making $\Omega$ fix the additional vertex labelled $0$. The action of $\Omega$ on $M(\Gamma)$ arises naturally from its action on the extended Dynkin diagram. Hence the action of $\Omega$ on $M(\Gamma)$ needs to verify 
\begin{center}
\begin{tabular}{l}
$\sigma.(\varphi_0^a,\varphi_0^b,\varphi_1^a,\varphi_1^b,\varphi_2^a,\varphi_2^b,\varphi_3^a,\varphi_3^b,\varphi_4^a,\varphi_4^b,\varphi_5^a,\varphi_5^b) = $ \\
$(\varphi_0^a,\varphi_0^b,\lambda_2\varphi_2^a,\delta_2\varphi_2^b,\lambda_1\varphi_1^a,\delta_1\varphi_1^b,\varphi_3^a,\varphi_3^b,\lambda_5\varphi_5^a,\delta_5\varphi_5^b,\lambda_4\varphi_4^a,\delta_4\varphi_4^b)$
\end{tabular}
\end{center}
with $\lambda_i,\delta_i \in \cc$.

The action of $\mathbb{Z}/2\mathbb{Z}$ needs to be consistent with the action on the singularity computed with matrices. It is known that the singularity $\cc^2/\Gamma$ is defined by $\{X^4 + Y^3 + Z^2=0\}$ where $X=108^{\frac{1}{4}}z_1z_2(z_1^4-z_2^4)$, $Y=\exp(\frac{i\pi}{3})(z_1^8 + z_2^8 + 14(z_1z_2)^4)$, $Z=(z_1^4 + z_2^4)^3-36(z_1z_2)^4(z_1^4 + z_2^4)$ (cf. Subsection~\ref{sub:groupaction}), and
\begin{center}
$\left\{ \renewcommand{\arraystretch}{1.3}
   \begin{array}{ccccc}
     \sigma.X & = & -X, \\
     \sigma.Y & = & Y, \\
     \sigma.Z & = & -Z.
   \end{array}
\right.$ 
\end{center}

For $i_1,...,i_k \in \mathbb{Z}_{ \geq 0}$, set $\Phi(i_1 i_2 \ldots i_k) =\mathrm{Tr}(\varphi_{i_1}^a\varphi_{i_1}^b\varphi_{i_2}^a\varphi_{i_2}^b \ldots \varphi_{i_k}^a\varphi_{i_k}^b)$. After computation of the special fibre of the moment map, one finds that $\cc^2/\Gamma \cong \{x^4 + y^3 + z^2=0 \}$ with $x=\frac{\exp(\frac{i\pi}{4})}{\sqrt{2}}\Phi(4^25)$, $y=\Phi(4^25^2)$ and $z=\Phi(3^24^25^2)  + \frac{1}{2}\Phi(4^25)^2$, which implies 
\begin{center}
$\renewcommand{\arraystretch}{1.3} \raisebox{-.6\height}{$\left\lbrace \begin{array}[h]{l}
\sigma.x=-(\lambda_5 \delta_5)^2\lambda_4 \delta_4x , \\
\sigma.y=(\lambda_4 \delta_4\lambda_5 \delta_5)^2y, \\
\sigma.z=-(\lambda_4 \delta_4\lambda_5 \delta_5)^2\Phi(3^24^25^2) + (-(\lambda_4 \delta_4\lambda_5 \delta_5)^2 + \frac{1}{2}(\lambda_4 \delta_4(\lambda_5 \delta_5)^2)^2)\Phi(4^25)^2.
\end{array}\right.$}$ 
\end{center}
For the action to satisfy our conditions, it needs to verify $\lambda_4\delta_4=1$ and $\lambda_5 \delta_5=\pm1$.

 \subsubsection{ Computations for $\cc^2/\Gamma$ of type $D_{4}$ and $\Gamma'/\Gamma=\mathfrak{S}_3$} 
The McKay quiver based on a Dynkin diagram of type $D_{4}$ is

 \begin{center}
 \begin{tikzpicture}[scale=0.8,  transform shape]
\tikzstyle{point}=[circle,draw,fill]
\tikzstyle{ligne}=[thick]
\tikzstyle{pointille}=[thick,dotted]

\node (1) at ( 0,0) [point]{};
\node (2) at ( 2,2) [point] {};
\node (3) at ( -2,2) [point]{} ;
\node (4) at ( 2,-2) [point] {};
\node (5) at ( -2,-2) [point] {};

\node at ( -2.5,2)  {$0$};
\node at ( -2.5,-2)  {$1$};
\node at ( 2.5,2)  {$3$};
\node at ( 2.5,-2)  {$4$};
\node at ( 0,0.5)  {$2$};

\node at ( -0.8,1.3)  {$\varphi_0^b$};
\node at ( -0.8,-1.3)  {$\varphi_1^a$};
\node at ( 1.3,0.8)  {$\varphi_3^b$};
\node at ( 1.4,-0.7)  {$\varphi_4^a$};

\node at ( 0.8,1.3)  {$\varphi_3^a$};
\node at ( 0.8,-1.3)  {$\varphi_4^b$};
\node at ( -1.3,0.7)  {$\varphi_0^a$};
\node at ( -1.4,-0.7)  {$\varphi_1^b$};

\draw [decoration={markings,mark=at position 1 with
    {\arrow[scale=1.2,>=stealth]{>}}},postaction={decorate}] (-0.2,0.3)  --  (-1.8,1.9);
\draw [decoration={markings,mark=at position 1 with
    {\arrow[scale=1.2,>=stealth]{>}}},postaction={decorate}]  (-1.9,1.7)  --  (-0.3,0.1);

\draw [decoration={markings,mark=at position 1 with
    {\arrow[scale=1.2,>=stealth]{>}}},postaction={decorate}]   (1.8,1.9)  --  (0.2,0.3);
\draw [decoration={markings,mark=at position 1 with
    {\arrow[scale=1.2,>=stealth]{>}}},postaction={decorate}]   (0.3,0.1)  --  (1.9,1.7) ;

\draw [decoration={markings,mark=at position 1 with
    {\arrow[scale=1.2,>=stealth]{>}}},postaction={decorate}]  (-1.8,-1.9)  --  (-0.2,-0.3) ;
\draw [decoration={markings,mark=at position 1 with
    {\arrow[scale=1.2,>=stealth]{>}}},postaction={decorate}]  (-0.3,-0.1)  --  (-1.9,-1.7) ;

\draw [decoration={markings,mark=at position 1 with
    {\arrow[scale=1.2,>=stealth]{>}}},postaction={decorate}] (0.2,-0.3)  --  (1.8,-1.9);
\draw [decoration={markings,mark=at position 1 with
    {\arrow[scale=1.2,>=stealth]{>}}},postaction={decorate}]  (1.9,-1.7)  --  (0.3,-0.1);

\end{tikzpicture}
\end{center}
and the dimension vector is $(1,1,2,1,1)$. According to Table~\ref{definhomogeneous}, we have $\Gamma=\mathcal{D}_2=\langle \begin{pmatrix}
i & 0  \\
0 & -i
\end{pmatrix},\begin{pmatrix}
0 & i \\
i & 0
\end{pmatrix} \rangle$ and  $\Gamma'=\mathcal{O}=\langle a= \begin{pmatrix}
\epsilon & 0  \\
0 & \epsilon^{-1}
\end{pmatrix},b=\begin{pmatrix}
0 & i \\
i & 0
\end{pmatrix},c=\frac{1}{\sqrt{2}}\begin{pmatrix}
\epsilon^{-1} & \epsilon^{-1} \\
-\epsilon & \epsilon
\end{pmatrix}  \rangle$ with $\epsilon=\exp(\frac{i\pi}{4})$. Hence $\Gamma'/\Gamma$ is the symmetric group $\mathfrak{S}_3$ generated by $\overline{c}$ and $\overline{a}$, with $\overline{c}^3=1$ and $\overline{a}^2=1$. The symbol $\ \bar{} \ $ means the class modulo $\mathcal{D}_2$. 

The action of $\Omega$ on the Dynkin diagram permutes the vertices with $\overline{c}$ acting as $(143)$ and $\overline{a}$ as $(34)$. We extend this action to the extended Dynkin diagram of type $\widetilde{\Delta}(\Gamma)$ by making $\Omega$ fix the additional vertex labelled $0$. The action of $\Omega$ on $M(\Gamma)$ arises naturally from its action on the extended Dynkin diagram. Hence the action of $\Omega$ on $M(\Gamma)$ needs to verify 

\begin{center}
$ \renewcommand{\arraystretch}{1.3} \left\{ \begin{array}[h]{l}
\overline{c}.(\varphi_0^a,\varphi_0^b,\varphi_1^a,\varphi_1^b,\varphi_3^a,\varphi_3^b,\varphi_4^a,\varphi_4^b)= (\varphi_0^a,\varphi_0^b,\lambda_3\varphi_3^a,\delta_3\varphi_3^b,\lambda_4\varphi_4^a,\delta_4\varphi_4^b,\lambda_1\varphi_1^a,\delta_1\varphi_1^b), \\
\overline{a}.(\varphi_0^a,\varphi_0^b,\varphi_1^a,\varphi_1^b,\varphi_3^a,\varphi_3^b,\varphi_4^a,\varphi_4^b)= (\varphi_0^a,\varphi_0^b,\varphi_1^a,\varphi_1^b,\alpha_4\varphi_4^a,\beta_4\varphi_4^b,\alpha_3\varphi_3^a,\beta_3\varphi_3^b),
\end{array} \right.$
\end{center}
with $\lambda_i,\delta_i,\alpha_i,\beta_i \in \cc$. 

In Subsection~\ref{sub:groupaction}, we saw that $\cc^2/\Gamma \cong \{X^3 + Y^2X + Z^2=0 \}$ with $X=(z_1z_2)^2$, $Y=\frac{i}{2}(z_1^4 + z_2^4)$ and $Z=\frac{-1}{2}z_1z_2(z_1^4-z_2^4)$. Computations then lead to 
\begin{center}
$ \renewcommand{\arraystretch}{1.3} \left\{
    \begin{array}{ccccc}
     \overline{c}.X & = & \frac{-X + iY}{2}, \\
     \overline{c}.Y & = & \frac{3iX-Y}{2}, \\
     \overline{c}.Z & = & Z,
    \end{array}
\right.$ \hspace{1cm} and \hspace{1cm} $ \renewcommand{\arraystretch}{1.3} \left\{
    \begin{array}{ccccc}
     \overline{a}.X & = & X, \\
     \overline{a}.Y & = & -Y, \\
     \overline{a}.Z & = & -Z.
    \end{array}
\right.$
\end{center}

After computation of the special fibre of the moment map, we obtain $\cc^2/\Gamma \cong \{x^3 + y^2x + z^2=0 \}$ with 

\begin{center}
$ \renewcommand{\arraystretch}{1.3} \left\{ \begin{array}{l}
x=2^{-\frac{2}{3}}\exp(\frac{2i\pi}{3})p_{34}, \\
y=2^{\frac{1}{3}}\exp(\frac{i\pi}{6})(p_{03} + \frac{1}{2}p_{34}), \\
z=q_{034},
\end{array}\right.$ such that $ \renewcommand{\arraystretch}{1.3} \left\{ \begin{array}{l}
p_{03}=\mathrm{Tr}(\varphi_0^b\varphi_3^a\varphi_3^b\varphi_0^a), \\
p_{34}=\mathrm{Tr}(\varphi_3^b\varphi_4^a\varphi_4^b\varphi_3^a), \\
q_{034}=\mathrm{Tr}(\varphi_0^b\varphi_4^a\varphi_4^b\varphi_3^a\varphi_3^b\varphi_0^a),
\end{array}\right.$ 
\end{center}
which implies 
\begin{center}
\begin{tabular}[h]{l}
$ \left\{ \begin{array}{l}
 \overline{c}.x=\lambda_1\delta_1\lambda_4\delta_4\frac{-x + i y}{3}, \\
 \overline{c}.y= 2^{\frac{1}{3}}\exp(\frac{i\pi}{6})(-\lambda_4\delta_4p_{34} + \lambda_4\delta_4(\frac{\lambda_1\delta_1}{2}-1)p_{03}), \\
 \overline{c}.z=z,
\end{array}\right.$  \\[6ex]
  $ \left\{ \begin{array}[h]{l}
\overline{a}.x= \alpha_3\beta_3\alpha_4\beta_4x, \\
\overline{a}.y= 2^{\frac{1}{3}}\exp(\frac{i\pi}{6})(p_{03}(-\alpha_4\beta_4) + p_{34}(\frac{1}{2}\alpha_3\beta_3\alpha_4\beta_4-\alpha_4\beta_4)), \\
\overline{a}.z=-\alpha_3\beta_3\alpha_4\beta_4z.
\end{array}\right.$
\end{tabular}
\end{center}

In order for the action of $\Gamma'/\Gamma$ on $M(\Gamma)$ to induce the natural action of $\mathfrak{S}_3$, it requires 
\begin{center}
$ \renewcommand{\arraystretch}{1.3} \raisebox{-.4\height}{$\left\lbrace \begin{array}[h]{l}
\lambda_1\delta_1=\lambda_3\delta_3=\lambda_4\delta_4=1, \\
\alpha_3\beta_3=\alpha_4\beta_4=-1.
\end{array}\right.$}$ 
\end{center}

\begin{remark}
It is possible to check that if $\Gamma'/\Gamma=\mathbb{Z}/3\mathbb{Z}$, the results are similar.
\end{remark}

 \subsubsection{Conclusions}\label{subsub:conclusions}
 
In the last four subsections, conditions have been found in order to ensure that the restriction of the action of $\Omega$ on $M(\Gamma)$ to the simple singularity $\cc^2/\Gamma$ corresponds to the natural action given in Subsection~\ref{sub:groupaction}. The results are summarised in the following table:
 \begin{center}
{\renewcommand{\arraystretch}{1.3}
 \begin{tabular}{|c|c|}
\hline
  Type of singularity   & Conditions on the action \\
\hhline{|=|=|}
      \multirow{2}{*}{$B_r=(A_{2r-1},\mathbb{Z}/2\mathbb{Z})$} & $\renewcommand{\arraystretch}{1.5} \raisebox{-.4\height}{$\left\lbrace \begin{array}[h]{l}
\lambda_i\delta_i=-1, \quad 0 \leq i \leq 2r-1, \\
 \prod_{i=0}^{2r-1}\lambda_i=\prod_{i=0}^{2r-1}\delta_i=(-1)^r.
\end{array}\right.$}$ 
 \\
\hline
    $C_r=(D_{r+1},\mathbb{Z}/2\mathbb{Z})$ & $\lambda_{r}\delta_{r}=\lambda_{r+1}\delta_{r+1}=1$ \\
\hline
    $F_4=(E_{6},\mathbb{Z}/2\mathbb{Z})$ & $\lambda_4\delta_4=1$ and $\lambda_5 \delta_5=\pm1$ \\
\hline
    \multirow{2}{*}{$G_2=(D_{4},\mathfrak{S}_3)$} & $ \renewcommand{\arraystretch}{1.5} \raisebox{-.4\height}{$\left\lbrace \begin{array}[h]{l}
\lambda_1\delta_1=\lambda_3\delta_3=\lambda_4\delta_4=1, \\
\alpha_3\beta_3=\alpha_4\beta_4=-1.
\end{array}\right.$}$ 
 \\
\hline
    \end{tabular}}
  \captionof{table}{Conditions on the action of $\Omega$ on $M(\Gamma)$ for the special fibre}
  \label{ConditionFibreSpeciale}
\end{center}

\subsection{Choice of an orientation}\label{sub:orientation}

Let $M(\Gamma)=\bigoplus_{a \in Q_1}\mathrm{Hom}(V_{s(a)},V_{t(a)})$ be the representation space of the McKay quiver $Q$, and $Q_1^+$ a set of positive arrows of $Q$ as defined in Section~\ref{sec:SlodowyCassens}. One can then write 
\begin{center}
\renewcommand{\arraystretch}{1.3}
$\begin{array}[h]{cl}M(\Gamma) &= \bigoplus_{a \in Q_1^ + }(\mathrm{Hom}(V_{s(a)},V_{t(a)}) \oplus \mathrm{Hom}(V_{s(\bar{a})},V_{t(\bar{a})})), \\
  & = \bigoplus_{a \in Q_1^ + }\mathrm{Hom}(V_{s(a)},V_{t(a)}) \oplus \bigoplus_{a \in Q_1^ + }\mathrm{Hom}(V_{s(a)},V_{t(a)})^*, \\
  & =\bigoplus_{a \in Q_1^ + }\mathrm{Hom}(V_{s(a)},V_{t(a)}) \oplus (\bigoplus_{a \in Q_1^ + }\mathrm{Hom}(V_{s(a)},V_{t(a)}))^*.
  \end{array}$ 
  \end{center}
Therefore $M(\Gamma)$ can be seen as the cotangent space on the vector space $\bigoplus_{a \in Q_1^ + }\mathrm{Hom}(V_{s(a)},V_{t(a)})$. The choice of $Q_1^ + $, which is equivalent to the choice of an orientation of $Q$, determines the space on which the cotangent fibre is constructed. Conversely, the choice of the base space of the cotangent fibre will define a subset $Q_1^ + $ of $Q_1$, and so an orientation of $Q$. 

By direct computations on the representation space $M(\Gamma)$, we prove the following theorem:

\begin{theorem}\label{thm:orientation}
The action of $\Omega=\Gamma'/\Gamma$ on $M(\Gamma)$ is symplectic and induces the natural action on the simple singularity when:
\vspace{-1ex}
\begin{itemize}\setlength\itemsep{0.3pt}
\item for $(A_{2r-1},\mathbb{Z}/2\mathbb{Z})$, $\Omega$ reverses the orientation of the McKay quiver.
\item for the other cases, $\Omega$ preserves the orientation of the McKay quiver.
\end{itemize}

\end{theorem}

By using the same notations as in the previous sections, one can find conditions on the orientation in order to ensure that the action of $\Omega$ on $M(\Gamma)$ is symplectic. The results are summarised in the following table:
 \begin{center}
{\renewcommand{\arraystretch}{1.3}
 \begin{tabular}{|c|c|}
\hline
  Type of singularity   & Conditions on the orientation \\
\hhline{|=|=|}
      $B_r=(A_{2r-1},\mathbb{Z}/2\mathbb{Z})$ & $\epsilon(a_i)\epsilon(b_{2r-1-i})=\lambda_i \delta_i$, $0 \leq i \leq 2r-1$  \\
\hline
    $C_r=(D_{r+1},\mathbb{Z}/2\mathbb{Z})$ & $\epsilon(\varphi_r^a)\epsilon(\varphi_{r+1}^a)=\lambda_{r}\delta_{r}=\lambda_{r+1}\delta_{r+1}$ \\
\hline
   \multirow{2}{*}{ $F_4=(E_{6},\mathbb{Z}/2\mathbb{Z})$} & $\renewcommand{\arraystretch}{1.5} \raisebox{-.4\height}{$\left\lbrace \begin{array}[h]{l}
\epsilon(\varphi_1^a)\epsilon(\varphi_{2}^a)=\lambda_1\delta_1=\lambda_2\delta_2, \\
\epsilon(\varphi_4^a)\epsilon(\varphi_{5}^a)=\lambda_4\delta_4=\lambda_5\delta_5.
\end{array}\right.$}$ 
 \\
\hline
    \multirow{3}{*}{$G_2=(D_{4},\mathfrak{S}_3)$} & $ \renewcommand{\arraystretch}{1.5} \raisebox{-.6\height}{$\left\lbrace \begin{array}[h]{l}
\lambda_1\delta_1=\epsilon(\varphi_1^a)\epsilon(\varphi_{4}^a), \\
\lambda_3\delta_3=\epsilon(\varphi_1^a)\epsilon(\varphi_{3}^a), \\
\lambda_4\delta_4=\epsilon(\varphi_3^a)\epsilon(\varphi_{4}^a).
\end{array}\right.$}$ 
 \\
\hline
    \end{tabular}}
  \captionof{table}{Conditions on the orientation for the action of $\Omega$ on $M(\Gamma)$}
  \label{ConditionOrientation}
\end{center}
 
One can see that it is possible to set constants $\lambda_i$ and $\delta_i$ as well as an orientation such that the conditions in Table~\ref{ConditionFibreSpeciale} and Table~\ref{ConditionOrientation} are satisfied. It follows that the action of $\Omega$ on $M(\Gamma)$ becomes symplectic and induces the natural action on the singularity $\cc^2/\Gamma$. It follows that we proved the following result:

\begin{theorem}\label{thm:compatibility}
For any McKay quiver built on a Dynkin diagram of type $A_{2r-1}$, $D_{r+1}$ or $E_6$, there exists an action of $\Omega=\Gamma'/\Gamma$ on $M(\Gamma)$ that is symplectic and induces the natural action on the simple singularity $\cc^2/\Gamma$.
\end{theorem}

\subsection{Compatibility of the action of \texorpdfstring{$\Gamma'/\Gamma$}{Lg} on \texorpdfstring{$G(\Gamma)$}{Lg}}\label{sub:compatibilitysurgroupe}

Our aim is now to define an action of $\Omega$ on $G(\Gamma)$ such that the action of $G(\Gamma)$ on $M(\Gamma)$ lifts to an action of the semidirect product $G(\Gamma) \rtimes \Omega$, i.e $\varpi.(g.\varphi)=(\varpi.g).(\varpi.\varphi)$ for all $\varpi \in \Omega$, $g \in G(\Gamma)$ and $\varphi \in M(\Gamma)$. Furthermore, the morphism $\tau:Z \rightarrow \mathfrak{h}$ has to be $\Omega$-equivariant. The action will be found by a case by case analysis.  

Set $\Delta(\Gamma)=A_{2r-1}$, and define the orientation of the quiver $Q$ by setting $Q_1^+=\{a_i \ | \ 0 \leq i \leq 2r-1\}$ (cf. Subsection~\ref{computationactiontypeA} for notations). Let $g=(g_0,\ldots ,g_{2r-1})$ be a representative of an element in $G(\Gamma)$ and let $\varphi=(a_0,\ldots ,a_{2r-1},b_0,\ldots ,b_{2r-1}) \in M(\Gamma)$. Define the action of $\Omega=\mathbb{Z}/2\mathbb{Z}=<\sigma>$ on $M(\Gamma)$ by  
\begin{center}
$\sigma.\varphi=(-b_{2r-1},\ldots ,-b_r, b_{r-1},\ldots ,b_0,a_{2r-1},\ldots ,a_r,-a_{r-1},\ldots ,-a_0)$.
\end{center}
It can be verified that this action satisfies the conditions of the previous section and is then symplectic, induces the natural action on the singularity $\cc^2/\Gamma$, and verifies $\sigma^2=\text{Id}$. 

Let $\Omega$ act on $G(\Gamma)$ by 
 \begin{center}
$\sigma.g=(g_0,g_{2r-1},g_{2r-2}, \ldots , g_2 ,g_1)$, for any $g \in G(\Gamma)$.
\end{center}
One may check explicitly that the desired relation $\sigma.(g.\varphi)=(\sigma.g).(\sigma.\varphi)$ is satisfied. This action also verifies the following proposition:

\begin{proposition}
The action of $\Omega$ on $G(\Gamma)$ defined above induces the correct action on the Cartan subalgebra $\mathfrak{h}$ of type $A_{2r-1}$.
\end{proposition}

\begin{proof}
By definition $G(\Gamma)=\prod_{i=0}^{2r-1}\mathrm{GL}_{d_i}(\cc)/\cc^*$ with $d_i=1$ for all $i$. The group $\Omega$ acts on $G(\Gamma)$ by permuting the $i^{\text{th}}$ and $(2r-i)^{\text{th}}$ coordinates for all $i \geq 1$, and by fixing the $0^{\text{th}}$ coordinate. As the action is linear, it stays the same on $\mathfrak{g}(\Gamma)=\mathrm{Lie}(G(\Gamma))=\ \bigoplus_{i=0}^{2r-1}\mathfrak{gl}_{d_i}(\cc)/\cc$. Recall that $Z$ is defined as the dual of the center of $\mathfrak{g}(\Gamma)$. If $z \in Z$, then there exist $z_0,\ldots ,z_{2r-1} \in \cc$ such that $\sum_{i=0}^{2r-1}d_iz_i=0$ and $z=(z_0\mathrm{Id}_{d_0},\ldots ,z_{2r-1}\mathrm{Id}_{d_{2r-1}})$. Afterwards, the action of $\Omega$ on $Z$ is given by 
\begin{center}
$\sigma.z=(z_0,z_{2r-1},\ldots ,z_1)$.
\end{center}
Thus $\Omega$ acts on $Z$ by permuting the coordinates $i$ and $2r-i$. 

Let $(\alpha_1,\ldots ,\alpha_{2r-1})$ be a set of simple roots of the root system of type $A_{2r-1}$, and consider $(\alpha_1^\vee,\ldots ,\alpha_{2r-1}^\vee)$ the corresponding coroots. The action of $\Omega$ on $\mathfrak{h}$ is the permutation of the $i^{\text{th}}$ and $(2r-i)^{\text{th}}$ coordinates, with $1 \leq i \leq 2r-1$. The isomorphism $\tau:Z \xrightarrow{\cong} \mathfrak{h}$ of Section~\ref{sec:SlodowyCassens} is given by
\begin{center}
\begin{tikzpicture}[scale=1,  transform shape]
\node (1) at ( 0,0) {$Z$};
\node (2) at ( 2,0) {$\mathfrak{h}$};
\node (3) at ( 2,-0.6)  {$h$};
\node (4) at ( 0,-0.6)  {$z$};
\node (9) at (-0.5,0) {$\tau:$};

\node (5) at (0.2,0) {};
\node (6) at (1.8,0) {};
\node (7) at (0.2,-0.6) {};
\node (8) at (1.8,-0.6) {};

\draw [decoration={markings,mark=at position 1 with
    {\arrow[scale=1.2,>=stealth]{>}}},postaction={decorate}] (5)  --  (6);
\draw [|-,decoration={markings,mark=at position 1 with
    {\arrow[scale=1.2,>=stealth]{>}}},postaction={decorate}] (7)  --  (8);
\end{tikzpicture}

\end{center}
with $\alpha_i(h)=-z_i$, for all $1 \leq i \leq 2r-1$. Using the fundamental root system and its coroot system, it can be verified that $\tau(\sigma.z)=\sigma.\tau(z)$ for all $z \in Z$, and so $\tau$ is $\Omega$-equivariant. \hfill $\Box$
\end{proof}

The cases $D_{r+1}$, $E_6$ as well as $(D_4,\mathfrak{S}_3)$ are dealt with in a similar fashion. 

\subsection{Flat coordinates}\label{flatcoordinates}

\subsubsection{Definition of flat coordinates}

Flat coordinates were defined by K. Saito in \cite{Saito79} and \cite{Saito83} in the context of flat structures of tangent bundles of parameter spaces of universal unfoldings. In this article, we restrict ourselves to the context of simple singularities. 

Let $\mathfrak{h}$ be a Cartan subalgebra of a simple Lie algebra $\mathfrak{g}$ of rank $r$ with simply laced Dynkin diagram, and $W$ the associated Weyl group. According to Chevalley's Theorem, the ring $S(\mathfrak{h}^*)^W$ of $W$-invariant polynomials on $\mathfrak{h}$ is generated by $r$ algebraically independent homogeneous polynomials $P_1,\ldots, P_r$ of degrees $m_1 + 1=2 <\ldots<m_r+1=h$, with $h$ the Coxeter number of $\mathfrak{g}$. Let $\Delta$ be the product of all the linear functions defining reflection hyperplanes of reflections in $W$ ($\Delta^2$ is called the discriminant of $\mathfrak{h}$). Then $\Delta$ is a fundamental anti-invariant of $W$ and is a homogeneous polynomial of degree $\frac{r h}{2}$ in $S(\mathfrak{h})$. The polynomial $\Delta^2$ is $W$-invariant and can thus be written as $\Delta^2=a_0P_r^r + a_1P_r^{r-1} + \ldots  + a_r$, with $a_i$ a polynomial in $P_1,\ldots, P_{r-1}$ of degree $hi$ for $0 \leq i \leq r$. Let $s$ be a Coxeter element of $W$. The eigenvalues of $s$ are $h^{\text{th}}$ roots of unity. Let $\xi \in \mathfrak{h}$ be an eigenvector of $s$ associated to the eigenvalue $\exp(\frac{2i\pi}{h})$. Because $\mathrm{deg} \ P_i<h$ for $1 \leq i<r$, one has $P_i(\xi)=0$ for $1 \leq i <r$. Therefore $a_i(\xi)=0$ for $1 \leq i \leq r$ and $\Delta(\xi)^2=a_0P_r(\xi)^r$. One can show (cf. Proposition 3, \S 6, V in \cite{Bou68}) that $\Delta(\xi) \neq 0$, hence $a_0 \neq 0$. 

The Killing form $\kappa$ on $\mathfrak{h}$ induces an inner product 
\begin{center}
$I(dP_i, dP_j)= \displaystyle \sum_{m, n}\frac{\partial P_i}{\partial  x_m}\frac{\partial P_j}{\partial  x_n}\kappa(x_m, x_n)$
\end{center}
on the cotangent vectors $dP_i$ ($i=1,\ldots, r$) on $\mathfrak{h}/W$. The following proposition can be found in \cite{SekYan79}.

\begin{proposition}
There exists a constant $c$ such that
\begin{center}
$\det(I(dP_i,dP_j)_{1 \leq i, j \leq r})=c\Delta^2$.
\end{center}
\end{proposition}

The proposition implies that the inner product degenerates along the discriminant. Because the degree of $P_r$ is maximal among the generators of $S(\mathfrak{h}^*)^W$, the operator $D= \frac{\partial}{\partial P_r}$ is unique up to scalar multiplication. We define a new inner product $J(dP_i, dP_j)=DI(dP_i, dP_j)$ on $\mathfrak{h}/W$. The next theorem comes from \cite{SYS80}.

\begin{theorem}
The following assertions are verified:
\vspace{-1ex}
\begin{enumerate}
\item The bilinear form $J$ is non-degenerate and $\det(J(dP_i, dP_j)_{1 \leq i, j \leq r})=a_0$.
\item The inner product $J$ does not depend on the coordinates $P_1, \ldots, P_r$.
\item There exists a set of coordinates $\psi_1, \ldots , \psi_r$ on $\mathfrak{h}/W$ such that, for any $1 \leq i, j \leq r$, the expression $J(d\psi_i, d\psi_j)$ is constant. These coordinates are called \textit{flat coordinates}.
\end{enumerate}
\end{theorem}

The next proposition is proved in \cite{Yano80} and explains why we take interest in flat coordinates.

\begin{proposition}
Let $\mathfrak{h}$ be a Cartan subalgebra of a simply laced simple Lie algebra $\mathfrak{g}$, and $W$ the associated Weyl group. Then the action of any automorphism of the Dynkin diagram of $\mathfrak{g}$ on $\mathfrak{h}/W$ is linear relative to flat coordinates.
\end{proposition}

In the next subsections, flat coordinates for $\mathfrak{g}$ of type $A_{2r-1}$, $D_{r+1}$ and $E_6$ will be explicitly given for later purposes. For more details on flat coordinates, the reader may consult \cite{Saito93} and \cite{Saito01}.

 \subsubsection{Flat coordinates for $A_{2r-1}$}\label{sub:flatA2r-1}
 
Let $\mathfrak{g}$ be a real simple Lie algebra of type $A_{2r-1}$, with $\mathfrak{h}$ a Cartan subalgebra of $\mathfrak{g}$ and $W$ the associated Weyl group. Let $(\xi_1, \ldots, \xi_{2r})$ be an orthonormal basis of a Euclidian space $V \cong \mathbb{R}^{2r}$. The subalgebra $\mathfrak{h}$ is isomorphic to the subspace of $V$ defined by $\{ \sum_{i=1}^{2r}y_i\xi_i \in V \ | \ \sum_{i=1}^{2r}y_i=0\}$. The set of roots of $A_{2r-1}$ consists of $\xi_i-\xi_j$ ($1 \leq i, j \leq 2r$ and $i \neq j$). The Weyl group $W$ is $\mathfrak{S}_{2r}$  and permutes the $\xi_i$'s. Hence the ring of $W$-invariant polynomials is generated by $\epsilon_2, \epsilon_3, \ldots, \epsilon_{2r}$ with $\epsilon_{i}(\xi)$ being the $i^{\text{th}}$ elementary symmetric polynomial in $2r$ variables $\xi=(\xi_1,\ldots, \xi_{2r})$. Set $I =\{2, 3,\ldots, 2r\}$. 

According to \cite{SYS80}, flat coordinates for $A_{2r-1}$ are $\psi_2,\psi_3, \ldots, \psi_{2r}$ with 
\begin{center}
$\psi_i=\displaystyle \sum_{d \geq 1}\frac{(-1)^{d-1}(\frac{1}{h}(h-i + 1),d-1)}{d!}X_i^d$
\end{center} 
where
\begin{center}
$X_i^d=\displaystyle \sum_{\begin{array}[h]{c}
i_1 + \ldots + i_d=i \\
i_j \in I
\end{array}}\epsilon_{i_1}\ldots \epsilon_{i_d}$,
\end{center}
$(a, n)=a(a + 1)\ldots(a + n-1)$, and $h$ is the Coxeter number of $\mathfrak{g}$. For any $i \in I$, the polynomial $\psi_i$ is homogeneous of degree $i$.

\subsubsection{Flat coordinates for $D_{r+1}$}\label{sub:flatD}

Let $\mathfrak{g}$ be a real simple Lie algebra of type $D_{r+1}$, with $\mathfrak{h}$ a Cartan subalgebra of $\mathfrak{g}$ and $W$ the associated Weyl group. Let $(\xi_1, \ldots, \xi_{r+1})$ be an orthonormal basis of $\mathbb{R}^{r+1}$. The set of roots of $D_{r+1}$ consists of $ \pm \xi_i \pm \xi_j$ ($1 \leq i < j \leq r+1)$.  The Weyl group $W$ is the semi-direct product $ (\mathbb{Z}/2\mathbb{Z})^{r} \rtimes \mathfrak{S}_{r+1}$. The group $\mathfrak{S}_{r+1}$ permutes the $\xi_i$'s and $(\mathbb{Z}/2\mathbb{Z})^r$ acts by $\xi_i \mapsto (\pm1)_i\xi_i$ such that $\prod_i (\pm1)_i=1$. Hence the ring of $W$-invariant polynomials is generated by $x_2, x_4, \ldots, x_{2r},\psi$ where $x_{2i}=\epsilon_i(\xi^2)$ with $\xi^2=(\xi_1^2,\ldots, \xi_{r+1}^2)$ and $\epsilon_i$ is the $i^{\text{th}}$ elementary symmetric polynomial in $r+1$ variables, and $\psi=\prod_{i=1}^{r+1}\xi_i$. Set $I =\{2, 4,\ldots, 2r\}$. 

According to \cite{SYS80}, flat coordinates for $D_{r+1}$ are $\psi_2,\psi_4, \ldots, \psi_{2r}, \psi$ with 
\begin{center}
$\psi_i=\displaystyle \sum_{d \geq 1}\frac{(-1)^{d-1}(\frac{1}{h}(h-i + 1),d-1)}{d!}X_i^d$
\end{center} 
where
\begin{center}
$X_i^d=\displaystyle \sum_{\begin{array}[h]{c}
i_1 + \ldots + i_d=i \\
i_j \in I
\end{array}}x_{i_1}\ldots x_{i_d}$,
\end{center}
$(a, n)=a(a + 1)\ldots(a + n-1)$, and $h$ is the Coxeter number of $\mathfrak{g}$ (the formulae are similar to the case $A_{2r-1}$, except for $X_i^d$ which is now a function of the $x_i$'s rather than the $\epsilon_i$'s). For any $i \in I$, the polynomial $\psi_i$ is homogeneous of degree $i$. 

According to the preceding formulae, flat coordinates for $D_4$ are $\psi_2, \psi_4, \psi_6$ and $\psi$ with: 

\begin{center}
$\raisebox{-.8\height}{$\left\lbrace \renewcommand{\arraystretch}{1} \begin{array}[h]{l}
\psi_2=x_2=\displaystyle \sum_{i=1}^4\xi_i^2, \\
\psi_4=x_4-\frac{1}{4}x_2^2=\displaystyle \sum_{1 \leq i<j \leq 4}\xi_i^2\xi_j^2-\frac{1}{4}(\sum_{i=1}^4\xi_i^2)^2,\\
\psi_6=x_6-\frac{1}{6}x_2x_4 + \frac{7}{216}x_2^3=\displaystyle \sum_{1 \leq i<j<k \leq 4}\xi_i^2\xi_j^2\xi_k^2-\frac{1}{6}(\sum_{i=1}^4\xi_i^2)(\sum_{1 \leq i<j \leq 4}\xi_i^2\xi_j^2) + \frac{7}{216}(\sum_{i=1}^4\xi_i^2)^3, \\
\psi=\xi_1\xi_2\xi_3\xi_4.
 \end{array}\right.$}$
\end{center}
These formulae will be used in Subsections~\ref{C3} and ~\ref{G2}.

\subsubsection{Flat coordinates for $E_6$}\label{sub:flatE6}

The computation of flat coordinates for type $E_6$ requires a more delicate approach than the previous cases. In \cite{Frame51}, J.S. Frame studied the Weyl group $W$ of type $E_6$ as the group of automorphisms of the 27 lines on a non-singular cubic surface. He described the 27 lines as 27 complex triples:
\begin{center}
$\raisebox{-.8\height}{$\left\lbrace \renewcommand{\arraystretch}{1.5} \begin{array}[h]{l}
(0,\omega^\lambda,\omega^\mu) \text{ with } \lambda, \mu=1,2,3, \\
(-\omega^\mu,0,\omega^\lambda)  \text{ with } \lambda, \mu=1,2,3, \\
(\omega^\lambda,-\omega^\mu,0)  \text{ with } \lambda, \mu=1,2,3,
\end{array}\right.$}$
\end{center}
with $\omega=\exp(\frac{2i\pi}{3})$. By defining $x_1,y_1,x_2,y_2,x_3,y_3$ as the real and imaginary parts of the complex triples, the 27 lines can be identified with the 27 vertices of a polyhedron in $\mathbb{R}^6$ and $W$ as its group of symmetries. The polyhedron has 36 hyperplanes of symmetries, each given by its normal vector which is one of the following:
\begin{center}
$ \left\lbrace \renewcommand{\arraystretch}{1.5} \begin{array}[h]{l}
D_{\kappa,\lambda,\mu}=\frac{1}{\sqrt{3}} (c_\kappa, s_\kappa, c_\lambda, s_\lambda, c_\mu, s_\mu)^T \text{ with } \kappa,\lambda,\mu=1,2 \text{ or } 3, \\
D_{\kappa,0,0}= (-s_\kappa, c_\kappa,0,0,0,0)^T \text{ with } \kappa =1,2 \text{ or } 3, \\
D_{0,\lambda,0}= (0,0,-s_\lambda, c_\lambda,0,0)^T \text{ with } \lambda =1,2 \text{ or } 3, \\
D_{0,0,\mu}= (0,0,0,0,-s_\mu, c_\mu)^T \text{ with } \mu =1,2 \text{ or } 3,
\end{array}\right.$
\end{center}
with $c_a=\cos(\frac{2a\pi}{3})$ and $s_a=\sin(\frac{2a\pi}{3})$. Let $s_k=\mathrm{Id}-2D_k D_k^T$ denote the reflection in $\mathbb{R}^6$ of hyperplane whose normal vector is $D_k$, with $k$ a triple. J.S. Frame proved that $W$ is generated by $s_{1,0,0}, s_{0,1,0}, s_{0,0,1}, s_{3,0,0}, s_{0,0,3}$ and $s_{3,3,3}$. 

Let $\alpha_1, \ldots, \alpha_6$ be a set of simple roots of the root system of type $E_6$. Index the Dynkin diagram in the following way (it is the indexation that will be used for the computation of the semiuniversal deformation):
\begin{center}
\begin{tikzpicture}[scale=0.6,  transform shape]
\tikzstyle{point}=[circle,draw,fill]
\tikzstyle{ligne}=[thick]

\node (1) at ( -4,0) [point] {};
\node (2) at ( 0,-2) [point] {};
\node (3) at ( -2,0) [point] {};
\node (4) at ( 0,0) [point] {};
\node (5) at ( 2,0) [point] {};
\node (6) at ( 4,0) [point] {};

\node (7) at ( -4,0.5)  {1};
\node (8) at ( -0.5,-2)  {3};
\node (9) at ( -2,0.5)  {4};
\node (10) at ( 0,0.5) {6};
\node (11) at ( 2,0.5) {5};
\node (12) at ( 4,0.5) {2};

\draw [ligne] (1)  --  (3);
\draw [ligne] (3)  --  (4);
\draw [ligne] (4)  --  (2);
\draw [ligne] (4)  --  (5);
\draw [ligne] (5)  --  (6);

\end{tikzpicture}
\end{center}

\begin{remark}
The indexation used here is not the standard one used in reference books like \cite{Bou68}, which is
\begin{center}
\begin{tikzpicture}[scale=0.6,  transform shape]
\tikzstyle{point}=[circle,draw,fill]
\tikzstyle{ligne}=[thick]

\node (1) at ( -4,0) [point] {};
\node (2) at ( 0,-2) [point] {};
\node (3) at ( -2,0) [point] {};
\node (4) at ( 0,0) [point] {};
\node (5) at ( 2,0) [point] {};
\node (6) at ( 4,0) [point] {};

\node (7) at ( -4,0.5)  {1};
\node (8) at ( -0.5,-2)  {2};
\node (9) at ( -2,0.5)  {3};
\node (10) at ( 0,0.5) {4};
\node (11) at ( 2,0.5) {5};
\node (12) at ( 4,0.5) {6};

\draw [ligne] (1)  --  (3);
\draw [ligne] (3)  --  (4);
\draw [ligne] (4)  --  (2);
\draw [ligne] (4)  --  (5);
\draw [ligne] (5)  --  (6);

\end{tikzpicture}
\end{center}
\end{remark}

If $s_{\alpha_i}$ denote the simple reflection in $W$ associated to $\alpha_i$, then $s_{\alpha_i}(\alpha_j)=\alpha_j-c_{ji}\alpha_i$, with $(c_{ij})_{1 \leq i \leq 6}$ the Cartan matrix of $E_6$. With the Cartan matrix, one can identify the $s_{\alpha_i}$'s with their corresponding $s_k$'s, and as such identify the $\alpha_i$'s with the $D_k$'s. The correspondence is as follows:
\begin{center}
$ \begin{array}[h]{lcl}
\alpha_1 \longleftrightarrow \sqrt{2}D_{3,0,0}=\sqrt{2}(0, 1,0,0,0,0)^T &;& \alpha_2 \longleftrightarrow \sqrt{2}D_{0,0,3}=\sqrt{2}(0,0,0,0,0,1)^T, \\
\alpha_3 \longleftrightarrow \sqrt{2}D_{0,1,0}=\sqrt{2}(0,0,-\frac{\sqrt{3}}{2},-\frac{1}{2},0,0)^T &;& \alpha_4 \longleftrightarrow \sqrt{2}D_{1,0,0}=\sqrt{2}(-\frac{\sqrt{3}}{2},-\frac{1}{2},0,0,0,0)^T, \\
\alpha_5 \longleftrightarrow \sqrt{2}D_{0,0,1}=\sqrt{2}(0,0,0,0,-\frac{\sqrt{3}}{2},-\frac{1}{2})^T &;& \alpha_6 \longleftrightarrow \sqrt{2}D_{3,3,3}=\sqrt{\frac{2}{3}}(1, 0,1,0,1,0)^T.
\end{array}$
\end{center} 

Set $p_i=x_i^2 + y_i^2$, $q_i=\frac{1}{3}x_i^3- x_i y_i^2$ for $1 \leq i \leq 3$, and define two differential operators $\Theta$ and $\Delta$ by
\begin{center}
$\left\lbrace   \begin{array}[h]{l}
\Theta=\displaystyle \sum_{i=1}^3(3q_i(p_j-p_k)-2p_i(q_j-q_k))\frac{\partial}{\partial p_i} + (\frac{1}{2}p_i^2(p_j-p_k)-3q_i(q_j-q_k))\frac{\partial}{\partial q_i} \\
 \hspace{6cm}\text{ for }(ijk)=(123), (231) \text{ and } (312) ; \\
\Delta=\displaystyle \sum_{i=1}^34\frac{\partial}{\partial p_i}p_i\frac{\partial}{\partial p_i} + 12q_i\frac{\partial^2}{\partial p_i \partial q_i} + p_i^2\frac{\partial^2}{\partial q_i^2}.
\end{array}\right.$
\end{center}
By defining $A=p_1 + p_2 + p_3$, $B=\frac{1}{5}\Theta A$, $H=\Theta B$, $C=\frac{1}{16}\Delta H$, $J=\frac{1}{9}(\Theta C-3A^2B)$ and $K=\frac{2}{3}\Theta J$, it is proved in \cite{SYS80} that flat coordinates of type $E_6$ are then given by
\begin{center}
$\left\lbrace  \renewcommand{\arraystretch}{1.3} \begin{array}[h]{l}
\psi_2=A,\\
\psi_5=B,\\
\psi_6=C-\frac{1}{8}A^3,\\
\psi_8=H-\frac{1}{4}AC + \frac{5}{192}A^4,\\
\psi_9=J,\\
\psi_{12}=K-\frac{1}{8}A^2H-\frac{1}{8}C^2 + \frac{5}{96}A^3C-AB^2-\frac{1}{256}A^6.
 \end{array}\right.$
\end{center}

\section{Computations for the inhomogeneous cases}\label{sec:computations}
In previous sections we have determined the necessary requirements on the action of $\Omega$, on the McKay quiver and on the coordinates of the base space $\mathfrak{h}/W$ to construct the semiuniversal deformations of the inhomogeneous simple singularities. The following subsections will present the results of the ensuing computations.

\subsection{Semiuniversal deformation for the type \texorpdfstring{$B_r=(A_{2r-1},\mathbb{Z}/2\mathbb{Z})$}{Lg}}\label{defA2r-1}

In this subsection, the groups $\Gamma$ and $\Gamma'$ are $\mathcal{C}_{2r}$ and $\mathcal{D}_r$, respectively, and so $\Omega =\mathbb{Z}/2\mathbb{Z}=<\sigma>$. The McKay quiver $Q$ is 

\begin{center}
 \begin{tikzpicture}[scale=0.8,  transform shape]
\tikzstyle{point}=[circle,draw,fill]
\tikzstyle{ligne}=[thick]
\tikzstyle{pointille}=[thick,dotted]

\node (1) at ( 5,0) [point]{};
\node (2) at ( 4,2) [point] {};
\node (3) at ( 2,2) [point]{} ;
\node (4) at ( -2,2) [point] {};
\node (5) at ( -4,2) [point] {};
\node (6) at ( -5,0) [point] {};
\node (7) at ( -4,-2) [point] {};
\node (8) at ( -2,-2) [point]{} ;
\node (9) at ( 2,-2) [point] {};
\node (10) at ( 4,-2) [point] {};

\node at ( -3,2.4)  {$b_{1}$};
\node at ( 3,2.4)  {$b_{r-2}$};
\node at ( -3,1.6)  {$a_{1}$};
\node at ( 3,1.6)  {$a_{r-2}$};

\node at ( -3,-2.4)  {$b_{2r-2}$};
\node at ( 3,-2.4)  {$b_{r+1}$};
\node at ( -3,-1.6)  {$a_{2r-2}$};
\node at ( 3,-1.6)  {$a_{r+1}$};

\node at ( -4.1,0.8)  {$a_{0}$};
\node at ( -4.8,1.3)  {$b_{0}$};
\node at ( -3.9,-0.8)  {$a_{2r-1}$};
\node at ( -5,-1.3)  {$b_{2r-1}$};

\node at ( 4,0.8)  {$a_{r-1}$};
\node at ( 5,1.3)  {$b_{r-1}$};
\node at ( 4.1,-0.8)  {$a_{r}$};
\node at ( 4.8,-1.3)  {$b_{r}$};

\draw [pointille] (1.8,2.1)  --  (-1.8,2.1);
\draw [pointille] (1.8,1.9)  --  (-1.8,1.9);
\draw [pointille] (1.8,-1.9)  --  (-1.8,-1.9);
\draw [pointille] (1.8,-2.1)  --  (-1.8,-2.1);

\draw [decoration={markings,mark=at position 1 with
    {\arrow[scale=1.2,>=stealth]{>}}},postaction={decorate}] (3.8,2.1)  --  (2.2,2.1);
\draw [decoration={markings,mark=at position 1 with
    {\arrow[scale=1.2,>=stealth]{>}}},postaction={decorate}](2.2,1.9)  -- (3.8,1.9) ;
\draw [decoration={markings,mark=at position 1 with
    {\arrow[scale=1.2,>=stealth]{>}}},postaction={decorate}] (-2.2,2.1) --   (-3.8,2.1);
\draw [decoration={markings,mark=at position 1 with
    {\arrow[scale=1.2,>=stealth]{>}}},postaction={decorate}] (-3.8,1.9) -- (-2.2,1.9) ;

\draw [decoration={markings,mark=at position 1 with
    {\arrow[scale=1.2,>=stealth]{>}}},postaction={decorate}] (2.2,-2.1)  --  (3.8,-2.1);
\draw [decoration={markings,mark=at position 1 with
    {\arrow[scale=1.2,>=stealth]{>}}},postaction={decorate}](3.8,-1.9)  -- (2.2,-1.9) ;
\draw [decoration={markings,mark=at position 1 with
    {\arrow[scale=1.2,>=stealth]{>}}},postaction={decorate}] (-3.8,-2.1) --  (-2.2,-2.1);
\draw [decoration={markings,mark=at position 1 with
    {\arrow[scale=1.2,>=stealth]{>}}},postaction={decorate}](-2.2,-1.9)  -- (-3.8,-1.9) ;

\draw [decoration={markings,mark=at position 1 with
    {\arrow[scale=1.2,>=stealth]{>}}},postaction={decorate}] (4,1.8)  --   (4.8,0) ;
\draw [decoration={markings,mark=at position 1 with
    {\arrow[scale=1.2,>=stealth]{>}}},postaction={decorate}]  (5,0.2)  --  (4.2,2) ;
\draw [decoration={markings,mark=at position 1 with
    {\arrow[scale=1.2,>=stealth]{>}}},postaction={decorate}] (4.8,0)   --   (4,-1.8);
\draw [decoration={markings,mark=at position 1 with
    {\arrow[scale=1.2,>=stealth]{>}}},postaction={decorate}] (4.2,-2)   --  (5,-0.2) ;

\draw [decoration={markings,mark=at position 1 with
    {\arrow[scale=1.2,>=stealth]{>}}},postaction={decorate}]   (-4.8,0) -- (-4,1.8) ;
\draw [decoration={markings,mark=at position 1 with
    {\arrow[scale=1.2,>=stealth]{>}}},postaction={decorate}]   (-4.2,2) -- (-5,0.2)  ;
\draw [decoration={markings,mark=at position 1 with
    {\arrow[scale=1.2,>=stealth]{>}}},postaction={decorate}] (-4,-1.8) -- (-4.8,0);
\draw [decoration={markings,mark=at position 1 with
    {\arrow[scale=1.2,>=stealth]{>}}},postaction={decorate}]  (-5,-0.2) -- (-4.2,-2) ;

\end{tikzpicture}
\end{center}

The orientation has been chosen as $\epsilon(a_i)=-\epsilon(b_i)=1$, for any $0 \leq i \leq 2r-1$, so that it is reversed by the action of $\Omega$. Define $M(\Gamma)$ as the representation space of this quiver with dimension vector $(1,\ldots, 1)$:
\vspace{-\topsep}
 \begin{center}
$\renewcommand{\arraystretch}{1.3} \begin{array}[t]{cl}
M(\Gamma) & =\cc^{2r }\oplus \cc^{2r}, \\
 & =\{(a_0,\ldots ,a_{2r-1},b_0,\ldots ,b_{2r-1}) \ | \ a_i, b_i \in \cc \}.
 \end{array}$
 \end{center}
 Furthermore $G(\Gamma)=(\cc^*)^{2r}/\cc^* \cong \cc^{2r-1}$ and the action of $g \in G(\Gamma)$ on $M(\Gamma)$ is done by conjugation:
 \begin{center}
 $g.(a_0,\ldots a_{2r-1},b_0,\ldots ,b_{2r-1})=(g_1g_0^{-1}a_0,\ldots g_0g_{2r-1}^{-1}a_{2r-1},g_0g_1^{-1}b_0,\ldots ,g_{2r-1}g_0^{-1}b_{2r-1})$.
 \end{center}
 
The procedure described in Section~\ref{sec:SlodowyCassens} is done explicitly in \cite{CaSlo98}. The authors obtained that the morphism
 
\begin{center}
\begin{tikzpicture}[scale=1,  transform shape]
\node (6) at ( -2.8,0) {$M(\Gamma)//G(\Gamma)=$};
\node (1) at ( 0,0) {$\{\displaystyle \prod_{i=0}^{2r-1}(z-\lambda_i)=xy\}$};
\node (2) at ( 5,0) {$ \mathfrak{h}$};
\node (3) at ( 5,-0.6)  {$(\lambda_0,\ldots ,\lambda_{2r-1})$};
\node (4) at ( 0,-0.6)  {$(\lambda_0,\ldots ,\lambda_{2r-1},x, y, z)$};

\draw [decoration={markings,mark=at position 1 with
    {\arrow[scale=1.2,>=stealth]{>}}},postaction={decorate}] (1.7,0)  -- node[above] {$\widetilde{\alpha}$} (3.8,0);
\draw [|-,decoration={markings,mark=at position 1 with
    {\arrow[scale=1.2,>=stealth]{>}}},postaction={decorate}] (1.7,-0.6)  --  (3.8,-0.6);
\end{tikzpicture}
\end{center}
where $\mathfrak{h}$ is a Cartan subalgebra of a Lie algebra of type $A_{2r-1}$, is the pullback of the semiuniversal deformation of $\widetilde{\alpha}^{-1}(0)=\{(x, y, z) \in \cc^3 \ | \ z^{2r}=xy\}$, which is a simple singularity of type $A_{2r-1}$. Hence the following diagram:

\begin{center}
\begin{tikzpicture}[scale=1,  transform shape]
\tikzstyle{point}=[circle,draw,fill]
\tikzstyle{ligne}=[thick]

\node (1) at ( 0,0) {$X_\Gamma \times_{\mathfrak{h}/W}\mathfrak{h}$};
\node (2) at ( 2,0) {$X_\Gamma$};
\node (3) at ( 2,-2)  {$\mathfrak{h}/W$};
\node (4) at ( 0,-2)  {$\mathfrak{h}$};
\node (10) at ( -2.4, 0)  {$M(\Gamma)//G(\Gamma)$};

\node (9) at ( 1,-1) {$\circlearrowright$};
\node (11) at ( -1.1,0) {$\cong$};

\draw [decoration={markings,mark=at position 1 with
    {\arrow[scale=1.2,>=stealth]{>}}},postaction={decorate}] (1)  -- node[above] {$\Psi$} (2);
\draw [decoration={markings,mark=at position 1 with
    {\arrow[scale=1.2,>=stealth]{>}}},postaction={decorate}] (2)  -- node[right] {$\alpha$} (3);
\draw [decoration={markings,mark=at position 1 with
    {\arrow[scale=1.2,>=stealth]{>}}},postaction={decorate}] (1)  -- node[left] {$\widetilde{\alpha}$} (4);
\draw [decoration={markings,mark=at position 1 with
    {\arrow[scale=1.2,>=stealth]{>}}},postaction={decorate}] (4)  -- node[below] {$\pi$} (3);

\end{tikzpicture}
\end{center}
is the pullback of the semiuniversal deformation $\alpha:X_\Gamma \longrightarrow  \mathfrak{h}/W$ of a simple singularity of type $A_{2r-1}$. 

It is known that $\cc[\mathfrak{h}]=\cc[x_0,\ldots ,x_{2r-1}]/(x_0 + \ldots  + x_{2r-1})$ and $\cc[\mathfrak{h}/W]=\cc[\psi_2(x),\ldots ,\psi_{2r}(x)]$ with the $\psi_i$'s being the flat coordinates from Subsection~\ref{sub:flatA2r-1}. Because \begin{tikzpicture}[scale=1,  transform shape, baseline=-0.5ex]
\node (1) at ( 0,0) {$\cc[\mathfrak{h}/W]$};
\node (2) at ( 2,0) {$\cc[\mathfrak{h}]$};

\node (3) at (0.5,0) {};
\node (4) at (1.5,0) {};

\draw [right hook-,decoration={markings,mark=at position 1 with
    {\arrow[scale=1.2,>=stealth]{>}}},postaction={decorate}] (1)  --  (2);

\end{tikzpicture} and $W$ is finite, there is a surjection 
\vspace{-\topsep}
\begin{center}
\begin{tikzpicture}[scale=1,  transform shape]
\node (1) at ( 0,0) {$ \mathfrak{h}$};
\node (2) at ( 4,0) {$\mathfrak{h}/W$};
\node (3) at ( 4,-0.6)  {$(\psi_2(\lambda),\ldots ,\psi_{2r}(\lambda))$.};
\node (4) at ( 0,-0.6)  {$(\lambda_0,\ldots ,\lambda_{2r-1})$};

\node (5) at (1.2,0) {};
\node (6) at (2.5,0) {};
\node (7) at (2.4,0) {};
\node (8) at (1.2,-0.6) {};
\node (9) at (2.5,-0.6) {};

\draw [decoration={markings,mark=at position 1 with
    {\arrow[scale=1.2,>=stealth]{>}}},postaction={decorate}] (5)  --  (6);
\draw [decoration={markings,mark=at position 1 with
    {\arrow[scale=1.2,>=stealth]{>}}},postaction={decorate}] (5)  -- node[above] {$\pi$} (7);
\draw [|-,decoration={markings,mark=at position 1 with
    {\arrow[scale=1.2,>=stealth]{>}}},postaction={decorate}] (8)  --  (9);

\end{tikzpicture}
\end{center}

If $\lambda \in \mathfrak{h}$, then one computes $\prod_{i=0}^{2r-1}(z-\lambda_i)=z^{2r}+\sum_{i=2}^{2r}(-1)^i\epsilon_i(\lambda)z^{2r-i}$ with $\epsilon_i$ the $i^{\text{th}}$ elementary symmetric polynomial in $2r$ variables. It is proved in \cite{SYS80} that  \begin{center}
$\epsilon_i=\displaystyle \sum_{d \geq 1}\frac{(2r-i+1,d-1)}{d! (2r)^{d-1}}Y_i^d$
\end{center} 
where
\begin{center}
$Y_i^d=\displaystyle \sum_{\begin{array}[h]{c}
i_1 + \ldots + i_d=i \\
i_j \in \{2,\ldots, 2r\}
\end{array}}\psi_{i_1}\ldots \psi_{i_d}$.
\end{center}
and $(a, n)=a(a + 1)\ldots(a + n-1)$.
Hence for any $i \in  \{2,\ldots, 2r\}$, there exists a polynomial $f_i$ such that $\epsilon_i=f_i(\psi)$, with $\psi=(\psi_2, \ldots, \psi_{2r})$. Set 
\begin{center}
$X_\Gamma=\{ (x, y, z, t_2,\ldots ,t_{2r}) \in \cc^3 \times \mathfrak{h}/W \ | \ \displaystyle z^{2r}+\sum_{i=2}^{2r}(-1)^i f_i(t_2,\ldots ,t_{2r}) z^{2r-i}=xy\}$,
\end{center}
and define two maps 
\begin{center}
\begin{tabular}{l}
\begin{tikzpicture}[scale=1,  transform shape]
\node (1) at ( 0,0) {$X_\Gamma$};
\node (2) at ( 4,0) {$\mathfrak{h}/W$};
\node (3) at ( 4,-0.4)  {$(t_2,\ldots ,t_{2r})$};
\node (4) at ( 0,-0.4)  {$(x, y, z, t_2,\ldots ,t_{2r}) $};

\node (11) at (-1.6,0) {$\alpha:$};
\node (10) at (5.5,-0.1) {,};

\node (5) at (1.4,0) {};
\node (6) at (3,0) {};
\node (7) at (1.4,-0.4) {};
\node (8) at (3,-0.4) {};

\draw [decoration={markings,mark=at position 1 with
    {\arrow[scale=1.2,>=stealth]{>}}},postaction={decorate}] (5)  --  (6);
\draw [|-,decoration={markings,mark=at position 1 with
    {\arrow[scale=1.2,>=stealth]{>}}},postaction={decorate}] (7)  --  (8);
\end{tikzpicture} \\

\begin{tikzpicture}[scale=1,  transform shape]
\node (1) at ( 0,0) {$ \mathfrak{h} $};
\node (2) at ( 4,0) {$\mathfrak{h}/W$};
\node (3) at ( 4,-0.4)  {$(\psi_2(\lambda),\ldots ,\psi_{2r}(\lambda)).$};
\node (4) at ( 0,-0.4)  {$(\lambda_0,\ldots ,\lambda_{2r-1})$};

\node (11) at (-1.4,0) {$\pi:$};

\node (5) at (1.1,0) {};
\node (6) at (2.6,0) {};
\node (9) at (2.7,0) {};
\node (7) at (1.1,-0.4) {};
\node (8) at (2.6,-0.4) {};

\draw [decoration={markings,mark=at position 1 with
    {\arrow[scale=1.2,>=stealth]{>}}},postaction={decorate}] (5)  --  (6);
\draw [decoration={markings,mark=at position 1 with
    {\arrow[scale=1.2,>=stealth]{>}}},postaction={decorate}] (5)  --  (9);
\draw [|-,decoration={markings,mark=at position 1 with
    {\arrow[scale=1.2,>=stealth]{>}}},postaction={decorate}] (7)  --  (8);
\end{tikzpicture}
\end{tabular}
\end{center}
Then the pullback of these maps is \\
\resizebox{\textwidth}{!}{$\renewcommand{\arraystretch}{1.5} \begin{array}[t]{rl}
  X_\Gamma \times_{\mathfrak{h}/W} \mathfrak{h} &= \{((x, y, z, t_2,\ldots ,t_{2r}),(\lambda_0,\ldots ,\lambda_{2r-1})) \in X_\Gamma \times \mathfrak{h} \ | \ \alpha(x, y, z, t_2,\ldots ,t_{2r})=\pi(\lambda_0,\ldots ,\lambda_{2r-1})\} ,\\
 &= \{((x, y, z, t_2,\ldots ,t_{2r}),(\lambda_0,\ldots ,\lambda_{2r-1})) \in X_\Gamma \times \mathfrak{h} \ | \ (t_2,\ldots ,t_{2r})=(\psi_2(\lambda),\ldots ,\psi_{2r}(\lambda))\}, \\
 &= \{(x, y, z,\lambda_0,\ldots ,\lambda_{2r-1}) \in \cc^3 \times \mathfrak{h} \ | \ \displaystyle  \sum_{i=0}^{2r}(-1)^i f_i(\psi_2(\lambda),\ldots ,\psi_{2r}(\lambda)) z^{2r-i}=xy\}, \\
 &= \{(x, y, z,\lambda_0,\ldots ,\lambda_{2r-1}) \in \cc^3 \times \mathfrak{h} \ | \ \displaystyle \prod_{i=0}^{2r-1}(z-\lambda_i) =xy\} ,\\
 &= M(\Gamma)//G(\Gamma),
\end{array}$}

\noindent and $X_{\Gamma,0}:=\alpha^{-1}(0)=\{  (x, y, z) \in \cc^3 \ | \ z^{2r}=xy\}$ is indeed a simple singularity of type $A_{2r-1}$. 

It remains to check that every map in the pullback diagram is $\Omega$-equivariant. Let us first verify the surjectivity of the map
 \begin{center}
 \begin{tikzpicture}[scale=1,  transform shape]
\node (1) at ( 0,0) {$X_\Gamma \times_{\mathfrak{h}/W} \mathfrak{h}$};
\node (2) at ( 5,0) {$X_\Gamma$};
\node (3) at ( 5,-0.5)  {$ (x, y, z, \psi_2(\lambda),\ldots , \psi_{2r}(\lambda))$};
\node (4) at ( 0,-0.5)  {$(x, y, z,\lambda_0,\ldots ,\lambda_{2r-1})$};

\node (9) at (-1.6,0) {$\Psi:$};

\node (5) at (1.6,0) {};
\node (6) at (3,0) {};
\node (7) at (1.6,-0.5) {};
\node (8) at (3,-0.5) {};

\draw [decoration={markings,mark=at position 1 with
    {\arrow[scale=1.2,>=stealth]{>}}},postaction={decorate}] (5)  --  (6);
\draw [|-,decoration={markings,mark=at position 1 with
    {\arrow[scale=1.2,>=stealth]{>}}},postaction={decorate}] (7)  --  (8);
\end{tikzpicture}   
\end{center}
Set $(x, y, z, t_2,\ldots ,t_{2r}) \in X_\Gamma$. Then the relation $z^{2r} + \sum_{i=2}^{2r}(-1)^i f_i(t_2, \ldots, t_{2r}) z^{2r-i} = xy$ is satisfied. Our aim is to find $\lambda \in \mathfrak{h}$ such that $\psi_2(\lambda)=t_2,\ldots , \psi_{2r}(\lambda)=t_{2r}$. It is a system of $2r-1$ polynomial equations in $2r$ variables with the relation $\sum_{i=0}^{2r-1}\lambda_i=0$. As the base field $\cc$ is algebraically closed and the $\psi_i$'s are algebraically independent, there exists a solution. Hence $\Psi$ is surjective. 

The action of $\Omega$ on $M(\Gamma)$ is defined by
\begin{center}
$\sigma.(a_0,\ldots ,a_{2r-1},b_0,\ldots ,b_{2r-1})=(-b_{2r-1},\ldots ,-b_r, b_{r-1},\ldots ,b_0,a_{2r-1},\ldots ,a_r,-a_{r-1},\ldots ,-a_0)$.
\end{center}
One can check that this action verifies the conditions in Tables~\ref{ConditionFibreSpeciale} and~\ref{ConditionOrientation}. This action is therefore symplectic and induces the natural action on the simple singularity of type $A_{2r-1}$.  

$\bullet$ Set $(x, y, z,\lambda_0,\ldots ,\lambda_{2r-1}) \in M(\Gamma)//G(\Gamma)$. It follows from the definition of the action and the construction of $M(\Gamma)//G(\Gamma)$ that
\begin{center}
$\sigma.(x, y, z,\lambda_0,\ldots ,\lambda_{2r-1})=((-1)^r y,(-1)^r x,-z,-\lambda_{2r-1},\ldots ,-\lambda_0) \in X_\Gamma \times_{\mathfrak{h}/W} \mathfrak{h}$.
\end{center}
Therefore $\widetilde{\alpha}(\sigma.(x, y, z,\lambda_0,\ldots ,\lambda_{2r-1}))=(-\lambda_{2r-1},\ldots ,-\lambda_0) \renewcommand{\arraystretch}{1.2} \begin{array}[t]{l} = \sigma.(\lambda_0,\ldots ,\lambda_{2r-1}) \text{ because of the} \\ \hspace*{0.25cm} \text{ isomorphism between } Z \text{ and }\mathfrak{h}, \\ =\sigma.\widetilde{\alpha}(\lambda_0,\ldots ,\lambda_{2r-1}),\end{array}$ \\
and $\widetilde{\alpha}$ is $\Omega$-equivariant. 
 
$\bullet$ Set $(x, y, z, t_2, \ldots ,t_{2r}) \in X_\Gamma$ and define the action of $\Omega$ on $X_\Gamma$ by 
\begin{center}
$\sigma.(x, y, z, t_2, \ldots ,t_{2r})=((-1)^r y,(-1)^r x,-z, t_2, -t_3,t_4,-t_5,\ldots ,t_{2r}) \in X_\Gamma$.
\end{center}
Therefore if $(x, y, z,\lambda_0,\ldots ,\lambda_{2r-1}) \in M(\Gamma)//G(\Gamma)$ one obtains
\begin{center}
$\renewcommand{\arraystretch}{1.3} \begin{array}[t]{rl}
\Psi(\sigma.(x, y, z,\lambda_0,\ldots ,\lambda_{2r-1})) &=\Psi((-1)^r y,(-1)^r x,-z,-\lambda_{2r-1},\ldots ,-\lambda_0), \\
& =((-1)^r y,(-1)^r x,-z, \psi_2(-\lambda),\psi_3(-\lambda), \ldots , \psi_{2r}(-\lambda)), \\
 & =((-1)^r y,(-1)^r x,-z, \psi_2(\lambda),-\psi_3(\lambda), \ldots , \psi_{2r}(\lambda)), \\
 & =\sigma.(x, y, z, \psi_2(\lambda), \ldots , \psi_{2r}(\lambda)), \\
  & = \sigma.\Psi(x, y, z,\lambda_0,\ldots ,\lambda_{2r-1}).
\end{array}$
\end{center}
This action of $\Omega$ on $X_\Gamma$ turns $\Psi$ into an $\Omega$-equivariant map. 

$\bullet$ If $(\lambda_0,\ldots ,\lambda_{2r-1}) \in \mathfrak{h}$, then $\renewcommand{\arraystretch}{1.3} \begin{array}[t]{rl}
\pi(\sigma.(\lambda_0,\ldots ,\lambda_{2r-1}))& =\pi(-\lambda_{2r-1},\ldots ,-\lambda_0) , \\
 & =(\psi_2(-\lambda),\ldots , \psi_{2r}(-\lambda)), \\
& =(\psi_2(\lambda),-\psi_3(\lambda),\ldots , \psi_{2r}(\lambda)),  \\
& =:\sigma.(\psi_2(\lambda), \psi_3(\lambda),\ldots , \psi_{2r}(\lambda)), \\
&=\sigma.\pi(\lambda_0,\ldots ,\lambda_{2r-1}).
\end{array}$ \\
The action of $\Omega$ on $\mathfrak{h}$ is naturally carried to $\mathfrak{h}/W$ by making $\pi$ an $\Omega$-equivariant map. 

We have just defined the action of $\Omega$ on $X_\Gamma$. However it needs to induce the natural action on the special fibre $X_{\Gamma,0}=\alpha^{-1}(0) \ \mathlarger{\mathlarger{\subset}} \ X_\Gamma$, i.e. the one obtained in Subsection~\ref{sub:groupaction}. The simple singularity $X_{\Gamma,0}$ is defined by $\{(x,y,z) \in \cc^3 \ | \  z^{2r}=xy\}$ and $\Omega$ acts on $X_{\Gamma,0}$ by 
\begin{center}
$\raisebox{-.6\height}{$\left\lbrace \renewcommand{\arraystretch}{1.3} \begin{array}[h]{l}
\sigma.x=(-1)^r y,\\
\sigma.y=(-1)^r x,\\
\sigma.z=-z,
\end{array}\right.$}$
\end{center}
which corresponds to the restriction to $X_{\Gamma,0}$ of action of $\Omega$ on $X_\Gamma$, ensuring that the action is well-defined. 

$\bullet$ For any $(x, y, z, t_2,\ldots ,t_{2r}) \in X_\Gamma$ one has 
\begin{center}
$\renewcommand{\arraystretch}{1.3} \begin{array}[h]{rl}
\alpha(\sigma.(x, y, z, t_2,\ldots ,t_{2r})) & =\alpha((-1)^r y,(-1)^r x,-z, t_2,-t_3,t_4,-t_5,\ldots ,t_{2r}), \\
 & = (t_2,-t_3,t_4,-t_5,\ldots ,t_{2r}), \\
 &=\sigma.(t_2,\ldots ,t_{2r}), \\
 & =\sigma.\alpha(x, y, z, t_2,\ldots ,t_{2r}).
 \end{array}$
\end{center}
It follows that $\alpha$ is $\Omega$-equivariant, and so the pullback diagram from the beginning of the subsection is $\Omega$-equivariant.

Let us look at what happens above the $\Omega$-fixed points $(\mathfrak{h}/W)^\Omega$. If $(x, y, z, t_2,\ldots ,t_{2r})$ is a point in $\alpha^{-1}((\mathfrak{h}/W)^\Omega)$, then $t_{2i + 1}=0$ for any $1 \leq i \leq r-1$.  \\
Set $X_{\Gamma, \Omega}=\{ (x, y, z, t_2, 0,t_4, 0, \ldots ,t_{2r}) \in X_\Gamma\}$. Then the restriction $\alpha^{\Omega}:X_{\Gamma,\Omega} \rightarrow (\mathfrak{h}/W)^\Omega$ of $\alpha$ is $\Omega$-invariant and $(\alpha^{\Omega})^{-1}(0) = X_{\Gamma,0}$. Therefore, due to Theorem~\ref{thm:defrestriction}, the morphism $\alpha^{\Omega}$ is a semiuniversal deformation of a simple singularity of type $(A_{2r-1},\mathbb{Z}/2\mathbb{Z})=B_r$.

\begin{remark}
As the elementary symmetric polynomials $\epsilon_i$ appear naturally in the expression of $M(\Gamma)//G(\Gamma)$, instead of flat coordinates we could have chosen $(\epsilon_2, \ldots, \epsilon_{2r})$ as a system of coordinates for $\mathfrak{h}/W$. 
\end{remark}

\subsection{Semiuniversal deformation for the type \texorpdfstring{$C_3=(D_4,\mathbb{Z}/2\mathbb{Z})$}{Lg}}\label{C3}
Set $\Gamma=\mathcal{D}_2$ and $\Gamma'=\mathcal{D}_4$. The McKay quiver $Q$ is then 

 \begin{center}
 \begin{tikzpicture}[scale=0.8,  transform shape]
\tikzstyle{point}=[circle,draw,fill]
\tikzstyle{ligne}=[thick]
\tikzstyle{pointille}=[thick,dotted]

\node (1) at ( 0,0) [point]{};
\node (2) at ( 2,2) [point] {};
\node (3) at ( -2,2) [point]{} ;
\node (4) at ( 2,-2) [point] {};
\node (5) at ( -2,-2) [point] {};

\node at ( -2.5,2)  {$0$};
\node at ( -2.5,-2)  {$1$};
\node at ( 2.5,2)  {$3$};
\node at ( 2.5,-2)  {$4$};
\node at ( 0,0.5)  {$2$};

\node at ( -0.8,1.3)  {$\varphi_0^b$};
\node at ( -0.8,-1.3)  {$\varphi_1^a$};
\node at ( 1.3,0.8)  {$\varphi_3^b$};
\node at ( 1.4,-0.7)  {$\varphi_4^a$};

\node at ( 0.8,1.3)  {$\varphi_3^a$};
\node at ( 0.8,-1.3)  {$\varphi_4^b$};
\node at ( -1.3,0.7)  {$\varphi_0^a$};
\node at ( -1.4,-0.7)  {$\varphi_1^b$};

\draw[decoration={markings,mark=at position 1 with
    {\arrow[scale=1.2,>=stealth]{>}}},postaction={decorate}] (-0.2,0.3)  --  (-1.8,1.9);
\draw[decoration={markings,mark=at position 1 with
    {\arrow[scale=1.2,>=stealth]{>}}},postaction={decorate}]  (-1.9,1.7)  --  (-0.3,0.1);

\draw[decoration={markings,mark=at position 1 with
    {\arrow[scale=1.2,>=stealth]{>}}},postaction={decorate}]   (1.8,1.9)  --  (0.2,0.3);
\draw[decoration={markings,mark=at position 1 with
    {\arrow[scale=1.2,>=stealth]{>}}},postaction={decorate}]   (0.3,0.1)  --  (1.9,1.7) ;

\draw[decoration={markings,mark=at position 1 with
    {\arrow[scale=1.2,>=stealth]{>}}},postaction={decorate}]  (-1.8,-1.9)  --  (-0.2,-0.3) ;
\draw[decoration={markings,mark=at position 1 with
    {\arrow[scale=1.2,>=stealth]{>}}},postaction={decorate}]  (-0.3,-0.1)  --  (-1.9,-1.7) ;

\draw[decoration={markings,mark=at position 1 with
    {\arrow[scale=1.2,>=stealth]{>}}},postaction={decorate}] (0.2,-0.3)  --  (1.8,-1.9);
\draw[decoration={markings,mark=at position 1 with
    {\arrow[scale=1.2,>=stealth]{>}}},postaction={decorate}]  (1.9,-1.7)  --  (0.3,-0.1);

\end{tikzpicture}
\end{center}

The dimension vector of the representation space $M(\Gamma)$ of this quiver is $(1,1,2,1,1)$, which corresponds to the coordinates of the highest root of the root system of type $D_4$. The orientation is fixed so that all $\varphi^a$'s are positive and all $\varphi^b$'s are negative, which implies that the orientation is preserved by $\Omega$. One can check that 

\begin{center}
\begin{tabular}{l}
$Z=\{ (\mu_0,\mu_1,\mu_2 \mathrm{Id}_2,\mu_3,\mu_4) \ | \ \mu_i \in \cc \text{ and } \mu_0 + \mu_1 + 2\mu_2 + \mu_3 + \mu_4=0\}$ and \\
$\mu_{CS}(\varphi)=(-\varphi_0^b\varphi_0^a,-\varphi_1^b\varphi_1^a,\varphi_0^a\varphi_0^b + \varphi_1^a\varphi_1^b + \varphi_3^a\varphi_3^b + \varphi_4^a\varphi_4^b,-\varphi_3^b\varphi_3^a,-\varphi_4^b\varphi_4^a)$. 
\end{tabular}
\end{center}

In \cite{LeBryPro90}, it is proved that the ring of invariants of the representation space of a quiver by the product of general linear groups associated with vertices, is generated by the traces of the oriented cycles of said quiver. Let us compute the traces of the oriented cycles of $Q$.

$\bullet$ Cycles of degree 2: $\mathrm{Tr}(\varphi_i^a\varphi_i^b)=\varphi_i^b\varphi_i^a=-\mu_i$ for $i=0,1,3,4$.

$\bullet$ Cycles of degree 4: $p_{ij}=\mathrm{Tr}(\varphi_i^a\varphi_i^b\varphi_j^a\varphi_j^b)$. Because $\varphi_i^a\varphi_i^b + \varphi_j^a\varphi_j^b + \varphi_k^a\varphi_k^b + \varphi_l^a\varphi_l^b=\mu_2\mathrm{Id}_2$, it follows that
\begin{center}
$\raisebox{-.6\height}{$\left\lbrace \renewcommand{\arraystretch}{1.3}  \begin{array}[h]{lcl}
p_{01} + p_{03} + p_{04} & = & -\mu_0 (\mu_0 + \mu_2),\\
p_{01} + p_{13} + p_{14} & = & -\mu_1 (\mu_1 + \mu_2),\\
p_{03} + p_{13} + p_{34} & = & -\mu_3 (\mu_3 + \mu_2),\\
p_{04} + p_{14} + p_{34} & = & -\mu_4 (\mu_4 + \mu_2).
\end{array}\right.$}$
\end{center}
This system only has 2 independent elements, and we choose $p_{03}$ and $p_{34}$.

$\bullet$ Cycles of degree 6: They are $q_{ijk}=\mathrm{Tr}(\varphi_k^a\varphi_k^b\varphi_j^a\varphi_j^b\varphi_i^a\varphi_i^b)=q_{j ki }=q_{k i j}$. One can verify that $q_{i j i}= -\mu_i p_{ij}$ and $q_{ijk} + q_{ijl} + q_{i j i} + q_{i j j}=  \mu_2 p_{ij}$. It follows that $q_{ijk} + q_{ijl}= (\mu_i + \mu_j + \mu_2) p_{ij}$. The system is then
\begin{center}
 $\raisebox{-.8\height}{$\left\lbrace\renewcommand{\arraystretch}{1.3} \begin{array}[h]{lcccl}
q_{013} + q_{014} & = &  (\mu_0 + \mu_1 + \mu_2) p_{01} & = & q_{103} + q_{104},\\
q_{103} + q_{034} & = &  (\mu_0 + \mu_3 + \mu_2) p_{03} & = & q_{013} + q_{304},\\
q_{104} + q_{304} & = &  (\mu_0 + \mu_4 + \mu_2) p_{04} & = & q_{014} + q_{034},\\
q_{134} + q_{013} & = &  (\mu_1 + \mu_3 + \mu_2) p_{13} & = & q_{314} + q_{103},\\
q_{014} + q_{314} & = &  (\mu_1 + \mu_4 + \mu_2) p_{14} & = & q_{104} + q_{134},\\
q_{134} + q_{034} & = &  (\mu_3 + \mu_4 + \mu_2) p_{34} & = & q_{314} + q_{304}.\\
\end{array}\right.$}$
\end{center}
There is only one independent element of degree 6, and we choose $q_{034}$.

$\bullet$ Cycles of degree 8: They are $z_{i j k l}=\mathrm{Tr}(\varphi_l^a\varphi_l^b\varphi_k^a\varphi_k^b\varphi_j^a\varphi_j^b\varphi_i^a\varphi_i^b)$. Like previously, computations lead to $z_{i j k k}= -\mu_k q_{ijk}$, $z_{i j k i}= -\mu_i q_{ijk}$, $z_{i j k j}= p_{jk} p_{ij}$ and $z_{i j k l}= (\mu_i + \mu_k + \mu_2) q_{ijk}-p_{ij} p_{jk}$ for $i, j, k, l$ pairwise distinct. Thus all degree 8 cycles are determined by elements with inferior degrees. 

$\bullet$ Cycles of degree $\geq 10$: Starting from degree 10, for any cycle, there is at least one repetition in the indices. Hence any element of degree $k \geq 10$ breaks down as the product of two elements with degrees strictly smaller than $k$. 

Eventually there are only 3 linearly independent elements: $p_{03}$, $p_{34}$ and $q_{034}$. By \cite{LeBryPro90}, they generate the ring of invariants. 

Because the space we want to exhibit is the pullback of the semiuniversal deformation of a simple singularity, and that such a deformation is defined by a single equation (cf. \cite{KasSchle72}), a unique relation has to link the generators (or one can compute the dimension of the quotient variety and see that it is a hypersurface). Let us find this relation. 
\begin{center}
$\renewcommand{\arraystretch}{1.3} \begin{array}[t]{rcl}
q_{034}^2 &= & q_{034}q_{034}, \\
& = & q_{034}( (\mu_0 + \mu_3 + \mu_2) p_{03}- (\mu_0 + \mu_1 + \mu_2) p_{01} +  (\mu_0 + \mu_4 + \mu_2) p_{04}-q_{304}), \\
& = &q_{034}( (\mu_0 + \mu_3 + \mu_2) p_{03}- (\mu_0 + \mu_1 + \mu_2) p_{01} +  (\mu_0 + \mu_4 + \mu_2) p_{04})-q_{034}q_{304}, \\
&  = &q_{034}( (\mu_0 + \mu_3 + \mu_2) p_{03}- (\mu_0 + \mu_1 + \mu_2) p_{01} +  (\mu_0 + \mu_4 + \mu_2) p_{04})-p_{03}p_{04}p_{34}.
\end{array}$ 
\end{center}
Using relations between the cycles of degree 4, one sees that $p_{01}$ and $p_{04}$ can be expressed using $p_{03}$ and $p_{34}$. By reinjecting these expressions in the preceding equation, one finds

\begin{center}
$q_{034}^2=q_{034}(p_{03}(\mu_3-\mu_4) + p_{34}(\mu_3-\mu_0) + C) + p_{03}p_{34}(p_{03} + p_{34}) + p_{03}p_{34}E$,
\end{center}
with $C=\mu_3(\mu_2 + \mu_3)(\mu_1 + \mu_2 + \mu_3)$ and $E=\frac{1}{2}(\mu_0 (\mu_0 + \mu_2)-\mu_1 (\mu_1 + \mu_2) + \mu_3 (\mu_3 + \mu_2) + \mu_4 (\mu_4 + \mu_2))$. 

Set $q_{034}'=q_{034}-\frac{1}{2}(p_{03}(\mu_3-\mu_4) + p_{34}(\mu_3-\mu_0) + C)$, $p_{03}'=p_{03} +  \frac{1}{4}(\mu_3-\mu_0)^2$ as well as $p_{34}'=p_{34} + \frac{1}{4}(\mu_3-\mu_4)^2$. The equation becomes
\begin{center}
$q_{034}'^2=p_{03}'p_{34}'(p_{03}' + p_{34}') + \mathcal{A}p_{03}'p_{34}' + \mathcal{B}p_{03}' + \mathcal{C}p_{34}' + \mathcal{D}$,
\end{center}
with  
\begin{center}$\raisebox{-.8\height}{$\left\lbrace  \renewcommand{\arraystretch}{1.5} \begin{array}[h]{l}
\mathcal{A} =-\mu_1\mu_2-\mu_2\mu_3-\mu_2\mu_4-\mu_2^2-\frac{1}{2}\mu_1\mu_4-\frac{1}{2}\mu_1\mu_3-\frac{1}{2}\mu_3\mu_4-\frac{1}{2}\mu_1^2-\frac{1}{2}\mu_3^2-\frac{1}{2}\mu_4^2,\\
\mathcal{B}= \frac{1}{16}(\mu_3-\mu_4)(\mu_3 + \mu_4)(2\mu_2 + \mu_3 + \mu_4)(2\mu_1 + 2\mu_2 + \mu_3 + \mu_4),\\
\mathcal{C}=\frac{1}{16}(\mu_1-\mu_4)(\mu_1 + \mu_4)(2\mu_2 + \mu_1 + \mu_4)(2\mu_3 + 2\mu_2 + \mu_1 + \mu_4), \\
\mathcal{D}=-\frac{1}{32}(2\mu_1\mu_2 + \mu_1\mu_3 + \mu_1\mu_4 + 2\mu_2^2 + 2\mu_2\mu_3 + 2\mu_2\mu_4 + \mu_3\mu_4 + \mu_4^2)(\mu_1\mu_3 + \mu_1\mu_4 \\
\hspace*{1.3cm} + 2\mu_2\mu_4 + \mu_3\mu_4 + \mu_4^2)(\mu_1\mu_3-\mu_1\mu_4-2\mu_2\mu_4-\mu_3\mu_4-\mu_4^2).
\end{array}\right.$}$ 
\end{center}
It is indeed the equation of a deformation of a simple singularity of type $D_4$. However, we want to look at the fibres over $\mathfrak{h}/W$, so the coefficients of the equation need to be invariant by the Weyl group $W$.  
 
From \cite{Kron89}, it is known that there is an isomorphism
\begin{center}
\begin{tikzpicture}[scale=1,  transform shape, baseline=-0.5ex]
\node (1) at ( 0,0) {$Z$};
\node (2) at ( 4,0) {$\mathfrak{h}$};
\node (3) at ( 4,-0.4)  {$h$};
\node (4) at ( 0,-0.4)  {$\mu=(\mu_0\mathrm{Id}_{d_0},\ldots ,\mu_4\mathrm{Id}_{d_4})$};

\node (9) at (-2.1,0) {$\tau:$};

\node (5) at (1.8,0) {};
\node (6) at (3.5,0) {};
\node (7) at (1.8,-0.4) {};
\node (8) at (3.5,-0.4) {};

\draw [decoration={markings,mark=at position 1 with
    {\arrow[scale=1.2,>=stealth]{>}}},postaction={decorate}] (5)  -- node[above] {$\cong$} (6);
\draw [|-,decoration={markings,mark=at position 1 with
    {\arrow[scale=1.2,>=stealth]{>}}},postaction={decorate}] (7)  --  (8);
\end{tikzpicture}
\end{center}
such that $\tau(\mu)=h$ with $\alpha_i(h)=-\mu_i$ for all $0 \leq i \leq 4$. Hence an element $\mu \in Z$ can be identified with $\sum_{i=1}^4-\mu_i\Lambda_i^{\vee} \in \mathfrak{h}$, where $(\Lambda_i^{\vee})_i$ are the fundamental coweights. The Weyl group $W$ of type $D_4$ is generated by the $r_{\alpha_j^{\vee}}$'s with 
\vspace{-\topsep}
\begin{center}
$\begin{array}[t]{lcl}
r_{\alpha_j^{\vee}}(\Lambda_i^{\vee}) & = &\displaystyle \Lambda_i^{\vee}-\frac{2(\Lambda_i^{\vee},\alpha_j^{\vee})}{(\alpha_j^{\vee},\alpha_j^{\vee})}\alpha_j^{\vee}, \\
 & = & \raisebox{-.0\height}{$\left\lbrace \renewcommand{\arraystretch}{1.5}  \begin{array}[h]{l}
\Lambda_i^{\vee} \text{ if } i \neq j, \\
\Lambda_i^{\vee}-\alpha_i^{\vee} \text{ if } i=j.
\end{array}\right.$}
\end{array}$
\end{center}
The simple coroots can be expressed as
\begin{center}
$\raisebox{-.6\height}{$\left\lbrace \renewcommand{\arraystretch}{1.5}  \begin{array}[h]{l}
\alpha_1^{\vee}= 2\Lambda_1^{\vee}-\Lambda_2^{\vee},\\
\alpha_2^{\vee}=-\Lambda_1^{\vee} + 2\Lambda_2^{\vee}-\Lambda_3^{\vee}-\Lambda_4^{\vee} ,\\
\alpha_3^{\vee}=-\Lambda_2^{\vee}  + 2\Lambda_3^{\vee},\\
\alpha_4^{\vee}= -\Lambda_2^{\vee} + 2\Lambda_4^{\vee},
\end{array}\right.$}$
\end{center}
using the Cartan matrix of type $D_4$. Hence $r_{\alpha_1^{\vee}}(\mu)=\mu_1\Lambda_1^{\vee}-(\mu_1 + \mu_2)\Lambda_2^{\vee}-\mu_3\Lambda_3^{\vee}-\mu_4\Lambda_4^{\vee}$, which means $r_{\alpha_1^{\vee}}(\mu_1, \mu_2, \mu_3, \mu_4)=(-\mu_1, \mu_1 + \mu_2, \mu_3, \mu_4)$. By doing the same for the other simple reflections, it can be verified that the coefficients $\mathcal{A},\mathcal{B},\mathcal{C}$ and $\mathcal{D}$ are invariant by these transformations, and thus are invariant by $W$. 

The base we chose for $\mathfrak{h}/W$ is the one with flat coordinates. Therefore the coefficients $\mathcal{A},\mathcal{B},\mathcal{C}$ and $\mathcal{D}$ have to be expressed using $\psi_2, \psi_4, \psi_6$ and $\psi$ from Subsection~\ref{sub:flatD}. \\
Let $(\overline{e_i})_{1 \leq i \leq 4}$ be the dual base of the canonical base $(e_i)_{1 \leq i \leq 4}$ of $\mathfrak{h}^*$. It is known from \cite{Bou68} that $\Lambda_1^{\vee}= \overline{e_1}$, $\Lambda_2^{\vee}=\overline{e_1} + \overline{e_2}$, $\Lambda_3^{\vee}=\frac{1}{2}(\overline{e_1} + \overline{e_2} + \overline{e_3}-\overline{e_4})$ and $\Lambda_4^{\vee}=\frac{1}{2}(\overline{e_1} + \overline{e_2} + \overline{e_3} + \overline{e_4})$. Because of the correspondence $ \mu \longleftrightarrow \sum_{i=1}^4-\mu_i\Lambda_i^{\vee}=\sum_{i=1}^4\xi_i\overline{e_i}$, the $\mu_i$'s can be expressed as functions of the $\xi_j$'s, and using the expressions from Subsection~\ref{sub:flatD}, one finds

\begin{center}
$\raisebox{-.6\height}{$\left\lbrace  \renewcommand{\arraystretch}{1.3}  \begin{array}[h]{l}
 \mathcal{A}=-\frac{1}{2}\psi_2,\\
 \mathcal{B}=-\psi,\\
 \mathcal{C}=-\frac{1}{2}(\psi+\frac{1}{2}\psi_4),\\
 \mathcal{D}=\frac{1}{4}(\psi_6+\frac{1}{6}\psi_2 \psi_4+\psi \psi_2+\frac{1}{108}\psi_2^3).
\end{array}\right.$}$
\end{center}
By setting $x=p_{34}'$, $y=p_{03}'$ and $z=q_{034}'$, the equation of the deformation becomes:
\begin{center}
$\displaystyle z^2=xy(x + y)-\frac{1}{2}\psi_2xy-\psi y -\frac{1}{2}(\psi+\frac{1}{2}\psi_4)x + \frac{1}{4}(\psi_6+\frac{1}{6}\psi_2 \psi_4+\psi \psi_2+\frac{1}{108}\psi_2^3)$.
\end{center}
The coefficients are indeed invariants of $W=\mathfrak{S}_4 \ltimes (\mathbb{Z}/2\mathbb{Z})^3$. According to Section~\ref{sec:SlodowyCassens}, it is the equation of the pullback of the semiuniversal deformation of a simple singularity of type $D_4$.
\bigbreak

Set $ \renewcommand{\arraystretch}{1.3}  \begin{array}[t]{ll} 
X_\Gamma = \{(x, y, z, t_2,t_4,t_6,t) \in \cc^3 \times \mathfrak{h}/W \ | &  z^2=xy(x + y)-\frac{1}{2}t_2xy-ty -\frac{1}{2}(t+\frac{1}{2}t_4)x \\
                                                                                                               &+ \frac{1}{4}(t_6+\frac{1}{6}t_2t_4+t t_2+\frac{1}{108}t_2^3)\}.
\end{array}$ \\
Because 
\begin{center}
\begin{tikzpicture}[scale=1,  transform shape, baseline=-0.5ex]
\node (1) at ( 0,0) {$\mathfrak{h}$};
\node (2) at ( 4,0) {$\mathfrak{h}/W$};
\node (3) at ( 4,-0.4)  {$(\psi_2(\xi),\psi_4(\xi),\psi_6(\xi),\psi(\xi))$};
\node (4) at ( 0,-0.4)  {$ (\xi_1,\xi_2,\xi_3,\xi_4)$};

\node (9) at (-1.2,0) {$\pi:$};

\node (5) at (1,0) {};
\node (6) at (2,0) {};
\node (7) at (1,-0.4) {};
\node (8) at (2,-0.4) {};

\draw [decoration={markings,mark=at position 1 with
    {\arrow[scale=1.2,>=stealth]{>}}},postaction={decorate}] (5)  --  (6);
\draw [|-,decoration={markings,mark=at position 1 with
    {\arrow[scale=1.2,>=stealth]{>}}},postaction={decorate}] (7)  --  (8);
\end{tikzpicture}  and \begin{tikzpicture}[scale=1,  transform shape, baseline=-0.5ex]
\node (1) at ( 0,0) {$X_\Gamma$};
\node (2) at ( 3,0) {$ \mathfrak{h}/W$};
\node (3) at ( 3,-0.4)  {$ (t_2,t_4,t_6,t)$};
\node (4) at ( 0,-0.4)  {$(x, y, z, t_2,t_4,t_6,t)$};

\node (9) at (-1.6,0) {$\alpha:$};

\node (5) at (1.3,0) {};
\node (6) at (2.2,0) {};
\node (7) at (1.3,-0.4) {};
\node (8) at (2.2,-0.4) {};

\draw [decoration={markings,mark=at position 1 with
    {\arrow[scale=1.2,>=stealth]{>}}},postaction={decorate}] (5)  --  (6);
\draw [|-,decoration={markings,mark=at position 1 with
    {\arrow[scale=1.2,>=stealth]{>}}},postaction={decorate}] (7)  --  (8);
\end{tikzpicture},
\end{center}
it follows that

\begin{center}
$ \renewcommand{\arraystretch}{1.5}  \begin{array}[t]{rl}
X_\Gamma \times_{\mathfrak{h}/W} \mathfrak{h} = & \{((x, y, z, t_2,t_4,t_6,t),(\xi_1,\xi_2,\xi_3,\xi_4)) \in X_\Gamma \times \mathfrak{h} \ | \ \alpha(x, y, z, t_i, t)=\pi(\xi_1,\xi_2,\xi_3,\xi_4) \}, \\
  = & \{(x, y, z, t_i, t,\xi_j) \ | \ t_2=\psi_2(\xi), t_4=\psi_4(\xi), t_6=\psi_6(\xi), t=\psi(\xi) \text{ and } z^2=xy(x + y) \\
   & -\frac{1}{2}t_2xy-ty -\frac{1}{2}(t+\frac{1}{2}t_4)x + \frac{1}{4}(t_6+\frac{1}{6}t_2t_4+t t_2+\frac{1}{108}t_2^3)\}, \\
    = & \{(x, y, z,\xi_j) \ | \ z^2=xy(x + y)-\frac{1}{2}\psi_2(\xi)xy-\psi(\xi)y -\frac{1}{2}(\psi(\xi)+\frac{1}{2}\psi_4(\xi))x \\
     & + \frac{1}{4}(\psi_6(\xi)+\frac{1}{6}\psi_2(\xi)\psi_4(\xi)+\psi(\xi) \psi_2(\xi)+\frac{1}{108}\psi_2(\xi)^3)\}, \\
      = & \mu_{CS}^{-1}(Z)//G(\Gamma).
\end{array}$
\end{center}
Finally we obtain the following commutative diagram:
\begin{center}
\begin{tikzpicture}[scale=1,  transform shape]
\tikzstyle{point}=[circle,draw,fill]
\tikzstyle{ligne}=[thick]

\node (1) at ( 0,0) {$X_\Gamma \times_{\mathfrak{h}/W}\mathfrak{h}$};
\node (2) at ( 2,0) {$X_\Gamma$};
\node (3) at ( 2,-2)  {$\mathfrak{h}/W$};
\node (4) at ( 0,-2)  {$\mathfrak{h}$};

\node (9) at ( 1,-1) {$\circlearrowright$};

\draw [decoration={markings,mark=at position 1 with
    {\arrow[scale=1.2,>=stealth]{>}}},postaction={decorate}] (1)  -- node[above] {$\Psi$} (2);
\draw [decoration={markings,mark=at position 1 with
    {\arrow[scale=1.2,>=stealth]{>}}},postaction={decorate}] (2)  -- node[right] {$\alpha$} (3);
\draw [decoration={markings,mark=at position 1 with
    {\arrow[scale=1.2,>=stealth]{>}}},postaction={decorate}] (1)  -- node[left] {$\widetilde{\alpha}$} (4);
\draw [decoration={markings,mark=at position 1 with
    {\arrow[scale=1.2,>=stealth]{>}}},postaction={decorate}] (4)  -- node[below] {$\pi$} (3);
\end{tikzpicture}
\end{center}
with $\alpha$ being the semiuniversal deformation of $X_{\Gamma,0}:=\alpha^{-1}(0)=\{(x, y, z) \in \cc^3 \ | \ z^2=xy(x + y) \}$, which is a simple singularity of type $D_4$. 

The action of $\Omega=\mathbb{Z}/2\mathbb{Z}=<\sigma>$ on $\mu_{CS}^{-1}(Z)//G(\Gamma)$ is the natural $3 \leftrightarrow 4$ permutation. There is also a natural action of $\Omega$ on $\mathfrak{h}$ given by permuting $\alpha_3^\vee$ and $\alpha_4^\vee$, which corresponds to the permutation $\overline{e_4} \leftrightarrow -\overline{e_4}$. This action is carried to $\mathfrak{h}/W$ by defining $\sigma.(t_2,t_4,t_6,t)=(t_2,t_4,t_6,-t)$. 

Above the fixed points $(\mathfrak{h}/W)^\Omega$, define $X_{\Gamma, \Omega}=\alpha^{-1}((\mathfrak{h}/W)^\Omega)=\{(x, y, z, t_2,t_4,t_6,0) \in X_\Gamma\}$. It was previously shown that
\begin{center}
$\raisebox{-.6\height}{$\left\lbrace \renewcommand{\arraystretch}{1.3} \begin{array}[h]{lcl}
x & = & p_{34} + \frac{1}{4}(\mu_3-\mu_4)^2,\\
y & = & p_{03} + \frac{1}{4}(\mu_3-\mu_0)^2,\\
z & = & q_{034}-\frac{1}{2}(p_{03}(\mu_3-\mu_4) + p_{34}(\mu_3-\mu_0)  + \mu_3(\mu_2 + \mu_3)(\mu_1 + \mu_2 + \mu_3)),
\end{array}\right.$}$
\end{center}
and $\mu_3 \xleftrightarrow{\sigma}  \mu_4$. Because of the relations between degree 4 elements, we obtain the following formulae:
\begin{center}
 $\raisebox{-.6\height}{$\left\lbrace \renewcommand{\arraystretch}{1.3} \begin{array}[h]{rcl}
\sigma.x & = & x, \\
\sigma.y & = & -x-y + \frac{1}{2}t_2, \\
\sigma.z & = & -z.
\end{array}\right.$}$ 
\end{center}
Above the origin, the fibre has the equation $z^2=xy(x + y)$. Changing variables by $X  =  -4^{-1/3}x$ and $Y  =  -4^{1/6}(y + \dfrac{x}{2})$ leads to the equation $X(Y^2-X^2) + z^2=0$ with $\sigma.X  =  X$, $\sigma.Y  =  -Y$ and $\sigma.z  =  -z$. This shows that the action on the singularity is the same as the one in Subsection~\ref{sub:groupaction}. 

Finally, the pullback diagram is $\Omega$-equivariant with the natural action on the singularity, so the restriction $\alpha^{\Omega}:X_{\Gamma, \Omega} \rightarrow (\mathfrak{h}/W)^\Omega$ of $\alpha$ is $\Omega$-invariant and $(\alpha^{\Omega})^{-1}(0) =X_{\Gamma,0}$. Because of Theorem~\ref{thm:defrestriction}, the morphism $\alpha^{\Omega}$ is a semiuniversal deformation of a simple singularity of type $(D_{4},\mathbb{Z}/2\mathbb{Z})=C_3$.

\subsection{Semiuniversal deformation for the type \texorpdfstring{$G_2=(D_4,\mathfrak{S}_3)$}{Lg}}\label{G2}
Set $\Gamma=\mathcal{D}_2$ and $\Gamma'=\mathcal{O}$. The McKay quiver $Q$ is given in the previous subsection and the same notations are kept. The pullback diagram of the semiuniversal deformation $\alpha$ of the simple singularity $\cc^2/\Gamma$ of type $D_4$ has also been previously determined. 

The symmetry group is $\Omega=\mathfrak{S}_3=\langle \sigma, \rho \rangle$ with $\sigma^2=\rho^3=1$. The group $\Omega$ induces a natural action on $\mu_{CS}^{-1}(Z)//G(\Gamma)$ and permutes the vertices $1,3$ and $4$. There is also a natural action of $\Omega$ on $\mathfrak{h}$ given by 
\begin{center}
\raisebox{.3\height}{\begin{tikzpicture}[baseline, scale=1,  transform shape, decoration={markings, mark=at position 1 with {\arrow[scale=1.2,>=stealth]{>}}}]
\node (1) at ( 0,0) {$\alpha_1^\vee$};
\node (2) at ( 1,0)  {$\alpha_3^\vee$};
\node (3) at ( 2,0) {$\alpha_4^\vee$} ;
\node (6) at ( 1,-0.35) {$\rho$} ;

\draw [postaction={decorate}] (0.2,0)  --  (0.8,0);
\draw [postaction={decorate}] (1.2,0)  --  (1.8,0);
\draw [postaction={decorate}] (1.8,-0.2) to [out=-135, in=-45] (0.2,-0.2);
\draw [postaction={decorate}] (0.21,-0.21) -- (0.2,-0.2);
\end{tikzpicture}} and \raisebox{.3\height}{\begin{tikzpicture}[baseline, scale=1,  transform shape]
\tikzstyle{point}=[circle,draw,fill]
\tikzstyle{ligne}=[thick]
\tikzstyle{pointille}=[thick,dotted]

\node (1) at ( 0,0) {$\alpha_3^\vee$};
\node (2) at ( 1.5,0)  {$\alpha_4^\vee$};

\node (6) at ( 0.75,-0.2) {$\sigma$} ;

\draw [decoration={markings,mark=at position 1 with
    {\arrow[scale=1.2,>=stealth]{>}}},postaction={decorate}] (0.2,0)  --  (1.3,0);
\draw [decoration={markings,mark=at position 1 with
    {\arrow[scale=1.2,>=stealth]{>}}},postaction={decorate}] (1.3,0)  --  (0.2,0);

\end{tikzpicture}}. 
\end{center}
From the previous case $(D_4, \mathbb{Z}/2\mathbb{Z})$, it is known that $\sigma$ corresponds to the permutation $\overline{e_4} \leftrightarrow -\overline{e_4}$. It remains to compute the action of $\rho$ on the base $(\overline{e_i})_{1\leq i \leq 4}$ of $\mathfrak{h}$. 

By expressing the $\overline{e_i}$'s in terms of the $\alpha_i^\vee$'s, one deduces that the action of $\rho$ on the $\overline{e_i}$'s is given by 
\begin{center}
$\left\lbrace \renewcommand{\arraystretch}{1.3}\begin{array}[h]{l}
 \overline{e_1} \mapsto \frac{1}{2}(\overline{e_1} + \overline{e_2} + \overline{e_3}-\overline{e_4}),\\
 \overline{e_2} \mapsto \frac{1}{2}(\overline{e_1} + \overline{e_2}-\overline{e_3} + \overline{e_4}),\\
 \overline{e_3} \mapsto \frac{1}{2}(\overline{e_1}-\overline{e_2} + \overline{e_3} + \overline{e_4}),\\
 \overline{e_4} \mapsto \frac{1}{2}(\overline{e_1}-\overline{e_2}-\overline{e_3}-\overline{e_4}).
\end{array}\right.$ 
\end{center}
It was seen in Subsection~\ref{sub:flatD} that \begin{tikzpicture}[scale=1,  transform shape, baseline=-0.5ex]
\node (1) at ( 0,0) {$\mathfrak{h}$};
\node (2) at ( 4,0) {$\mathfrak{h}/W$};
\node (3) at ( 4,-0.4)  {$(\psi_2(\xi),\psi_4(\xi),\psi_6(\xi),\psi(\xi))$};
\node (4) at ( 0,-0.4)  {$ (\xi_1,\xi_2,\xi_3,\xi_4)$};

\node (9) at (-1.6,0) {$\pi:$};

\node (5) at (1,0) {};
\node (6) at (2,0) {};
\node (7) at (1,-0.4) {};
\node (8) at (2,-0.4) {};

\draw [decoration={markings,mark=at position 1 with
    {\arrow[scale=1.2,>=stealth]{>}}},postaction={decorate}] (5)  --  (6);
\draw [|-,decoration={markings,mark=at position 1 with
    {\arrow[scale=1.2,>=stealth]{>}}},postaction={decorate}] (7)  --  (8);
\end{tikzpicture} and thus 
\begin{center}
$\pi(\rho.(\xi_1,\xi_2,\xi_3,\xi_4)) =\pi \renewcommand{\arraystretch}{1.5}\begin{pmatrix} \frac{1}{2}(\xi_1 + \xi_2 + \xi_3 + \xi_4)  \\
 \frac{1}{2}(\xi_1 + \xi_2-\xi_3-\xi_4) \\
  \frac{1}{2}(\xi_1-\xi_2 + \xi_3-\xi_4) \\
   \frac{1}{2}(-\xi_1 + \xi_2 + \xi_3-\xi_4) 
   \end{pmatrix}  =\begin{pmatrix}\psi_2(\xi)  \\  -\frac{1}{2}\psi_4(\xi)-3\psi(\xi) \\ \psi_6(\xi)  \\  \frac{1}{4}\psi_4(\xi)-\frac{1}{2}\psi(\xi) \end{pmatrix}.$
   \end{center}
The action of $\rho$ on $\mathfrak{h}/W$ is therefore defined by $\rho.(t_2,t_4,t_6,t)=(t_2,  -\frac{1}{2}t_4-3t , t_6  ,  \frac{1}{4}t_4-\frac{1}{2}t )$, which makes $\pi$ $\mathbb{Z}/3\mathbb{Z}$-equivariant. It has already been shown that $\pi$ is equivariant with respect to $\sigma$, and so $\pi$ is $\Omega$-equivariant. 

Above the fixed points $(\mathfrak{h}/W)^\Omega$, define $X_{\Gamma, \Omega}=\alpha^{-1}((\mathfrak{h}/W)^{\Omega})=\{(x, y, z, t_2,0,t_6,0) \in X_\Gamma\}$. It was proved in the previous section that 
\begin{center}
$\raisebox{-.7\height}{$\left\lbrace  \renewcommand{\arraystretch}{1.3} \begin{array}[h]{lcl}
x & = & p_{34} + \frac{1}{4}(\mu_3-\mu_4)^2,\\
y & = & p_{03} + \frac{1}{4}(\mu_3-\mu_0)^2,\\
z & = & q_{034}-\frac{1}{2}(p_{03}(\mu_3-\mu_4) + p_{34}(\mu_3-\mu_0) + \mu_3(\mu_2 + \mu_3)(\mu_1 + \mu_2 + \mu_3)),
\end{array}\right.$}$
\end{center}
and the action of $\Omega$ on the $\mu_i$'s leads to
\begin{center}
$\left\lbrace \renewcommand{\arraystretch}{1.3} \begin{array}[h]{lcl}
\sigma.x & = & x, \\
\sigma.y & = & -x-y + \frac{1}{2}t_2, \\
\sigma.z & = & -z,
\end{array}\right.$ \quad and \quad  $\left\lbrace \renewcommand{\arraystretch}{1.3} \begin{array}[h]{lcl}
\rho.x & = & y, \\
\rho.y & = & -x-y + \frac{1}{2}t_2, \\
\rho.z & = & z.
\end{array}\right.$
\end{center}
The action of $\sigma$ on the singularity is known to be the correct one. The equation of the special fibre is $z^2=xy(x + y)$. By setting $X  =  -\frac{1}{\sqrt[3]{4}}x$ and $Y  =  -\sqrt[6]{4}(y + \frac{1}{2}x)$, one sees that $X(Y^2-X^2) + z^2=0$ with $\rho.X = \frac{1}{2}(Y-X)$, $\rho.Y =  -\frac{1}{2}(Y + 3X)$ and $\rho.z = z$. It is clear from Subsection~\ref{sub:groupaction} that the action on the singularity is the right one. 

The pullback diagram is finally $\Omega$-equivariant with the natural action on the singularity, hence the restriction $\alpha^{\Omega}:X_{\Gamma, \Omega} \rightarrow (\mathfrak{h}/W)^{\Omega}$ of $\alpha$ is $\Omega$-invariant and $(\alpha^{\Omega})^{-1}(0) =X_{\Gamma,0}$. Like before, one concludes that $\alpha^{\Omega}$ is a semiuniversal deformation of an inhomogeneous simple singularity of type $(D_4,\mathfrak{S}_3)=G_2$.

\subsection{Semiuniversal deformation for the type \texorpdfstring{$F_4=(E_6,\mathbb{Z}/2\mathbb{Z})$}{Lg}}\label{E6}
Set $\Gamma=\mathcal{T}$ and $\Gamma'=\mathcal{O}$. The McKay quiver $Q$ is 

\begin{center}
\begin{tikzpicture}[scale=0.7,  transform shape,>=angle 60]
\tikzstyle{point}=[circle,draw,fill]
\tikzstyle{ligne}=[thick]

\node (1) at ( -2,0) [point] {};
\node (2) at ( 0,0) [point]{};
\node (3) at ( 2,0) [point] {};
\node (4) at ( 4,0) [point]{};
\node (5) at ( 6,0) [point] {};
\node (6) at ( 2,-2) [point] {};
\node (7) at ( 2,-4) [point] {};

\node (8) at ( -2,0.5) [] {0};
\node (9) at ( 0,0.5) []{3};
\node (10) at ( 2,0.5) [] {6};
\node (11) at ( 4,0.5) []{5};
\node (12) at ( 6,0.5) [] {2};
\node (13) at ( 2.5,-2) [] {4};
\node (14) at ( 2.5,-4) [] {1};

\node (15) at ( -1,0.5) [] {$\varphi_0^a$};
\node (16) at ( -1,-0.5) [] {$\varphi_0^b$};

\node (17) at ( 1,0.5) [] {$\varphi_3^a$};
\node (18) at ( 1,-0.5) [] {$\varphi_3^b$};

\node (19) at ( 3,0.5) [] {$\varphi_5^a$};
\node (20) at ( 3,-0.5) [] {$\varphi_5^b$};

\node (21) at ( 5,0.5) [] {$\varphi_2^a$};
\node (22) at ( 5,-0.5) [] {$\varphi_2^b$};

\node (23) at ( 2.5,-1) [] {$\varphi_4^a$};
\node (24) at ( 1.5,-1) [] {$\varphi_4^b$};

\node (25) at ( 2.5,-3) [] {$\varphi_1^a$};
\node (26) at ( 1.5,-3) [] {$\varphi_1^b$};

\draw [decoration={markings,mark=at position 1 with
    {\arrow[scale=1.2,>=stealth]{>}}},postaction={decorate}] (-1.8,0.1)  --  (-0.2,0.1);
\draw [decoration={markings,mark=at position 1 with
    {\arrow[scale=1.2,>=stealth]{>}}},postaction={decorate}] (-0.2,-0.1) -- (-1.8,-0.1);
\draw [decoration={markings,mark=at position 1 with
    {\arrow[scale=1.2,>=stealth]{>}}},postaction={decorate}] (0.2,0.1)  --  (1.8,0.1);
\draw [decoration={markings,mark=at position 1 with
    {\arrow[scale=1.2,>=stealth]{>}}},postaction={decorate}] (1.8,-0.1) --  (0.2,-0.1) ;
\draw [decoration={markings,mark=at position 1 with
    {\arrow[scale=1.2,>=stealth]{>}}},postaction={decorate}]  (3.8,0.1) -- (2.2,0.1);
\draw [decoration={markings,mark=at position 1 with
    {\arrow[scale=1.2,>=stealth]{>}}},postaction={decorate}] (2.2,-0.1) -- (3.8,-0.1) ;
\draw [decoration={markings,mark=at position 1 with
    {\arrow[scale=1.2,>=stealth]{>}}},postaction={decorate}]  (5.8,0.1) -- (4.2,0.1) ;
\draw [decoration={markings,mark=at position 1 with
    {\arrow[scale=1.2,>=stealth]{>}}},postaction={decorate}]  (4.2,-0.1) -- (5.8,-0.1) ;
\draw [decoration={markings,mark=at position 1 with
    {\arrow[scale=1.2,>=stealth]{>}}},postaction={decorate}] (2.1,-1.8) -- (2.1,-0.2);
\draw [decoration={markings,mark=at position 1 with
    {\arrow[scale=1.2,>=stealth]{>}}},postaction={decorate}] (1.9,-0.2)  --  (1.9,-1.8);
\draw [decoration={markings,mark=at position 1 with
    {\arrow[scale=1.2,>=stealth]{>}}},postaction={decorate}] (2.1,-3.8) -- (2.1,-2.2);
\draw [decoration={markings,mark=at position 1 with
    {\arrow[scale=1.2,>=stealth]{>}}},postaction={decorate}] (1.9,-2.2)  --  (1.9,-3.8);
\end{tikzpicture}
\end{center}

The dimension vector of this quiver is $(1,1,1,2,2,2,3)$ and corresponds to the coordinates of the highest root of the root system of type $E_6$. The orientation is fixed so that all $\varphi^a$'s are positive and all $\varphi^b$'s are negative, which implies that the orientation is preserved by $\Omega$. By definition 
\vspace{-\topsep}
\begin{center}
 \begin{tabular}[h]{l}
 \scalebox{0.9}{$Z=\{ (\mu_0,\mu_1,\mu_2 ,\mu_3\mathrm{Id}_2,\mu_4\mathrm{Id}_2,\mu_5\mathrm{Id}_2,\mu_6\mathrm{Id}_3) \ | \ \mu_i \in \cc \text{ and }\mu_0 + \mu_1 + \mu_2 + 2(\mu_3 + \mu_4 + \mu_5) + 3\mu_6=0\}$} and \\
 \scalebox{0.95}{$\mu_{CS}(\varphi)= (-\varphi_0^b\varphi_0^a,-\varphi_1^b\varphi_1^a,-\varphi_2^b\varphi_2^a,\varphi_0^a\varphi_0^b - \varphi_3^b\varphi_3^a,\varphi_1^a\varphi_1^b - \varphi_4^b\varphi_4^a,\varphi_2^a\varphi_2^b - \varphi_5^b\varphi_5^a,\varphi_3^a\varphi_3^b+\varphi_4^a\varphi_4^b+\varphi_5^a\varphi_5^b)$}.
 \end{tabular}
\end{center}
Like in Section~\ref{C3}, the aim is to compute the traces of the oriented cycles of the quiver and find those which are linearly independent. Set 
\begin{center}
$\Phi(i_1i_2 \ldots i_k)=\mathrm{Tr}((\varphi_{i_1}^a\varphi_{i_1}^b)(\varphi_{i_2}^a\varphi_{i_2}^b)\ldots (\varphi_{i_k}^a\varphi_{i_k}^b))$ 
\end{center}
for $i_1,i_2, \ldots, i_k \in \{3,4,5\}$. Through lengthy computation, one can find that there are only three linearly independent elements, and they are of degree 3, 4 and 6 respectively. We choose $x =\Phi(4^25)$, $y =\Phi(4^25^2)$ and $z =\Phi(3^24^25^2)$. These are the generators of the ring of invariants. However, a unique relation has to link these generators as the latter ring is the coordinate ring of the pullback of the semiuniversal deformation of a simple singularity, and such a deformation is defined by a single equation (cf. \cite{KasSchle72}). The relation has the following form:

\begin{center}
$\renewcommand{\arraystretch}{1.5} \begin{array}[h]{rcl}
x^2z + y^3 + z^2 & = & a_{x^3}x^3 + a_{x^2y}x^2y + a_{xy^2}xy^2 + a_{x^2}x^2 + a_{xy}xy + a_{xz}xz  \\
 & & + a_{y^2}y^2 + a_{yz}yz + a_xx + a_y y + a_z z + a_0,
 \end{array}$
\end{center}
 with $a_0, a_x, a_y, a_z, a_{xz}, a_{yz}, a_{xy}, a_{x^2}, a_{x^3}, a_{y^2}, a_{x^2y}$ and $a_{xy^2}$ homogeneous polynomials in $\mu_1, \ldots, \mu_6$. The sheer size of these expressions is too much for them to be listed in this paper. Indeed, all together they contain 6147 terms.  

In order for the equation of the deformation to have the form predicted by Theorem~\ref{Kas-Schlessinger}, it requires a change of variables. Replacing $(x,y,z)$ by $(x + \alpha, y + \beta x + \gamma, z  + \delta x +  \epsilon y  + k x^2 + \mu)$ with
\begin{center}
$\raisebox{-.10\height}{$\left\lbrace \renewcommand{\arraystretch}{1.5} \begin{array}[h]{l}
k =-\frac{1}{2},\\
\beta =  \frac{1}{3}a_{xy^2},\\
\epsilon = \frac{1}{2}a_{yz},\\
\alpha =  -a_{x^3}- \frac{1}{3}a_{x^2y}a_{xy^2}- \frac{2}{27}a_{xy^2}^3 +  \frac{1}{2}a_{xz} +  \frac{1}{6}a_{yz}a_{xy^2},\\
\gamma =  -a_{xy^2}a_{x^3}- \frac{1}{3}a_{xy^2}^2a_{x^2y}- \frac{2}{27}a_{xy^2}^4 +  \frac{1}{2}a_{xy^2}a_{xz} +  \frac{1}{6}a_{xy^2}^2a_{yz} + a_{y^2} +  \frac{1}{4}a_{yz}^2,\\
\delta = \frac{1}{2}(a_{xz}-2\alpha + \frac{1}{3}a_{yz}a_{xy^2}),\\
\mu = \frac{1}{2}(a_{xz}\alpha + a_{yz}\gamma-\alpha^2 + a_z),
\end{array}\right.$}$
\end{center}
transforms the equation into
\begin{center}
$-\frac{1}{4}x^4 + y^3 + z^2 + A_{x^2y}x^2y + A_{x^2}x^2 + A_{xy}xy + A_xx + A_y y + A_0=0$ 
\end{center}
with new coefficients $A_{x^2y}$, $A_{x^2}$, $A_{xy}$, $A_x$, $A_y$ and $A_0$.

As in subsection~\ref{C3}, we use the isomorphism $\tau:Z \xrightarrow{\cong} \mathfrak{h}$ to identify an element $\mu \in Z$ with $\sum_{i=1}^6-\mu_i\Lambda_i^{\vee} \in \mathfrak{h}$, with $(\Lambda_i^{\vee})_{1 \leq i \leq 6}$ being the fundamental coweights of the root system of type $E_6$. The Weyl group $W$ of type $E_6$ is generated by the $r_{\alpha_1^{\vee}}, \ldots, r_{\alpha_6^{\vee}}$, and their actions on $(\mu_i)_{1 \leq i \leq 6}$ can be computed. Using a computer software (Maple 2015), one can explicitly verify that the coefficients $A_0$, $A_x$, $A_y$, $A_{x^2}$, $A_{xy}$, $A_{x^2y}$ are invariant by the $r_{\alpha_j^\vee}$'s and are thus $W$-invariant. The next step is to express them using the flat coordinates defined in Subsection~\ref{sub:flatE6}. 

According to our indexation of the vertices of the Dynkin diagram, the expressions of the fundamental weights are given below (cf. \cite{Bou68}):
\begin{center}
$ \renewcommand{\arraystretch}{1.5} \left\lbrace \begin{array}[h]{l}
\Lambda_1=\frac{1}{3}(4\alpha_1 + 2\alpha_2 + 3\alpha_3 + 5\alpha_4 + 4\alpha_5 + 6\alpha_6), \\
\Lambda_2=\frac{1}{3}(2\alpha_1 + 4\alpha_2 + 3\alpha_3 + 4\alpha_4 + 5\alpha_5 + 6\alpha_6), \\
\Lambda_3=\alpha_1 + \alpha_2 + 2\alpha_3 + 2\alpha_4 + 2\alpha_5 + 3\alpha_6, \\
\Lambda_4=\frac{1}{3}(5\alpha_1 + 4\alpha_2 + 6\alpha_3 + 10\alpha_4 + 8\alpha_5 + 12\alpha_6), \\
\Lambda_5=\frac{1}{3}(4\alpha_1 + 5\alpha_2 + 6\alpha_3 + 8\alpha_4 + 10\alpha_5 + 12\alpha_6), \\
\Lambda_6=2\alpha_1 + 2\alpha_2 + 3\alpha_3 + 4\alpha_4 + 5\alpha_5 + 6\alpha_6.
\end{array}\right.$
\end{center} 
Therefore it follows that 

\begin{center}
$ \renewcommand{\arraystretch}{1.5} 
\left\lbrace \begin{array}[h]{l}
\Lambda_1=(-\frac{\sqrt{6}}{6},\frac{\sqrt{2}}{2},\frac{\sqrt{6}}{6},-\frac{\sqrt{2}}{2},0,0)^T, \\
\Lambda_2=(0,0,\frac{\sqrt{6}}{6},-\frac{\sqrt{2}}{2},-\frac{\sqrt{6}}{6},\frac{\sqrt{2}}{2})^T, \\ 
\Lambda_3=(0,0,0,-\sqrt{2},0,0)^T, \\
\Lambda_4=(-\frac{\sqrt{6}}{3},0,\frac{\sqrt{6}}{3},-\sqrt{2},0,0)^T, \\
\Lambda_5=(0,0,\frac{\sqrt{6}}{3},-\sqrt{2},-\frac{\sqrt{6}}{3},0)^T, \\
\Lambda_6=(0,0,\frac{\sqrt{6}}{2},-\frac{3}{\sqrt{2}},0,0)^T,
\end{array}\right.$
\end{center}
using notations from Subsection~\ref{sub:flatE6}. Writing $(x_1,y_1, \ldots , y_3)^T=\sum_{i=1}^6\lambda_i \Lambda_i $, one finds
\begin{center}
$ \renewcommand{\arraystretch}{1.5}  \left\lbrace \begin{array}[h]{l}
x_1=-\frac{\sqrt{6}}{6}\lambda_1-\frac{\sqrt{6}}{3}\lambda_4, \\
y_1=\frac{\sqrt{2}}{2}\lambda_1, \\
x_2=\frac{\sqrt{6}}{6}\lambda_1 + \frac{\sqrt{6}}{6}\lambda_2 + \frac{\sqrt{6}}{3}\lambda_4 + \frac{\sqrt{6}}{3}\lambda_5 + \frac{\sqrt{6}}{2}\lambda_6, \\
y_2=-\frac{\sqrt{2}}{2}\lambda_1-\frac{\sqrt{2}}{2}\lambda_2-\sqrt{2}\lambda_3-\sqrt{2}\lambda_4-\sqrt{2}\lambda_5-\frac{3}{\sqrt{2}}\lambda_6, \\
x_3=-\frac{\sqrt{6}}{6}\lambda_2-\frac{\sqrt{6}}{3}\lambda_5, \\
y_3=\frac{\sqrt{2}}{2}\lambda_2.
\end{array}\right.$
\end{center} 
The isomorphism $\tau$ implies that $\mu_i=-\lambda_i$ for any $1 \leq i \leq 6$, and thus the coefficients of the equation can be expressed with the flat coordinates:
\begin{center}
$\raisebox{-.6\height}{$\left\lbrace \renewcommand{\arraystretch}{1.5}  \begin{array}[h]{l}
A_0= \frac{1}{576}(\psi_{12}(\mu)-\frac{1}{8}\psi_8(\mu)\psi_2(\mu)^2-\frac{1}{8}\psi_6(\mu)^2 + \frac{1}{96}\psi_6(\mu)\psi_2(\mu)^3-\psi_5(\mu)^2\psi_2(\mu)),\\
A_x= \frac{\sqrt{6}}{144}(-\psi_9(\mu) + \frac{1}{4}\psi_5(\mu)\psi_2(\mu)^2),\\
A_y= \frac{1}{48}(-\psi_8(\mu) + \frac{1}{4}\psi_6(\mu)\psi_2(\mu)-\frac{1}{192}\psi_2(\mu)^4),\\
A_{x^2}= \frac{1}{48}(\psi_6(\mu)-\frac{1}{8}\psi_2(\mu)^3),\\
A_{xy}= \frac{1}{2\sqrt{6}}\psi_5(\mu), \\
A_{x^2y}= -\frac{1}{4}\psi_2(\mu).
\end{array}\right.$}$
\end{center}

\noindent Set \\
\scalebox{0.9}{\begin{tabular}[t]{ll}
$X_\Gamma  =  \{(x, y, z, t_2,t_5,t_6,t_8,t_9,t_{12}) \in \cc^3 \times \mathfrak{h}/W \ | $   & \kern-1em $-\frac{1}{4}x^4 + y^3 + z^2-\frac{1}{4}t_2x^2y + \frac{1}{2\sqrt{6}}t_5xy + \frac{1}{48}(t_6-\frac{1}{8}t_2^3)x^2$   \\
 & \kern-1em $+ \frac{1}{48}(-t_8 + \frac{1}{4}t_6t_2-\frac{1}{192}t_2^4)y + \frac{\sqrt{6}}{144}(-t_9 + \frac{1}{4}t_5t_2^2)x $ \\
  &  \kern-1em $ + \frac{1}{576}(t_{12}-\frac{1}{8}t_8t_2^2-\frac{1}{8}t_6^2 + \frac{1}{96}t_6t_2^3-t_5^2t_2)=0\}$.
\end{tabular}}

\noindent We have 
\begin{center}
\begin{tabular}{l}
\begin{tikzpicture}[scale=1,  transform shape, baseline=-0.5ex]
\node (1) at ( 0,0) {$X_\Gamma$};
\node (2) at ( 4,0) {$ \mathfrak{h}/W$};
\node (3) at ( 4,-0.4)  {$ (t_2,t_5,t_6,t_8,t_9,t_{12})$};
\node (4) at ( 0,-0.4)  {$(x, y, z, t_2,t_5,t_6,t_8,t_9,t_{12})$};

\node (9) at (-2.2,0) {$\alpha:$};

\node (5) at (1.9,0) {};
\node (6) at (2.65,0) {};
\node (7) at (1.9,-0.4) {};
\node (8) at (2.65,-0.4) {};

\draw [decoration={markings,mark=at position 1 with
    {\arrow[scale=1.2,>=stealth]{>}}},postaction={decorate}] (5)  --  (6);
\draw [|-,decoration={markings,mark=at position 1 with
    {\arrow[scale=1.2,>=stealth]{>}}},postaction={decorate}] (7)  --  (8);
\end{tikzpicture} and \\
\begin{tikzpicture}[scale=1,  transform shape, baseline=-0.5ex]
\node (1) at ( 0,0) {$\mathfrak{h}$};
\node (2) at ( 4,0) {$\mathfrak{h}/W$};
\node (3) at ( 4,-0.4)  {$(\psi_2(\mu),\psi_5(\mu), \ldots ,\psi_{12}(\mu))$};
\node (4) at ( 0,-0.4)  {$ (\mu_1,\ldots ,\mu_6)$};
\node (10) at (6,0)  {.};

\node (9) at (-1.2,0) {$\pi:$};

\node (5) at (1,0) {};
\node (6) at (2,0) {};
\node (7) at (1,-0.4) {};
\node (8) at (2,-0.4) {};

\draw [decoration={markings,mark=at position 1 with
    {\arrow[scale=1.2,>=stealth]{>}}},postaction={decorate}] (5)  --  (6);
\draw [|-,decoration={markings,mark=at position 1 with
    {\arrow[scale=1.2,>=stealth]{>}}},postaction={decorate}] (7)  --  (8);
\end{tikzpicture}
\end{tabular}
\end{center}
As $Z \cong \mathfrak{h}$ and $(\mu_i)_i$ are coordinates on $Z$, we use the same coordinates on $\mathfrak{h}$. It follows that
\vspace{-\topsep}
\begin{center}
\resizebox{\textwidth}{!}{$ \renewcommand{\arraystretch}{1.5}  \begin{array}{r l l}
X_\Gamma \times_{\mathfrak{h}/W} \mathfrak{h}  = & \{((x, y, z, t_2,t_5,t_6,t_8,t_9,t_{12}),(\mu_1,\ldots,\mu_6)) \in X_\Gamma \times \mathfrak{h} \ | \ \alpha(x, y, z, t_i)=\pi(\mu_1,\ldots,\mu_6) \}, \\
  = & \begin{array}[t]{ll} 
          \{(x, y, z, t_i,\mu_j) \ | & t_2=\psi_2(\mu), t_5=\psi_5(\mu), t_6=\psi_6(\mu), t_8=\psi_8(\mu),t_9=\psi_9(\mu), \\
              &  t_{12}=\psi_{12}(\mu), \text{ and } -\frac{1}{4}x^4 + y^3 + z^2-\frac{1}{4}t_2x^2y + \frac{1}{2\sqrt{6}}t_5xy  \\
              &  + \frac{1}{48}(t_6-\frac{1}{8}t_2^3)x^2 + \frac{1}{48}(-t_8 + \frac{1}{4}t_6t_2-\frac{1}{192}t_2^4)y + \frac{\sqrt{6}}{144}(-t_9 + \frac{1}{4}t_5t_2^2)x \\
              & + \frac{1}{576}(t_{12}-\frac{1}{8}t_8t_2^2-\frac{1}{8}t_6^2 + \frac{1}{96}t_6t_2^3-t_5^2t_2)=0\},
              \end{array} \\
  = &  \mu_{CS}^{-1}(Z)//G(\Gamma). 
\end{array}$}
\end{center}
Finally we obtain
  
\begin{center}
\begin{tikzpicture}[scale=1,  transform shape]
\tikzstyle{point}=[circle,draw,fill]
\tikzstyle{ligne}=[thick]

\node (1) at ( 0,0) {$X_\Gamma \times_{\mathfrak{h}/W}\mathfrak{h}$};
\node (2) at ( 2,0) {$X_\Gamma$};
\node (3) at ( 2,-2)  {$\mathfrak{h}/W$};
\node (4) at ( 0,-2)  {$\mathfrak{h}$};

\node (9) at ( 1,-1) {$\circlearrowright$};

\draw [decoration={markings,mark=at position 1 with
    {\arrow[scale=1.2,>=stealth]{>}}},postaction={decorate}] (1)  -- node[above] {$\Psi$} (2);
\draw [decoration={markings,mark=at position 1 with
    {\arrow[scale=1.2,>=stealth]{>}}},postaction={decorate}] (2)  -- node[right] {$\alpha$} (3);
\draw [decoration={markings,mark=at position 1 with
    {\arrow[scale=1.2,>=stealth]{>}}},postaction={decorate}] (1)  -- node[left] {$\widetilde{\alpha}$} (4);
\draw [decoration={markings,mark=at position 1 with
    {\arrow[scale=1.2,>=stealth]{>}}},postaction={decorate}] (4)  -- node[below] {$\pi$} (3);
\end{tikzpicture}
\end{center}
with $\alpha$ being a semiuniversal deformation of \scalebox{0.97}{$X_{\Gamma,0}  =\alpha^{-1}(0)=\{(x, y, z) \in \cc^3 \ | \ -\frac{1}{4}x^4 + y^3 + z^2=0 \}$}, which is a simple singularity of type $E_6$. Like in previous sections, it can be checked that this diagram is $\Omega$-equivariant with the natural action on the singularity. It follows that the restriction $\alpha^{\Omega}:X_{\Gamma, \Omega} \rightarrow (\mathfrak{h}/W)^\Omega$ of $\alpha$ to $X_{\Gamma,\Omega}=\alpha^{-1}((\mathfrak{h}/W)^\Omega)$ is $\Omega$-invariant and $(\alpha^{\Omega})^{-1}(0)=X_{\Gamma,0}$. This implies that $\alpha^{\Omega}$ is a semiuniversal deformation of an inhomogeneous simple singularity of type $(E_{6},\mathbb{Z}/2\mathbb{Z})=F_4$.

\section{Quotient of the semiuniversal deformation of \texorpdfstring{$(\cc^2/\Gamma, \Omega)$}{Lg} by \texorpdfstring{$\Omega$}{Lg}}\label{sec:QuotientInhomogeneousDeformations}

\subsection{Objectives}

It was shown in the previous section that the restriction $\alpha^{\Omega}:X_{\Gamma,\Omega} \rightarrow (\mathfrak{h}/W)^{\Omega}$ over the fixed points $(\mathfrak{h}/W)^{\Omega}$ of a semiuniversal deformation of the simple singularity $\cc^2/\Gamma$ is a semiuniversal deformation of the inhomogeneous simple singularity of type $\Delta(\Gamma,\Gamma')$, and is thus $\Omega$-invariant. Hence $\Omega$ acts on each fibre of $\alpha^{\Omega}$ and the fibres can be quotiented. By construction, the special fibre of $\alpha^{\Omega}$ is $(\alpha^{\Omega})^{-1}(0) = X_{\Gamma,0}=\cc^2/\Gamma$. Therefore the fibre above the origin of the quotient map is $(\alpha^{\Omega})^{-1}(0)/\Omega=X_{\Gamma,0}/\Omega \cong (\cc^2/\Gamma)/(\Gamma'/\Gamma) \cong \cc^2/\Gamma'$, which is also a simple singularity. It follows that the family given by the quotient map $\overline{\alpha^{\Omega}}:X_{\Gamma,\Omega} /\Omega \rightarrow (\mathfrak{h}/W)^{\Omega}$ is a deformation of the simple singularity $\cc^2/\Gamma'$. 

In Proposition~\ref{propsubDynkin} P. Slodowy described the singular configurations of the fibres of a deformation of a simple singularity of type $\Delta(\Gamma)$ around the special fibre $\cc^2/\Gamma$ in terms of subdiagrams of the Dynkin diagram of type $\Delta(\Gamma)$. Furthermore, he gave a relation mentioned in Remark~\ref{h0andh1} between $(\mathfrak{h}/W)^{\Omega}$ and the quotient $\mathfrak{h}_0/W_0$. This raises two questions:

\begin{enumerate}
\item How is the map $\overline{\alpha^{\Omega}}$ related to the semiuniversal deformation of the simple singularity of type $\cc^2/\Gamma'$, and how to describe the base space $(\mathfrak{h}/W)^\Omega$ using a Cartan subalgebra and a Weyl group of type $\Delta(\Gamma')$? 
\item Can one describe the singularities in the fibres of $\overline{\alpha^{\Omega}}$ around the special fibre $\cc^2/\Gamma'$ in terms of sub-root systems of the root system of type $\Delta(\Gamma')$, like P. Slodowy did for the homogeneous case?
\end{enumerate}

The content of Section~\ref{sec:QuotientInhomogeneousDeformations} is only a first step in the search for answers to the previous questions. In order to investigate the nature and the regularity of the fibres of the quotient of a semiuniversal deformation of a simple singularity of inhomogeneous type, we use the explicit realisation of the map $\alpha^{\Omega}:X_{\Gamma,\Omega} \rightarrow (\mathfrak{h}/W)^{\Omega}$ obtained in the Section~\ref{sec:computations}. The quotient map $\overline{\alpha^{\Omega}}$ is computed for the cases $(A_{2r-1},\mathbb{Z}/2\mathbb{Z})$ ($r \geq 2$), $(D_4,\mathbb{Z}/2\mathbb{Z})$, $(D_4, \mathfrak{S}_3)$ and $(E_6,\mathbb{Z}/2\mathbb{Z})$. In particular, when the original singularity $\cc^2/\Gamma$ is of type $A_3$ or $D_4$, one notices that every fibre of the map $\overline{\alpha^{\Omega}}$ is singular.

\subsection{Case \texorpdfstring{$(A_{2r-1},\mathbb{Z}/2\mathbb{Z})$}{Lg}}\label{quotientBr}
Using the same notations as in Subsection~\ref{defA2r-1}, a semiuniversal deformation of a simple singularity of type $(A_{2r-1},\mathbb{Z}/2\mathbb{Z})$ is the projection $\alpha^\Omega:X_{\Gamma,\Omega} \rightarrow (\mathfrak{h}/W)^\Omega$ with
\vspace{-\topsep}
\begin{center}
 $X_{\Gamma, \Omega}=\{ (x, y, z, t_2,0,t_4,0,\ldots ,t_{2r}) \in \cc^3 \times \mathfrak{h}/W \ | \ \displaystyle z^{2r}+\sum_{i=1}^{r}f_{2i}(t_2, \ldots, t_{2r}) z^{2(r-i)}=xy\}$.
 \end{center}
\vspace{-\topsep}
The purpose of this subsection is to compute the quotient morphism $\overline{\alpha^\Omega}:X_{\Gamma,\Omega}/\Omega \rightarrow (\mathfrak{h}/W)^\Omega$ as well as study its regularity. 

\subsubsection{The morphism $\overline{\alpha^\Omega}$}

The action of $\Omega$ on $X_{\Gamma, \Omega}$ is $\sigma.(x, y, z,t_2,0,t_4,\ldots ,0,t_{2r})=((-1)^ry,(-1)^rx,-z,t_2,0,t_4,\ldots ,0,t_{2r})$. Let us compute the $\Omega$-invariant subring of the coordinate ring of $X_{\Gamma, \Omega}$. 

Set $p \in \cc[X_{\Gamma, \Omega}]$. Then $p=p(x, y, z,t_2,t_4,..,t_{2r})$ and $\sigma.p=p((-1)^ry,(-1)^rx,-z,t_2,t_4, \ldots ,t_{2r})$. As the $t_{2k}$ ($1 \leq k \leq r$) are invariant by the action of $\Omega$, they can be treated as constants and will be omitted in order to simplify the notations. We write $p=\sum_{k=0}^np_k(x,y)z^k \in \cc[x,y][z]$ and assume that $p$ is $\Omega$-invariant. Then by identifying the coefficients of the powers of $z$ in $\sigma.p$ and $p$, one obtains $p_i((-1)^ry,(-1)^rx)=(-1)^ip_i(x,y)$ for all $0 \leq i \leq n$. 

Assume $r$ to be even. Then $p_i$ is symmetric if $i$ is even and skew-symmetric otherwise. The polynomial $p$ can be rewritten in the following form 
\begin{center}
$p=\underbrace{p_0(x,y)}_{\text{symmetric}} + \underbrace{p_2(x,y)}_{\text{symmetric}}z^2 + \ldots  + \underbrace{p_1(x,y)}_{\text{skew-symmetric}}z+ \underbrace{p_3(x,y)}_{\text{skew-symmetric}}z^3 + \ldots $ \quad .
\end{center}
However, a skew-symmetric polynomial in two variables $(x,y)$ can be written as a product between $(x-y)$ and some symmetric polynomial. Hence for all $i$ odd, $p_i(x,y)=(x-y)\text{sym}_i(x,y)$ for some symmetric polynomial $\text{sym}_i$, and $p$ can be expressed as
\begin{center}
$p=\text{sym}_0(x,y) + z^2\text{sym}_2(x,y) + \ldots  + z(x-y)\text{sym}_1(x,y) + z^2z(x-y)\text{sym}_3(x,y) + \ldots $ \quad .
\end{center}
Therefore any $\Omega$-invariant polynomial is generated by $z^2,z(x-y)$ and symmetric polynomials in $(x,y)$, which are themselves generated by $(x + y)$ and $xy$. As the converse is trivial, we have proved
\begin{center}
\scalebox{0.95}{$\cc[X_{\Gamma, \Omega}]^{\Omega} = \cc[z^2,z(x-y),x + y, xy, t_2,\ldots ,t_{2r}]/(z^{2r} + f_2(t_2, \ldots, t_{2r})z^{2r-2} + \ldots f_{2r}(t_2, \ldots, t_{2r})=xy)$}.
\end{center}
By changing variables $X=x + y,Y=xy, Z=z^2$ and $W=iz(x-y)$, the quotient space can be written as
\begin{center}
$\renewcommand{\arraystretch}{1.5} \begin{array}[h]{rcl}
X_{\Gamma, \Omega}/\Omega & = & \{(X,Z,W,t_2,\ldots ,t_{2r}) \in \cc^3\times (\mathfrak{h}/W)^\Omega\ | \ Z(X^2-4Z^r) + W^2 \\
 & & -4 f_2(t_2, \ldots, t_{2r})Z^{r}-4 f_4(t_2, \ldots, t_{2r})Z^{r-1}-\ldots -4 f_{2r}(t_2, \ldots, t_{2r})Z=0\}.
\end{array}$
\end{center}
The quotient of the morphism $\alpha^{\Omega}$ is therefore 
\begin{center}
\begin{center}
\begin{tikzpicture}[scale=1,  transform shape, baseline=-0.5ex]
\node (1) at ( 0,0) {$X_{\Gamma, \Omega}/\Omega$};
\node (2) at ( 4.7,0) {$ (\mathfrak{h}/W)^\Omega$};
\node (3) at ( 4.7,-0.4)  {$(t_2,0,t_4,0,\ldots 0,t_{2r})$.};
\node (4) at ( 0,-0.4)  {$(X,Z,W,t_2,\ldots ,t_{2r})$};

\node (9) at (-2.3,0) {$\overline{\alpha^{\Omega}}:$};

\node (5) at (1.5,0) {};
\node (6) at (3.2,0) {};
\node (7) at (1.5,-0.4) {};
\node (8) at (3.2,-0.4) {};

\draw [decoration={markings,mark=at position 1 with
    {\arrow[scale=1.2,>=stealth]{>}}},postaction={decorate}] (5)  --  (6);
\draw [|-,decoration={markings,mark=at position 1 with
    {\arrow[scale=1.2,>=stealth]{>}}},postaction={decorate}] (7)  --  (8);
\end{tikzpicture}
\end{center}
\end{center}
It is a deformation of $(\overline{\alpha^{\Omega}})^{-1}(0)=\{(X,Z,W) \in \cc^3 \ | \ Z(X^2-4Z^r) + W^2=0\}$, which is a simple singularity of type $D_{r + 2}$. However, the deformation $\overline{\alpha^{\Omega}}$ is not semiuniversal. Indeed, according to Theorem~\ref{Brieskorn}, the base space of a semiuniversal deformation of a simple singularity of type $D_{r + 2}$ is $\mathfrak{h}_{D_{r + 2}}/W_{D_{r + 2}}$, which is of dimension $r + 2$. But here the base space is $(\mathfrak{h}/W)^\Omega$ and of dimension $r < r + 2$. Hence the deformation cannot be semiuniversal.

With the same type of arguments, a similar result can be obtained for $r$ odd.

\subsubsection{Example: regularity of the fibres of $\alpha^\Omega$ and $\overline{\alpha^\Omega}$ when $r=2$}

When $r=2$, the formulae are the following: 

$X_{\Gamma, \Omega}  =  \{ (x, y, z, t_2,t_4) \in \cc^3 \times \mathfrak{h}/W \ | \ \displaystyle z^{4}+f_2(t_2,t_4) z^{2}+f_4(t_2,t_4) =xy\}$, 

\scalebox{0.93}{$X_{\Gamma, \Omega}/\Omega  =  \{(X,Z,W,t_2,t_{4}) \in \cc^3 \times  (\mathfrak{h}/W)^\Omega \ | \ Z(X^2 - 4Z^2) + W^2  - 4f_2(t_2,t_4)Z^{2}  - 4f_4(t_2,t_4)Z=0 \}$}. 

\noindent For a fibre of $\alpha^\Omega$ to be singular, it requires $f_4(t_2,t_4)=0 \ \mathrm{ or } \ f_2(t_2,t_4)^2=4f_4(t_2,t_4)$.  

\begin{itemize}\setlength\itemsep{0.3pt}
\item If $f_4(t_2,t_4)=t_4+\frac{1}{8}t_2^2=0$ then $(t_2,t_4)=(t_2,-\frac{1}{8}t_2^2)$ and $(\alpha^\Omega)^{-1}(t_2,-\frac{1}{8}t_2^2)$ is singular at the origin.
\item If $f_2(t_2,t_4)^2=t_2^2=4f_4(t_2,t_4)=4(t_4+\frac{1}{8}t_2^2)$, then $(t_2,t_4)=(t_2,\frac{1}{8}t_2^2)$ and $(\alpha^\Omega)^{-1}(t_2,\frac{1}{8}t_2^2)$ is singular.
\end{itemize}

In any fibre $(\overline{\alpha^\Omega})^{-1}(t_2,t_4) \subset X_{\Gamma, \Omega}/\Omega$, the points $(X,W,Z)=(\pm 2\sqrt{f_4(t_2,t_4)},0,0)$ are singular. Hence all the fibres of $\overline{\alpha^\Omega}$ are singular and the following proposition is proved:

\begin{proposition}\label{A3singulier}
Every fibre of the deformation $\overline{\alpha^\Omega}:X_{\Gamma,\Omega}/\Omega \rightarrow (\mathfrak{h}/W)^\Omega$ of a simple singularity of type $D_4$ is singular.
\end{proposition}

\subsection{Case \texorpdfstring{$(D_{4},\mathbb{Z}/2\mathbb{Z})$}{Lg}}
According to Subsection~\ref{C3}, a semiuniversal deformation of a simple singularity of type $(D_{4},\mathbb{Z}/2\mathbb{Z})$ is the projection $\alpha^\Omega:X_{\Gamma,\Omega} \rightarrow (\mathfrak{h}/W)^\Omega$ with

\begin{center}
 $X_{\Gamma, \Omega}=\{(x, y, z, t_2,t_4,t_6,0) \in \cc^3 \times \mathfrak{h}/W \ | \  z^2=xy(x + y) -\frac{1}{2}t_2xy  -\frac{1}{4}t_4x + \frac{1}{4}(t_6+\frac{1}{6}t_2t_4+\frac{1}{108}t_2^3)\}$. 
 \end{center}
 
\subsubsection{The morphism $\overline{\alpha^\Omega}$}

The following relations were obtained in Subsection~\ref{C3}:
\begin{center}
 $\renewcommand{\arraystretch}{1.2} \raisebox{-.6\height}{$\left\lbrace \begin{array}[h]{lcl}
\sigma.x & = & x, \\
\sigma.y & = & -x-y + \frac{1}{2}t_2, \\
\sigma.z & = & -z,\\
\sigma.t_i & = & t_i, \ i=2,4,6.
\end{array}\right.$}$ 
\end{center}
We change variable by setting $y'=\frac{1}{2}x + y-\frac{1}{4}t_2$ and obtain  
\begin{center}
$ \renewcommand{\arraystretch}{1.5}  \begin{array}[h]{lcll}
X_{\Gamma, \Omega}  & = & \{(x,y',z, t_2, t_4, t_6, 0) \in \cc^3 \times \mathfrak{h}/W \ | \  &  z^2= \frac{1}{4}t_6+\frac{1}{24}t_2t_4 +\frac{1}{432}t_2^3 +(-\frac{1}{16}t_2^2-\frac{1}{4}t_4)x \\
 & & & -\frac{1}{4}x^3+\frac{1}{4}t_2x^2+x(y')^2\} ,
 \end{array}$
\end{center}
with the action on $y'$ being $\sigma.y'  =  -y'$. 

Set $p \in \cc[X_{\Gamma, \Omega}]^{\Omega}$. As $x$ and the $t_i$'s are not altered by $\sigma$, they will be omitted from the notation. The action gives $\sigma.p(y',z)=p(-y',-z)=p(y',z)$, and it is then clear that $p$ is a polynomial in $y'^2$, $z^2$ and $y'z$. By defining $X=x$, $Y=y'^2$, $Z=z^2$, $W=y'z$, and with the following substitution: $Y  \rightarrow Y+\frac{1}{8}X^2-\frac{1}{8}Xt_2+\frac{1}{32}t_2^2+\frac{1}{8}t_4$, the equation of $X_{\Gamma, \Omega}/\Omega$ becomes
\begin{center}
$-\frac{1}{64}X^5+XY^2-W^2+A_{X^4}X^4+A_{X^3}X^3+A_{X^2}X^2+A_{X}X+A_{Y}Y+A_0=0$ $(\star)$
\end{center}
with
\begin{center}
 $\raisebox{-.4\height}{$\left\lbrace  \renewcommand{\arraystretch}{1.3}  \begin{array}[h]{l}
 A_{X^4}= \frac{t_2}{32},\\
A_{X^3}=-\frac{3}{128}t_2^2-\frac{1}{32}t_4, \\
A_{X^2}=\frac{7}{192}t_2t_4+\frac{1}{32}t_6+\frac{7}{864}t_2^3, \\
A_{X}=-\frac{1}{32}t_6t_2-\frac{5}{384}t_2^2t_4-\frac{35}{27648}t_2^4-\frac{1}{64}t_4^2, \\
A_{Y}=\frac{1}{4}t_6+\frac{1}{24}t_2t_4+\frac{1}{432}t_2^3, \\
A_0=\frac{1}{128}t_6t_2^2+\frac{1}{32}t_6t_4+\frac{11}{6912}t_2^3t_4+\frac{1}{192}t_2t_4^2+\frac{1}{13824}t_2^5.
\end{array}\right.$}$ 
\end{center}

Hence a fibre of $\overline{\alpha^\Omega}$ in $X_{\Gamma, \Omega}/\Omega$ is defined by $(\star)$. One notices that it is a subfamily of the semiuniversal deformation of a simple singularity of type $D_6$. Therefore the projection 
\begin{center}
$\overline{\alpha^\Omega}:X_{\Gamma, \Omega}/\Omega \rightarrow (\mathfrak{h}/W)^{\Omega}$
\end{center}
is a deformation of a simple singularity of type $D_6$. However $\mathrm{dim}(\mathfrak{h}/W)^\Omega=3 < 6 =\mathrm{dim} \ \mathfrak{h}_{D_6}/W_{D_6}$ so $\overline{\alpha^\Omega}$ is not semiuniversal.

\subsubsection{The discriminant of $\overline{\alpha^\Omega}$}

Let us determine the discriminant of $\overline{\alpha^\Omega}$, i.e. the elements $(t_2,t_4,t_6) \in  (\mathfrak{h}/W)^{\Omega}$ such that the fibre $(\overline{\alpha^\Omega})^{-1}(t_2,t_4,t_6)$ is singular. The fibre $(\overline{\alpha^\Omega})^{-1}(t_2,t_4,t_6)$ is defined by the equation $f(X, Y, W)= 0$ of $(\star)$, and it is singular if and only if the following system has a solution:
\begin{center}
 $\raisebox{-.8\height}{$\left\lbrace \renewcommand{\arraystretch}{1.5}  \begin{array}[h]{cl}
f(X, Y, W) & = 0, \\
 \frac{\partial f}{\partial X}(X, Y, W) & = 0, \\
 \frac{\partial f}{\partial Y}(X, Y, W) & = 0, \\
 \frac{\partial f}{\partial W}(X, Y, W) & = 0.
\end{array}\right.$}$ 
\end{center}
Let $X_s$ be a solution of $108X^3-108X^2t_2+(108t_4+27t_2^2)X-t_2^3-18t_2t_4-108t_6=0$, which exists for any $(t_2,t_4,t_6) \in (\mathfrak{h}/W)^\Omega$ because the base field is algebraically closed, and then define $Y_s=-\frac{1}{32}(4X_s^2-4X_st_2+t_2^2+4t_4)$. It follows that the point $(X_s,Y_s,0)$ is a singular point of $(\overline{\alpha^\Omega})^{-1}(t_2,t_4,t_6)$. The following result is then proved:

\begin{proposition}\label{D4Zsur2Zsingulier}
Every fibre of the deformation $\overline{\alpha^\Omega}:X_{\Gamma,\Omega}/\Omega \rightarrow (\mathfrak{h}/W)^\Omega$ of a simple singularity of type $D_6$ is singular.
\end{proposition}

\subsection{Case \texorpdfstring{$(D_{4},\mathfrak{S}_3)$}{Lg}}
According to Subsection~\ref{G2}, a semiuniversal deformation of a simple singularity of type $(D_{4},\mathfrak{S}_3)$ is the projection $\alpha^\Omega:X_{\Gamma,\Omega} \rightarrow (\mathfrak{h}/W)^\Omega$ with
\begin{center}
$X_{\Gamma, \Omega}=\{(x, y, z, t_2,0,t_6,0) \in \cc^3 \times \mathfrak{h}/W \ | \  z^2=xy(x + y) -\frac{1}{2}t_2xy + \frac{1}{4}(t_6+\frac{1}{108}t_2^3)\}$.
\end{center}
 
\subsubsection{The morphism $\overline{\alpha^\Omega}$}
In Subsection~\ref{G2} we obtained 
\begin{center}
 $\left\lbrace  \renewcommand{\arraystretch}{1.2}  \begin{array}[h]{lcl}
\sigma.x & = & x,\\
\sigma.y & = & -x-y+\frac{1}{2}t_2,\\
\sigma.z & = & -z, \\
\sigma.t_i & = & t_i, \ i=2,6,
\end{array}\right.$ \quad and \quad  $\left\lbrace  \renewcommand{\arraystretch}{1.2} \begin{array}[h]{lcl}
\rho.x & = & y, \\
\rho.y & = & -x-y + \frac{1}{2}t_2, \\
\rho.z & = & z, \\
\rho.t_i & = & t_i, \ i=2,6.
\end{array}\right.$
\end{center}
Define $X=(-3-i\sqrt{3})x + (-3 + i\sqrt{3})y + t_2$ and $Y=(-3 + i\sqrt{3})x + (-3-i\sqrt{3})y + t_2$. The action of $\rho$ on the vector space generated by $(X, Y, z, t_2)$ becomes diagonal and so
\begin{center}
$ \renewcommand{\arraystretch}{1.3}  \begin{array}[t]{rcl} \cc[(\alpha^\Omega)^{-1}(t_2,t_6)]^{\mathbb{Z}/3\mathbb{Z}} & = & \cc[X^3,Y^3,XY,z,t_2,t_6]/( -z^2-\frac{1}{216}Y^3-\frac{1}{216}X^3-\frac{1}{432}t_2^3  +\frac{1}{72}XYt_2+\frac{1}{4}t_6),\\
 & = & \cc[\mathcal{X},\mathcal{Y},\mathcal{W},z, t_2, t_6]/(-z^2-\frac{1}{216}(\mathcal{X} + \mathcal{Y} + \frac{1}{2}t_2^3) + \frac{1}{72}\mathcal{W}t_2 + \frac{1}{4}t_6,\mathcal{X}\mathcal{Y}-\mathcal{W}^3).
\end{array}$
\end{center}
Set $\mathfrak{X} = \mathcal{X} + \mathcal{Y}$ and $\mathfrak{Y} = \mathcal{X}-\mathcal{Y}$. The invariant ring is 
\begin{center}
$\cc[(\alpha^\Omega)^{-1}(t_2,t_3)]^{\mathbb{Z}/3\mathbb{Z}} = \cc[\mathfrak{X},\mathfrak{Y},\mathcal{W},z, t_2, t_6]/(-z^2-\frac{1}{216}\mathfrak{X}-\frac{1}{432}t_2^3 + \frac{1}{72}\mathcal{W}t_2 + \frac{1}{4}t_6,\frac{1}{4}(\mathfrak{X}^2-\mathfrak{Y}^2)-\mathcal{W}^3)$,
\end{center}
with the action of $\mathbb{Z}/2\mathbb{Z}$ given by
\begin{center}
$\sigma.\mathfrak{X} = \mathfrak{X}$, $\sigma.\mathfrak{Y} = -\mathfrak{Y}$, $\sigma.\mathcal{W} = \mathcal{W}$, $\sigma.z=-z $ and $\sigma.t_i=t_i$ for $i=2,6$.
\end{center}
 Therefore \\
 $\cc[\alpha^{-1}(t_2,t_6)]^{\mathfrak{S}_3}  \renewcommand{\arraystretch}{1.3}   \begin{array}[t]{l} =((\cc[(\alpha^\Omega)^{-1}(t_2,t_6)])^{\mathbb{Z}/3\mathbb{Z}})^{\mathbb{Z}/2\mathbb{Z}}, \\
 =\scaleobj{0.95}{\cc[\mathfrak{X},\mathfrak{Y}^2,z^2,\mathfrak{Y}z,\mathcal{W},t_2,t_6]/(-z^2-\frac{1}{216}\mathfrak{X}-\frac{1}{432}t_2^3  + \frac{1}{72}\mathcal{W}t_2 + \frac{1}{4}t_6,  \frac{1}{4}(\mathfrak{X}^2-\mathfrak{Y}^2)-\mathcal{W}^3)}.
 \end{array}$ \\
After several other analytical changes of variables and substitutions, the equation of the quotient space turns into the form predicted by Theorem~\ref{Kas-Schlessinger}:
\begin{center}
\small{$X^3Y-11664Y^3+Z^2+(-\frac{11}{32}t_2^6-\frac{189}{4}t_2^3t_6-729t_6^2)Y+(-\frac{15}{16}t_2^4-81t_2t_6)XY+324t_2XY^2+(189t_2^3+5832t_6)Y^2=0$} $(\star \star)$
\end{center}
Hence the fibre of $\overline{\alpha^\Omega}$ above $(t_2,t_6)$ is defined by $(\star \star)$. It is a subfamily of the semiuniversal deformation of a simple singularity of type $E_7$. Therefore the projection 
\begin{center}
$\overline{\alpha^\Omega}:X_{\Gamma, \Omega}/\Omega \rightarrow (\mathfrak{h}/W)^{\Omega}$
\end{center}
is a deformation of a simple singularity of type $E_7$. However $\mathrm{dim}(\mathfrak{h}/W)^\Omega=2 < 7 =\mathrm{dim} \ \mathfrak{h}_{E_7}/W_{E_7}$ so $\overline{\alpha^\Omega}$ is not semiuniversal.

\subsubsection{The discriminant of $\overline{\alpha^\Omega}$}

The fibre $(\overline{\alpha^\Omega})^{-1}(t_2,t_6)$ is defined as the zero locus of the polynomial
\begin{center}
$\renewcommand{\arraystretch}{1.3} \begin{array}[h]{rl}
f(X, Y, Z) = & X^3Y-11664Y^3+Z^2+(-\frac{11}{32}t_2^6-\frac{189}{4}t_2^3t_6-729t_6^2)Y+(-\frac{15}{16}t_2^4-81t_2t_6)XY \\ & +324t_2XY^2+(189t_2^3+5832t_6)Y^2.
\end{array}$
\end{center}
The fibre is singular if and only if the following system has a solution:
\begin{center}
 $\raisebox{-.8\height}{$\left\lbrace \renewcommand{\arraystretch}{1.5}  \begin{array}[h]{rcl}
f(X, Y, Z) & = & 0, \\
 \frac{\partial f}{\partial X}(X, Y, Z) & = & 0, \\
 \frac{\partial f}{\partial Y}(X, Y, Z) & = & 0, \\
\frac{\partial f}{\partial Z}(X, Y, Z) & = & 0.
\end{array}\right.$}$ 
\end{center}
If $Y=Z=0$, this system is equivalent to the following equation:
\begin{center}
$X^3+(-\frac{15}{16}t_2^4-81t_2t_6)X-\frac{11}{32}t_2^6-\frac{189}{4}t_2^3t_6-729t_6^2=0$.
\end{center}
It is a polynomial of degree 3, which always has a solution $X_s$ because the base field is algebraically closed. Therefore $(X_s,0,0)$ is singular, and the proof of the next proposition is achieved.

\begin{proposition}\label{D4S3singulier}
Every fibre of the deformation $\overline{\alpha^\Omega}:X_{\Gamma,\Omega}/\Omega \rightarrow (\mathfrak{h}/W)^\Omega$ of a simple singularity of type $E_7$ is singular.
\end{proposition}

\subsection{Case \texorpdfstring{$(E_{6},\mathbb{Z}/2\mathbb{Z})$}{Lg}}\label{quotientF4}
According to Subsection~\ref{E6}, a semiuniversal deformation of a simple singularity of type $(E_{6},\mathbb{Z}/2\mathbb{Z})$ is given by the projection $\alpha^\Omega:X_{\Gamma,\Omega} \rightarrow (\mathfrak{h}/W)^\Omega$ with 
\vspace{-\topsep}
\begin{center}
$\renewcommand{\arraystretch}{1.5}  \begin{array}[h]{rcl}
X_{\Gamma, \Omega} & = & \{(X, Y, Z, t_2,0,t_6,t_8,0,t_{12}) \in \cc^3 \times \mathfrak{h}/W \ | \  -\frac{1}{4}X^4 + Y^3 + Z^2-\frac{1}{4}t_2X^2Y + \frac{1}{48}(t_6-\frac{1}{8}t_2^3)X^2 \\
 & &  + \frac{1}{48}(-t_8 + \frac{1}{4}t_6t_2-\frac{1}{192}t_2^4)Y  + \frac{1}{576}(t_{12}-\frac{1}{8}t_8t_2^2-\frac{1}{8}t_6^2 + \frac{1}{96}t_6t_2^3)=0\}. 
\end{array}$
\end{center} 
This morphism is $\Omega$-invariant and $\Omega$ acts on each of its fibres. The action of $\Omega$ on $X_{\Gamma, \Omega}$ is \begin{center}
 $\raisebox{-.8\height}{$\left\lbrace \renewcommand{\arraystretch}{1.2} \begin{array}[h]{l}
\sigma.X=-X, \\
\sigma.Y=Y,\\
\sigma.Z=-Z,\\
\sigma.t_i=t_i, \ i=2,6,8,12.
\end{array}\right.$}$ 
\end{center}
After analytically changing and substituting the variables, the equation of the quotient fibre becomes
\begin{center}
$\renewcommand{\arraystretch}{1.3} \begin{array}[h]{c} - \frac{1}{4}X^3 + XY^3 +Z^2-\frac{1}{4}t_2X^2Y + \frac{1}{48}(t_6-\frac{1}{8}t_2^3)X^2 + \frac{1}{48}(-t_8 + \frac{1}{4}t_6t_2-\frac{1}{192}t_2^4)XY  \\
+ \frac{1}{576}(t_{12}-\frac{1}{8}t_8t_2^2-\frac{1}{8}t_6^2 + \frac{1}{96}t_6t_2^3)X=0. \ (\star \star \star)
\end{array}$
\end{center}
Therefore the fibre of $\overline{\alpha^\Omega}$ above $(t_2,t_6,t_8,t_{12})$ is defined by $(\star \star \star)$. It is a subfamily of the semiuniversal deformation of a simple singularity of type $E_7$. It follows that the projection 
\begin{center}
$\overline{\alpha^\Omega}:X_{\Gamma, \Omega}/\Omega \rightarrow (\mathfrak{h}/W)^{\Omega}$
\end{center}
is a deformation of a simple singularity of type $E_7$. However $\mathrm{dim}(\mathfrak{h}/W)^\Omega=4 < 7 =\mathrm{dim} \ \mathfrak{h}_{E_7}/W_{E_7}$ so $\overline{\alpha^\Omega}$ is not semiuniversal.

\section{Perspectives}

Starting from finite subgroups $\Gamma  \lhd \Gamma'$ of $\mathrm{SU}_2$, the procedure by H. Cassens and P. Slodowy allows the construction of a semiuniversal deformation $\alpha:X_\Gamma \rightarrow \mathfrak{h}/W$ of the simple singularity $\cc^2/\Gamma$, and its restriction above the $\Omega$-fixed points of the base space gives a semiuniversal deformation $\alpha^\Omega:X_{\Gamma,\Omega} \rightarrow (\mathfrak{h}/W)^\Omega$ of the simple inhomogeneous singularity $(\cc^2/\Gamma,\Omega)$ with discriminant $\mathbb{D}_\Omega$. 

Let $\mathfrak{g}$ be a simply laced simple Lie algebra with root system $\Phi$. Denote by $\chi$ the adjoint quotient of $\mathfrak{g}$ and by $S_e$ a Slodowy slice to a subregular nilpotent element $e$ of $\mathfrak{g}$. Finally, let  $\mathbb{D}$ be the discriminant of $\restr{\chi}{S_e}$ and, for any root $\alpha \in \Phi$, let $H_{\alpha}$ be the kernel of $\alpha$. The situation is illustrated below:
\vspace{-\topsep}
\begin{center}
\begin{tikzpicture}[scale=0.9,  transform shape]
\node (1) at ( 2,2) {$S_e$};

\node (3) at (2,0) {$\mathfrak{h}/W$};
\node (4) at (2,-0.4) {$\bigcup$};
\node (5) at (2,-0.8) {$\mathbb{D}$};

\node (7) at (0,0) {$\mathfrak{h}$};
\node (8) at (0,-0.4) {$\bigcup$};
\node (9) at (0,-0.8) {$\bigcup_{\alpha \in \Phi^+}H_\alpha$};

\node(10) at (0.7,0) {};
\node(11) at (1.7,0) {};
\node(12) at (0.7,-0.8) {};
\node(13) at (1.7,-0.8) {};

\draw  [decoration={markings,mark=at position 1 with
    {\arrow[scale=1.2,>=stealth]{>}}},postaction={decorate}] (1)  -- node[right] {$\restr{\chi}{S_e}$} (3);
\draw  [decoration={markings,mark=at position 1 with
    {\arrow[scale=1.2,>=stealth]{>}}},postaction={decorate}] (10)  -- node[above] {$\pi$} (11);
\draw  [|-,decoration={markings,mark=at position 1 with
    {\arrow[scale=1.2,>=stealth]{>}}},postaction={decorate}] (12)  --  (13);
\end{tikzpicture}.
\end{center}
\vspace{-\topsep}
In \cite{Slo80}, P. Slodowy proved that the singularities appearing in $S_e$ above a point $\pi(h) \in \mathbb{D}$ can be determined by the hyperplanes arrangement of the $H_\alpha$'s containing $h$.

In Remark~\ref{h0andh1}, we mentioned that by setting $\mathfrak{h}_1=\mathfrak{h}^{\Omega}$ and $W_1=\{w \in W \ | \ w\gamma=\gamma w, \ \forall \gamma \in \Omega\}$, the natural map $\mathfrak{h}_1 \rightarrow (\mathfrak{h}/W)^{\Omega}$ induces a $\mathbb{G}_m$-equivariant isomorphism $f_1:\mathfrak{h}_1/W_1 \rightarrow (\mathfrak{h}/W)^{\Omega}$. Furthermore, it is known that $\mathfrak{h}_1$ and $W_1$ are a Cartan subalgebra and a Weyl group of type $\Delta(\Gamma,\Gamma')$ respectively (cf. \cite{Carter72} Section 13.3). Let $\pi_1 : \mathfrak{h}_1 \rightarrow \mathfrak{h}_1/W_1$ be the natural projection. According to Subsection~\ref{sub: Brieskorngeneralised}, the singular configuration of the fibre of $\alpha^\Omega:X_{\Gamma,\Omega} \rightarrow (\mathfrak{h}/W)^\Omega$ above a point $\pi_1(h_1)$ can be determined by the hyperplanes arrangement of the $H_\beta$'s containing $h_1$, where the $\beta$'s are the roots of a root system of type $\Delta(\Gamma,\Gamma')$. However, as seen in Example~\ref{ConterexamplewithA5}, this singular configuration cannot be determined by starting from a preimage of $h_1$ in a Cartan subalgebra of a Lie algebra of type $\Delta(\Gamma)$. 

Subsequently in Section~\ref{sec:QuotientInhomogeneousDeformations}, the quotient map $\overline{\alpha^\Omega}:X_{\Gamma,\Omega}/\Omega \rightarrow (\mathfrak{h}/W)^\Omega$ was computed and it was determined to be a non-semiuniversal deformation of the simple singularity $\cc^2/\Gamma'$ of type $\Delta(\Gamma')$. Let $\overline{\mathbb{D}_\Omega}$ denote the discriminant of this morphism. It is natural to wonder if the same reasoning can be applied to this map, i.e. is there a subalgebra $\mathfrak{h}_2$ of the Cartan subalgebra of type $\Delta(\Gamma')$ and a subgroup $W_2$ of the Weyl group of type $\Delta(\Gamma')$ such that there exists an isomorphism $f_2:\mathfrak{h}_2/W_2 \rightarrow (\mathfrak{h}/W)^{\Omega}$? If so, it would be interesting to know whether one can use the diagram
\vspace{-\topsep}
\begin{center}
\begin{tikzpicture}[scale=0.9,  transform shape]
\node (1) at (3,2) {$X_{\Gamma,\Omega}/\Omega$};

\node (3) at (3,0) {$(\mathfrak{h}/W)^\Omega$};
\node (4) at (3,-0.4) {$\bigcup$};
\node (5) at (3,-0.8) {$\overline{\mathbb{D}_\Omega}$};

\node (7) at (0,0) {$\mathfrak{h}_2/W_2$};
\node (8) at (0,-0.4) {$\bigcup$};
\node (9) at (0,-0.8) {$f_2^{-1}(\overline{\mathbb{D}_\Omega})$};
\node (10) at (-2,0) {$\mathfrak{h}_2$};

\node (12) at (1.55,-0.2) {$\cong$};

\node(13) at (0.7,0) {};
\node(14) at (2.4,0) {};
\node(15) at (0.7,-0.8) {};
\node(16) at (2.4,-0.8) {};

\draw  [decoration={markings,mark=at position 1 with
    {\arrow[scale=1.2,>=stealth]{>}}},postaction={decorate}] (1)  -- node[right] {$\overline{\alpha^\Omega}$} (3);
\draw  [decoration={markings,mark=at position 1 with
    {\arrow[scale=1.2,>=stealth]{>}}},postaction={decorate}] (13)  -- node[above] {$f_2$} (14);
\draw  [|-,decoration={markings,mark=at position 1 with
    {\arrow[scale=1.2,>=stealth]{>}}},postaction={decorate}] (15)  --  (16);
\draw  [decoration={markings,mark=at position 1 with
    {\arrow[scale=1.2,>=stealth]{>}}},postaction={decorate}] (10)  -- node[above] {$\pi_2$} (7);
\end{tikzpicture}
\end{center}
\vspace{-\topsep}
to describe the singular configurations of the fibres of $\overline{\alpha^\Omega}$ using the root system of type $\Delta(\Gamma')$.

\section*{Acknowledgements}
I would like to thank my advisor Kenji Iohara whose guidance and advice enabled me to obtain the results contained in this paper. My gratitude also goes to Marion Jeannin who helped me put this paper into its present form.
\addcontentsline{toc}{section}{Acknowledgements}

\addcontentsline{toc}{section}{References}
\bibliographystyle{smfplain}
\bibliography{bibliography}

\end{document}